\documentclass{amsart}

\usepackage{amssymb}
\usepackage{amsfonts}
\usepackage{amsmath}
\usepackage{amsthm}
\usepackage{graphicx}
\usepackage{mathrsfs}
\usepackage{dsfont}
\usepackage{amscd}
\usepackage{multirow}
\usepackage[all]{xy}
\normalfont
\usepackage[T1]{fontenc}
\usepackage{calligra}
\usepackage{verbatim}
\usepackage[usenames,dvipsnames]{color}
\usepackage[colorlinks=true,linkcolor=Blue,citecolor=Violet]{hyperref}
\usepackage{enumerate}

\newcommand{\N}{{\mathds{N}}}
\newcommand{\Z}{{\mathds{Z}}}

\newcommand{\R}{{\mathds{R}}}
\newcommand{\C}{{\mathds{C}}}
\newcommand{\T}{{\mathds{T}}}

\newcommand{\D}{{\mathfrak{D}}}
\newcommand{\A}{{\mathfrak{A}}}
\newcommand{\B}{{\mathfrak{B}}}
\newcommand{\M}{{\mathfrak{M}}}

\newcommand{\bigslant}[2]{{\raisebox{.2em}{$#1$}\left/\raisebox{-.2em}{$#2$}\right.}}

\newcommand{\Lip}{{\mathsf{L}}}

\newcommand{\dist}{{\mathsf{dist}}}

\newcommand{\propinquity}[1]{{\mathsf{\Lambda}_{#1}}}

\newcommand{\modpropinquity}[1]{{\mathsf{\Lambda}^{\mathsf{mod}}_{#1}}}

\newcommand{\Kantorovich}[1]{{\mathsf{mk}_{#1}}}
\newcommand{\KantorovichMod}[1]{{\mathsf{k}_{#1}}}

\newcommand{\Haus}[1]{{\mathsf{Haus}_{#1}}}

\newcommand{\StateSpace}{{\mathscr{S}}}

\newcommand{\MongeKant}{{Mon\-ge-Kan\-to\-ro\-vich metric}}

\newcommand{\Lqcms}{{\JLL} quantum compact metric space}
\newcommand{\Qqcms}[1]{{$#1$}--\gQqcms}
\newcommand{\gQqcms}{quasi-Leibniz quantum compact metric space}

\newcommand{\gQVB}{metrized quantum vector bundle}
\newcommand{\QVB}[3]{$\left({#1},{#2},{#3}\right)$--\gQVB}

\newcommand{\unit}{1}

\newcommand{\sa}[1]{{\mathfrak{sa}\left({#1}\right)}}

\newcommand{\inner}[3]{{\left<{#1},{#2}\right>_{#3}}}
\newcommand{\JLL}{Lei\-bniz}

\newcommand{\bridge}[1]{#1} 
\newcommand{\dom}[1]{{\operatorname*{dom}\left({#1}\right)}}
\newcommand{\codom}[1]{{\operatorname*{codom}\left({#1}\right)}}
\newcommand{\diam}[2]{{\mathrm{diam}\left({#1},{#2}\right)}}

\newcommand{\bridgeset}[3]{{\text{\calligra Bridges}\,\left[{#1}\stackrel{#3}{\longrightarrow}{#2}\right]}}

\newcommand{\bridgereach}[1]{{\varrho\left(\bridge{#1}\right)}}
\newcommand{\bridgebasicreach}[1]{{\varrho_\flat\left({#1}\right)}}
\newcommand{\bridgemodularreach}[1]{{\varrho^\sharp\left({#1}\right)}}
\newcommand{\bridgeheight}[1]{{\varsigma\left(\bridge{#1}\right)}}

\newcommand{\bridgeimprint}[1]{{\varpi\left(\bridge{#1}\right)}}
\newcommand{\bridgelength}[1]{{\lambda\left(\bridge{#1}\right)}}
\newcommand{\bridgepivot}[1]{{\operatorname*{pivot}\left(\bridge{#1}\right)}}
\newcommand{\bridgeanchors}[1]{{\operatorname*{anchors}\left(\bridge{#1}\right)}}
\newcommand{\bridgecoanchors}[1]{{\operatorname*{coanchors}\left(\bridge{#1}\right)}}
\newcommand{\treklength}[1]{{\lambda\left({#1}\right)}}
\newcommand{\trekset}[3]{{\text{\calligra Treks}\left[{#2}\stackrel{#1}{\longrightarrow}{#3}\right]}}

\newcommand{\bridgenorm}[2]{{\mathsf{bn}_{ \bridge{#1}  }\left({#2}\right)}}
\newcommand{\decknorm}[2]{{\mathsf{dn}_{ \bridge{#1}  }\left({#2}\right)}}

\newcommand{\itineraries}[4]{{\text{\calligra Itineraries}\left({#2}\stackrel{#1}{\longrightarrow} {#3}  \middle\vert {#4} \right)}}

\newcommand{\Jordan}[2]{{{#1}\circ{#2}}} 
\newcommand{\Lie}[2]{{\left\{{#1},{#2}\right\}}} 
\newcommand{\targetsetbridge}[3]{{\mathfrak{t}_{\bridge{#1}}\left({#2}\middle\vert{#3}\right)}}
\newcommand{\targetsettrek}[3]{{\mathfrak{T}_{#1}\left({#2}\middle\vert{#3}\right)}}

\newcommand{\CDN}{{\mathsf{D}}}

\newcommand{\worknote}[1]{} 
\newcommand{\opnorm}[3]{{\left|\mkern-1.5mu\left|\mkern-1.5mu\left| {#1} \right|\mkern-1.5mu\right|\mkern-1.5mu\right|_{#3}^{#2}}}
\newcommand{\basespace}[1]{{\operatorname*{bqs}\left({#1}\right)}}

\newcommand{\alg}[1]{{\mathfrak{#1}}}
\newcommand{\alglip}[2]{{\alg{L}_{#1}\left({#2}\right)}}
\newcommand{\modlip}[2]{{\module{D}_{#1}\left({#2}\right)}}
\newcommand{\module}[1]{{\mathscr{#1}}}
\newcommand{\qt}[1]{{\mathcal{A}_{#1}}}
\newcommand{\HeisenbergGroup}{{\mathds{H}_3}}
\newcommand{\HeisenbergMod}[2]{{\module{H}_{#1}^{#2}}}

\theoremstyle{plain}
\newtheorem{theorem}{Theorem}[section]

\newtheorem{claim}[theorem]{Claim}

\newtheorem{corollary}[theorem]{Corollary}

\newtheorem{lemma}[theorem]{Lemma}
\newtheorem{proposition}[theorem]{Proposition}

\newtheorem{theorem-definition}[theorem]{Theorem-Definition}

\theoremstyle{definition}
\newtheorem{definition}[theorem]{Definition}

\newtheorem{example}[theorem]{Example}

\newtheorem{convention}[theorem]{Convention}
\newtheorem{hypothesis}[theorem]{Hypothesis}

\theoremstyle{remark}

\newtheorem{remark}[theorem]{Remark}

\newtheorem{notation}[theorem]{Notation}

\renewcommand{\geq}{\geqslant}
\renewcommand{\leq}{\leqslant}

\numberwithin{equation}{subsection}

\allowdisplaybreaks[4]

\begin{document}

\title{The Modular Gromov-Hausdorff Propinquity}
\author{Fr\'{e}d\'{e}ric Latr\'{e}moli\`{e}re}
\thanks{This work is part of the project supported by the grant H2020-MSCA-RISE-2015-691246-QUANTUM DYNAMICS}
\email{frederic@math.du.edu}
\urladdr{http://www.math.du.edu/\symbol{126}frederic}
\address{Department of Mathematics \\ University of Denver \\ Denver CO 80208}

\date{\today}

\begin{abstract}
Motivated by the quest for an analytic framework to study classes of C*-algebras and associated structures as geometric objects, we introduce a metric on Hilbert modules equipped with a generalized form of a differential structure, thus extending Gromov-Hausdorff convergence theory to vector bundles and quantum vector bundles --- not convergence as total space but indeed as quantum vector bundle. Our metric is new even in the classical picture, and creates a framework for the study of the moduli spaces of modules over C*-algebras from a metric perspective. We apply our construction, in particular, to the continuity of Heisenberg modules over quantum $2$-tori.



\end{abstract}
\maketitle

\tableofcontents



\section{Introduction}

Noncommutative metric geometry, as pioneered by Connes \cite{Connes80,Connes} and advanced by Rieffel \cite{Rieffel98a,Rieffel99,Rieffel00}, enables the construction of new topologies over classes of quantum metric spaces, namely, of noncommutative algebras seen as generalizations of the algebras of Lipschitz functions over metric spaces. These new topologies, in turn, permit the discussion of such problems as finite dimensional approximations of quantum spaces \cite{Rieffel01,Latremoliere05,Latremoliere13c,Latremoliere16, Latremoliere15d} as motivated by problems from mathematical physics, as well as quantitative analysis of quantum metric perturbations \cite{Latremoliere15b,Latremoliere15c}, continuity of families of C*-algebras in a topological sense \cite{Latremoliere05,li03, Latremoliere13c,Latremoliere15c,Latremoliere16c}, and topological properties of classes of quantum spaces --- in particular, generalization of the Gromov compactness theorem \cite{Latremoliere15}. The topologies defined on quantum metric spaces are induced by various noncommutative analogues of the Gromov-Hausdorff distance \cite{Gromov81,Gromov}, and open new avenues in the study of noncommutative geometry by means of topological and geometric methods applied to entire classes of C*-algebras. 

We propose in this paper to construct a topology on classes of metrized Hilbert modules over quantum metric spaces, addressing a long standing question about the behavior of modules under metric convergence \cite{Rieffel09,Rieffel10c,Rieffel11,Rieffel15}. As modules over C*-algebras are a core ingredient to much of the study of C*-algebras, from K-theory to KK-theory, from Morita equivalence to the noncommutative generalizations of vector bundles, our new metric opens an important new field  of investigation on the geometry of quantum spaces. A core motivation of our construction is to provide a framework for the analysis of approximations of modules in mathematical physics.

Our construction starts with our understanding of Gromov-Hausdorff convergence generalized to quantum compact metric spaces, as developed in \cite{Latremoliere13,Latremoliere13b,Latremoliere14} with the introduction of the Gromov-Hausdorff propinquity. The propinquity defines a metric on the class of {\gQqcms s} and induces the same topology as the Gromov-Hausdorff distance on classical compact metric spaces, yet its construction is functional in nature, and thus provides a new perspective on convergence of spaces. In this paper, We construct a far-reaching extension of the quantum Gromov-Hausdorff propinquity to Hilbert modules \cite{Rieffel74} over quantum metric spaces, equipped with a metric generalization of a connection, called a D-norm. We prove that our new metric, called the modular Gromov-Hausdorff propinquity, is indeed a distance on our class of {\gQVB s} up to a strong notion of isomorphism, which preserves the module structure, the inner product --- hence the $C^\ast$-Hilbert norm --- and the D-norm. We check that our new distance extends the topology of the quantum propinquity, since C*-algebras are canonically Hilbert modules over themselves. There are no analogues of the Gromov-Hausdorff distance on vector bundles even classically, hence the modular propinquity introduces new possibilities even in the classical picture. However, our first main application for our metric is given in \cite{Latremoliere17a} and concerns Heisenberg modules \cite{Connes80, Rieffel83} over quantum tori.

Our research program began with the observation by A. Connes in \cite{Connes89} that an noncommutative analogue of the {\MongeKant} could be defined on the state space of a C*-algebra by means of a spectral triple. Rieffel laid the foundation for the study of quantum metric spaces \cite{Rieffel98a}, recognizing that a quantum metric is encoded in a type of seminorm whose dual seminorm induces a noncommutative {\MongeKant} which, crucially, Rieffel requires to induce the weak* topology on the state space, just as its classical counter part. Such seminorms are called Lip-norms. This observation laid the groundwork for the construction of noncommutative analogues of the Gromov-Hausdorff distance --- a task fraught with unexpected challenges, and which gave rise to a succession of candidates. The first construction is also due to Rieffel who defined the quantum Gromov-Hausdorff distance \cite{Rieffel00}. 

We introduce the Gromov-Hausdorff propinquity in \cite{Latremoliere13,Latremoliere13b,Latremoliere14} as our answer to the challenge of devising an analogue of the Gromov-Hausdorff distance which was well-behaved with respect to the C*-algebraic structure: the propinquity only involves C*-algebras and encode the connection between the quantum topology --- encoded in the algebraic structure --- and the quantum metric --- encoded in a Lip-norm --- in some form of the Leibniz inequality; quantum compact metric spaces which satisfy this additional requirement are called {\gQqcms s} \cite{Latremoliere15}. This idea actually raises some serious difficulties when attempting to construct an analogue of the Gromov-Hausdorff distance, as seen with the early work of Kerr \cite{kerr02} or, later on, with the quantum proximity of Rieffel \cite{Rieffel10c}, where in each case, working exclusively with Leibniz Lip-norms, or even their generalizations, lead to analogues of the Gromov-Hausdorff distance which are not satisfying, as far as we know, the triangle inequality. 

We thus assume given a sequence of {\gQqcms s} converging for the quantum propinquity --- the quantum propinquity is a relatively strong version of the Gromov-Hausdorff propinquity, which is actually a family of metrics. There are many interesting such sequences already known: quantum tori \cite{Latremoliere13c}, fuzzy tori converging to quantum tori \cite{Latremoliere13c}, matrix algebras converging to the sphere \cite{Rieffel01,Rieffel10c}, certain sequences of AF-algebras \cite{Latremoliere15d}, finite dimensional {\gQqcms s} converging to any nuclear quasi-diagonal {\gQqcms} \cite{Latremoliere15}, conformal deformations of {\Lqcms s} constructed from spectral triples \cite{Latremoliere15b}, curved quantum tori \cite{Latremoliere15c}, noncommutative solenoids \cite{Latremoliere16}, to name but a few. Given such a sequence, we wish to make sense of the statement: each {\gQqcms s} in the sequence carries a module, presumably equipped with additional metric data, such that the resulting sequence of modules converge in some sense to a module over the limit {\gQqcms s}.

Our approach is best presented by first discussing it informally in the context of manifolds and their vector bundles, where our work is already new. A vector bundle $V$ over a compact, connected Riemannian manifold $M$, may be equipped with a metric --- which in this context, means a smooth section of the bundle of inner products over each fiber of $V$. In other words, we pick an inner product on the $C(M)$-module $\Gamma$ of continuous sections of $V$, where $C(M)$ is the C*-algebra of $\C$-valued continuous functions over $M$. This is a particular example of a Hilbert module over a C*-algebra, and thus we will work in this paper with left Hilbert modules. Yet, there is one more essential tool of differential geometry when working with modules: namely, the notion of a connection. Indeed, given a metric on a bundle, one may always find a so-called metric connection, and under stronger assumption, this connection is in fact unique. In other words, to a metric corresponds a natural notion of parallel transport. 

In noncommutative geometry, there still exists metric connections, appropriately defined, on left Hilbert modules, under rather generous conditions. The matter of uniqueness issue is less clear. In any event, we adopt the perspective that the metric information needed to work with vector bundles include not only an inner product on its module of sections, but also a choice of a connection, or rather an additional norm on the space of smooth functions which encode some aspects of the connection which are of use to define our metric. We do not need the full strength of a connection in this work, and we will address the issue of convergence for differential structures in a forthcoming paper.

For our purpose, therefore, the metric data which we will consider to define a {\gQVB} includes a Hilbert module equipped with an additional, densely defined norm called a D-norm with a natural topological condition inspired by the commutative picture above. There is a clear relation between our notion of a {\gQVB} and the notion of a {\gQqcms} which we take as a sign that our approach is sensible. Moreover, a {\gQVB} will of course be defined over a particular {\gQqcms}, its base space. Just as is the case with proving that certain semi-norms are Lip-norms, establishing that a given norm is a D-norm may be challenging. 

The modular Gromov-Hausdorff propinquity is thus defined on the class of all {\gQVB s}. Our idea begins with the concept of a bridge between {\gQVB s}, as an extension of the idea of bridges used in the construction of the quantum propinquity. A bridge between {\gQqcms s} is a particular type of embedding of the C*-algebras underlying the quantum metric spaces, which allows to quantify how far two quantum metric spaces may be. The idea is that, if we start from a convergent sequence of spaces for the quantum propinquity, then we already have access to bridges between spaces. We then augment these bridges by adding some families of elements from modules over these spaces, resulting what we call modular bridges. We associate a quantity to these bridges, in we then prove that our idea indeed leads to a metric on {\gQVB s}. Naturally, there are many new challenges raised by working with modules.

In summary, we present in this paper the class of {\gQVB s}, on which we then define a metric akin to a Gromov-Hausdorff distance for (noncommutative) vector bundles. We prove that our metric, the modular Gromov-Hausdorff propinquity, is indeed a metric up to full quantum isometry of {\gQVB s} --- i.e. an appropriate notion of morphism between left Hilbert modules over possibly different C*-algebras which also preserves all the metric data. We then prove that we apply our metric to the subclass of the {\gQVB s} canonically constructed from {\gQqcms s} --- extending the observation that any C*-algebra is a left Hilbert module over itself --- we recover the topology of the quantum Gromov-Hausdorff topology. We also show that, reassuringly, the modular propinquity between free modules over {\gQqcms s} behave as expected --- close base spaces in the quantum propinquity and same rank of free modules imply close in the modular propinquity. We discuss a sufficient condition for the direct sum of {\gQVB s} to be continuous on certain classes called iso-pivotal.

We include as a last section the main result of \cite{Latremoliere17a}, which is our main application of the metric presented in this paper: a result concerning the continuity of the family of Heisenberg modules over quantum $2$-tori. This result involves much technicalities --- including proving that Heisenberg modules, when equipped with their natural connection \cite{Connes80}, actually fit within our framework. Nonetheless, the last section contains the principal conclusion of \cite{Latremoliere17a}, as a non-trivial application of the modular Gromov-Hausdorff distance.

\section{Quantum Compact Metric Spaces}

A quantum compact metric space is a noncommutative generalization of the algebras of Lipschitz functions over a metric space. Our work on the Gromov-Hausdorff propinquity \cite{Latremoliere13,Latremoliere13b,Latremoliere14,Latremoliere15,Latremoliere15c} emphasizes the role of a relation between the generalized Lipschitz seminorms and the multiplicative structure of the underlying algebra, though this relation can be quite general. We will thus work in the category of {\gQqcms s}, defined as follows.

\begin{notation}
The norm of a normed vector space $E$ will be denoted by $\|\cdot\|_E$ by default.
\end{notation}

\begin{notation}
Let $\A$ be a unital C*-algebra. The space of self-adjoint elements of $\A$ is denoted by $\sa{\A}$, while the state space of $\A$ is denoted by $\StateSpace(\A)$. The unit of $\A$ is denoted by $\unit_\A$.
\end{notation}

\begin{definition}[\cite{Latremoliere15}]\label{admissible-fn-def}
A function $F : [0,\infty)^4 \rightarrow [0,\infty)$ is \emph{admissible} when for all:
\begin{equation*}
(x_1,x_2,x_3,x_4), (y_1,y_2,y_3,y_4) \in [0,\infty)^4
\end{equation*}
such that $x_j \leq y_j$ for all $j \in \{1,2,3,4\}$, we have:
\begin{equation*}
F(x_1,x_2,x_3,x_4)\leq F(y_1,y_2,y_3,y_4)
\end{equation*}
and $x_1x_3 + x_2x_4 \leq F(x_1,x_2,x_3,x_4)$.
\end{definition}

\begin{definition}[\cite{Rieffel98a,Rieffel99,Latremoliere13,Latremoliere15}]\label{Qqcms-def}
An \emph{\Qqcms{F}} $(\A,\Lip)$, for some admissible function $F$, is a unital C*-algebra $\A$ and a seminorm $\Lip$ defined on a dense Jordan-Lie subalgebra $\dom{\Lip}$ of $\sa{\A}$, such that:
\begin{enumerate}
\item $\{a\in\dom{\Lip} : \Lip(a) = 0 \} = \R\unit_\A$,
\item the \emph{\MongeKant} $\Kantorovich{\Lip}$, defined on the state space $\StateSpace(\A)$ by setting for all $\varphi, \psi \in \StateSpace(\A)$:
\begin{equation*}
\Kantorovich{\Lip}(\varphi,\psi) = \sup\left\{ |\varphi(a) - \psi(a)| : a\in\sa{\A}, \Lip(a)\leq 1 \right\}\text{,}
\end{equation*}
metrizes the weak* topology on $\StateSpace(\A)$,
\item $\Lip$ is lower semi-continuous, i.e. $\{a\in\sa{\A} : \Lip(a)\leq 1\}$ is norm closed,
\item for all $a,b \in \dom{\Lip}$, we have:
\begin{equation*}
\max\left\{ \Lip\left(\Jordan{a}{b}\right), \Lip\left(\Lie{a}{b}\right) \right\} \leq F\left(\|a\|_\A,\|b\|_\A,\Lip(a),\Lip(b)\right) \text{.}
\end{equation*}
\end{enumerate}

The seminorm $\Lip$ of a {\Qqcms{F}} $(\A,\Lip)$ is called an \emph{L-seminorm} of type $F$. If $F : x,y,l_x,l_y \mapsto x l_y + y l_x$, then $(\A,\Lip)$ is simply called an {\Lqcms} and $\Lip$ is said to be of Leibniz type.
\end{definition}

We adopt the following convention in our exposition to keep our notations simple:
\begin{convention}
If $\Lip$ is a seminorm defined on a subspace $\dom{\Lip}$ of a vector space $E$ then, for all $x\in E\setminus\dom{\Lip}$, we set $\Lip(x) = \infty$. Thus $\dom{\Lip} = \{ x\in E : \Lip(x) < \infty \}$ with this convention. With the additional convention that $0\cdot\infty = 0$ and $\infty + r = r + \infty = \infty$ for all $r\in [0,\infty]$ when working with seminorms, we can see $\Lip$ as a seminorm on $E$ taking the value $\infty$.
\end{convention}

The classical prototype of a {\Lqcms} is given by the ordered pair $(C(X),\mathrm{Lip})$ of the C*-algebra $C(X)$ of $\C$-valued continuous functions over a compact metric space $(X,\mathrm{d})$ and the Lipschitz seminorm induced on $\sa{C(X)}$ --- the Banach algebra of $\R$-continuous functions --- by the metric $\mathrm{d}$. In this case, the metric $\Kantorovich{\Lip}$ is indeed the {\MongeKant} on the space of regular Borel probability measures over $X$, a fundamental object introduced by Kantorovich \cite{Kantorovich40} in the study of Monge's transportation problem. The form of the {\MongeKant} which we generalize in Definition (\ref{Qqcms-def}) was discovered by Kantorovich and Rubinstein \cite{Kantorovich58}.

Important examples of {\gQqcms s} include the quantum tori \cite{Rieffel98a,Rieffel02}, Connes-Landi sphere \cite{li05}, full C*-algebras of Hyperbolic groups \cite{Ozawa05}, AF-algebras with a faithful tracial state \cite{Latremoliere15d}, curved quantum tori \cite{Latremoliere15c}, conformal perturbations of quantum metric spaces obtained from Dirac operators \cite{Latremoliere15b}, C*-algebras of nilpotent groups \cite{Rieffel15b}, noncommutative solenoids \cite{Latremoliere16}, among many other. Moreover, finite dimensional C*-algebras can be endowed with many quantum metric structures which play an important role when approximating C*-algebras of continuous functions over coadjoint orbits of semisimple Lie groups \cite{Rieffel01}, quantum tori \cite{Latremoliere05,Latremoliere13c}, AF-algebras \cite{Latremoliere15d}, and arbitrary nuclear quasi-diagonal {\gQqcms s} \cite{Latremoliere15}. 

The first occurrence of a noncommutative version of the {\MongeKant} is due to Connes in \cite{Connes89}, where it was observed that a spectral triple give rise to a metric on the state space of a C*-algebra. Rieffel initiated the study of compact quantum metric spaces in \cite{Rieffel98a} by requiring that the {\MongeKant} in noncommutative geometry should metrize the weak* topology on the state space, and can be built without appealing to the theory of spectral triple, but rather using a generalized Lipschitz seminorm. As a matter of terminology, a seminorm $\Lip$ satisfying properties (1) and (2) is called a \emph{Lip-norm}. We choose our new terminology to avoid writing the rather long expression ``quasi-Leibniz lower semi-continuous Lip-norm'' too often. A pair $(\A,\Lip)$ where $\Lip$ is a Lip-norm is called a compact quantum metric space. 

Quantum locally compact metric spaces were introduced in \cite{Latremoliere12b}, building on our work in \cite{Latremoliere05b}, and provide a far-reaching generalization of Definition (\ref{Qqcms-def}).

Definition (\ref{Qqcms-def}) evolved with the role of the algebraic structure of a compact quantum metric spaces. In \cite{Rieffel98a}, Rieffel's original notion of Lip-norm was defined over normed vector spaces (and the notion of state was replaced with a more general notion). In \cite{Rieffel99}, the focus was on order-unit spaces, and this was the setting for the construction of the quantum Gromov-Hausdorff distance \cite{Rieffel00}, and the first examples of continuity for that metric \cite{Rieffel01,Latremoliere05,li05}. As research in noncommutative metric geometry became focused on the relationship between convergence for analogues of the Gromov-Hausdorff distance and C*-algebraic structures --- in particular modules \cite{Rieffel09,Rieffel10c} --- it became apparent that Lip-norms should be connected to the underlying C*-algebraic structure. We proposed Definition (\ref{Qqcms-def}) by adapting the idea of $F$-Leibniz seminorms of Kerr's \cite{kerr02}, with two differences. 

L-seminorms are defined on a dense *-subspace of the self-adjoint part of C*-algebras, and in general $\sa{\A}$ is not a *-subalgebra of $\A$ for non-Abelian C*-algebras. It is a Jordan-Lie algebra, and thus we use the Jordan and Lie product in the definition of the quasi-Leibniz property. Our insistence on working with L-seminorms defined only on self-adjoint elements will be justified when discussing quantum isometries later on. Second, we require that the quasi-Leibniz property be no sharper than the Leibniz property, which actually ensures that for any given choice of an admissible function $F$, the quantum propinquity can be restricted to the class of {\Qqcms{F}s} and never involve any space outside of this class. We refer to \cite{Latremoliere13,Latremoliere13b,Latremoliere14,Latremoliere15b,Latremoliere15c} for details.

The class of quantum compact metric spaces form a category when morphisms are defined using a natural Lipschitz condition. In fact, there are at least three ideas one may consider when defining a notion of a Lipschitz morphism between compact quantum spaces, and these notions will not agree in general. However, as we impose that L-seminorms are lower-semicontinuous, all three ideas agree for {\gQqcms s}.

We choose the following definition for morphisms of {\gQqcms s}. 

\begin{definition}\label{Lipschitz-mor-def}
A \emph{$k$-Lipschitz morphism} $\pi : (\A,\Lip_\A) \rightarrow (\B,\Lip_\B)$ between two {\gQqcms s} $(\A,\Lip_\A)$ and $(\B,\Lip_\B)$ is a *-morphism $\pi : \A\rightarrow \B$ such that for all $a\in\dom{\Lip_\A}$:
\begin{equation*}
\Lip_\B \circ \pi (a) \leq k \Lip_\A(a) \text{.}
\end{equation*}
\end{definition}

We then show that other natural ideas for morphisms of {\gQqcms s} agree with Definition (\ref{Lipschitz-mor-def}).

\begin{theorem}[{\cite{Rieffel99, Latremoliere16b}}]\label{Lipschitz-mor-thm}
Let $(\A,\Lip_\A)$ and $(\B,\Lip_\B)$ be two {\gQqcms s} and $\pi : \A\rightarrow\B$ be a *-morphism. The following assertions are equivalent for any $k \geq 0$:
\begin{enumerate}
\item $\pi$ is a $k$-Lipschitz morphism,
\item $\pi^\ast : \varphi\in\StateSpace(\B) \mapsto \varphi\circ\pi \in \StateSpace(\A)$ is a $k$-Lipschitz map from $(\StateSpace(\B),\Kantorovich{\Lip_\B})$ to $(\StateSpace(\A),\Kantorovich{\Lip_\A})$,
\item $\pi(\dom{\Lip_\A}) \subseteq \dom{\Lip_\B}$.
\end{enumerate}
\end{theorem}

Assertion (2) in Theorem (\ref{Lipschitz-mor-thm}) was the initial definition of a Lipschitz morphism in \cite{Rieffel99,Rieffel00}. The equivalence between Assertion (1) and Assertion (3) was the subject of \cite{Latremoliere16b} while Rieffel proved the equivalence between Assertion (1) and (2) in \cite{Rieffel99}, where the importance of lower semicontinuity for Lip-norms was discovered.

It is straightforward to check that, taking for objects our {\gQqcms s} and for morphisms our Lipschitz morphisms give rise to a category.

The stronger notion of isometry between {\gQqcms s} must be well-understood in our context, since the Gromov-Hausdorff propinquity is a metric up to isometry, properly defined. The original notion of isometry for compact quantum metric space \cite{Rieffel00} did not involve *-morphisms, since two Rieffel's distance could be null between compact quantum metric spaces which were not *-isomorphic. 

Rieffel's insight into the proper notion of an isometric embedding rests on Bla\-schke's theorem \cite[Theorem 7.3.8]{burago01}, which states that a \emph{real valued} $k$-Lipschitz function on some nonempty subset of a metric space can be extended to a $k$-Lipschitz function on the whole space. For our purpose, the main consequence of Blaschke's theorem is that, if $\pi : (X,\mathsf{d})\hookrightarrow (Z,\mathsf{m})$ is an injection between two compact metric spaces, then $\pi$ is an isometry if and only if the Lipschitz seminorm on $C(X)$ induced by $\mathrm{d}$ is the quotient seminorm of the Lipschitz seminorm on $C(Z)$ induced by $\mathrm{m}$. This is the origin of the definition of a quantum isometry.

Blaschke's theorem is not valid as stated for $\C$-valued Lipschitz functions: in general, the best statement for $\C$-value Lipschitz functions is that a $k$-Lipschitz function over a subset can be extended to the whole space as a $\frac{4 k}{\pi}$-Lipschitz function \cite{Rieffel09}. It means that the relationship between Lipschitz seminorms provided, on the $\R$-valued functions, by isometries, does not hold for $\C$-valued function, rendering the generalization of these ideas to the noncommutative realm less obvious, and thus justifying in large part the choice to work with L-seminorms defined only for self-adjoint elements in general.

The construction of the propinquity was in large part motivated by ensuring that *-isomorphism is necessary for distance zero, and thus we arrive at the notion of quantum isometries which we have used since our work in \cite{Latremoliere13}:

\begin{definition}[{\cite{Rieffel00,Latremoliere13}}]\label{quantum-isometry-def}
Let $(\A,\Lip_\A)$ and $(\B,\Lip_\B)$ be two {\gQqcms s}. A \emph{quantum isometry} $\pi : (\A,\Lip_\A) \twoheadrightarrow (\B,\Lip_\B)$ is a *-epimorphism such that for all $b\in\B$ we have:
\begin{equation*}
\Lip_\B(b) = \inf\{ \Lip_\A(a) : \pi(a) = b\} \text{.}
\end{equation*}
A \emph{full quantum isometry} $\pi$ is a quantum isometry and a *-isomorphism such that $\pi^{-1}$ is also a quantum isometry.
\end{definition}

If $\pi : (\A,\Lip_\A) \twoheadrightarrow (\B,\Lip_\B)$ is a quantum isometry, then in particular $\Lip_\B\circ\pi(a) \leq \Lip_\A(a)$ for all $a\in\sa{\A}$ and thus $\pi$ is a $1$-Lipschitz morphism.

If $\pi$ is a full quantum isometry and $a\in\sa{\A}$ then $\Lip_\B\circ\pi(a) = \Lip_\A(a)$, since:
\begin{equation*}
\Lip_\B(\pi(a)) \leq \Lip_\A(a) = \Lip_\A(\pi^{-1}(\pi(a))) \leq \Lip_\B(\pi(a))\text{.}
\end{equation*}

One may therefore define a subcategory of {\gQqcms s} whose morphisms are quantum isometries, as quantum isometries compose to quantum isometries by \cite[Proposition 3.7]{Rieffel00}. In this category, full quantum isometries are the isomorphisms. The Gromov-Hausdorff propinquity is null between two {\gQqcms s} if and only if they are fully quantum isometric \cite{Latremoliere13,Latremoliere13b,Latremoliere15}.

We now turn to the following question: what metric structure may we equip modules over {\gQqcms s}, so that we then might develop a generalized notion of convergence for such metrized modules?

\section{D-norms}

Our work in this article is concerned with the construction of a metric on modules over {\gQqcms s}, appropriately defined. For the current research, a module over a {\gQqcms} is a left Hilbert $\A$-module over the underlying C*-algebra, equipped with an additional norm defined on some dense subspace (which is not in general a submodule) satisfying a particular topological requirement, and with various basic inequalities connecting all the ingredients of such a structure. These inequalities generalize the Leibniz inequality for L-seminorms.

As a matter of fixing our notations, we recall the definition of left Hilbert modules:
\begin{definition}[{\cite{Rieffel74}}]
A \emph{left pre-Hilbert module} $(\module{M},\inner{\cdot}{\cdot}{\module{M}})$ over a C*-algebra $\A$ is a left module $\module{M}$ over $\A$ equipped with a sesquilinear map $\inner{\cdot}{\cdot}{\module{M}} : \module{M}\times\module{M} \rightarrow \A$ such that for all $\omega,\eta\in\module{M}$ and $a\in\A$ we have:
\begin{enumerate}
\item $\inner{a\omega}{\eta}{\module{M}} = a\inner{\omega}{\eta}{\module{M}}$,
\item $\left(\inner{\omega}{\eta}{\module{M}}\right)^\ast = \inner{\eta}{\omega}{\module{M}}$,
\item $\inner{\omega}{\omega}{\module{M}} \geq 0$.
\item $\inner{\omega}{\omega}{\module{M}} = 0$ if and only if $\omega = 0$.
\end{enumerate}
\end{definition}

Let $(\module{M},\inner{\cdot}{\cdot}{\module{M}})$ be a left pre-Hilbert module over a C*-algebra $\A$. We note that Conditions (1) and (2) together prove that $\inner{\cdot}{\cdot}{\module{M}}$ possesses a modular form of sesquilinearity, i.e. $\inner{\omega}{a \eta}{\module{M}} = \inner{\omega}{\eta}{\module{M}} a^\ast$ for all $a\in\A$ and $\omega,\eta\in\module{M}$.

A version of the Cauchy-Schwarz inequality is valid for left pre-Hilbert modules, so that for all $\omega,\eta \in \module{M}$ we have:
\begin{equation*}
\inner{\omega}{\eta}{\module{M}}\inner{\eta}{\omega}{\module{M}} \leq \left\|\inner{\omega}{\omega}{\module{M}}\right\|_\A \inner{\eta}{\eta}{\module{M}}
\end{equation*}
and thus, together with the rest of the properties of the inner product, we may define a module norm on $\module{M}$ from the inner product $\inner{\cdot}{\cdot}{\module{M}}$:

\begin{proposition}[\cite{Rieffel74}]\label{Hilbert-norm-prop}
If $(\module{M},\inner{\cdot}{\cdot}{\module{M}})$ is a left pre-Hilbert module over a C*-algebra $\A$, and if, for all $\omega\in\module{M}$, we set:
\begin{equation*}
\|\omega\|_{\module{M}} = \sqrt{\left\|\inner{\omega}{\omega}{\module{M}}\right\|_\A}
\end{equation*}
then $\|\cdot\|_{\module{M}}$ is a module norm on $\module{M}$. i.e. a norm such that for all $a\in\A$ and $\omega\in\module{M}$, the following holds:
\begin{equation*}
\|a\omega\|_{\module{M}} \leq \|a\|_\A \|\omega\|_{\module{M}} \text{.}
\end{equation*}

Moreover, for any $\omega,\eta\in\module{M}$, we also have:
\begin{equation*}
\left| \inner{\omega}{\eta}{\module{M}} \right| \leq \|\omega\|_{\module{M}} \|\eta\|_{\module{M}} \text{.}
\end{equation*}
\end{proposition}

\begin{notation}
For a left pre-Hilbert module $(\module{M},\inner{\cdot}{\cdot}{\module{M}})$, we adopt the convention that $\|\cdot\|_{\module{M}}$ always refer to the norm defined in Proposition (\ref{Hilbert-norm-prop}) and we call this norm the \emph{Hilbert morm} of $(\module{M},\inner{\cdot}{\cdot}{\module{M}})$.
\end{notation}

We thus may require completeness of a left pre-Hilbert module for its $C^\ast$-Hilbert norm, leading to the following definition. 
\begin{definition}[{\cite{Rieffel74}}]
A \emph{left Hilbert module} $(\module{M},\inner{\cdot}{\cdot}{\module{M}})$ over a C*-algebra $\A$ is a left pre-Hilbert module over $\A$ which is complete for its $C^\ast$-Hilbert norm $\|\cdot\|_{\module{M}}$.
\end{definition}

We define our notion of a morphism between left Hilbert modules. Our morphisms can be defined between left Hilbert modules over different base algebras, and this concept will be folded in our notion of a morphism for a {\gQVB}.

\begin{definition}\label{module-morphism-def}
Let $(\module{M},\inner{\cdot}{\cdot}{\module{M}})$ and $(\module{N},\inner{\cdot}{\cdot}{\module{N}})$ be two left Hilbert modules over, respectively, two C*-algebras $\A$ and $\B$. A \emph{module morphism} $(\Theta,\theta)$ is given by a a *-morphism $\theta : \A \rightarrow \B$, and a $\C$-linear map $\Theta : \module{M} \rightarrow \module{N}$, such that for all $a\in\A$ and $\omega,\eta\in\module{M}$, we have:
\begin{enumerate}
\item $\Theta(a\omega)=\theta(a)\Theta(\omega)$,
\item $\inner{\Theta(\omega)}{\Theta(\eta)}{\module{N}} = \inner{\omega}{\eta}{\module{M}}$. 
\end{enumerate}
The module morphism $(\Theta,\theta)$ is \emph{unital} when $\theta$ is a unital *-morphism.
\end{definition}
We note that if $(\Theta,\theta)$ is a module morphism, $\Theta$ is continuous of norm $1$ by definition.

In continuing with the tradition in noncommutative geometry to name structures after their commutative analogues, we shall thus define a {\gQVB} as follows. We first extend the notion of an admissible function to a triple of functions, as we shall have three versions of the Leibniz inequality in our definition of {\gQVB s}.

\begin{definition}\label{admissible-triple-def}
A triple $(F,G,H)$ is \emph{admissible} when:
\begin{enumerate}
\item $F : [0,\infty)^4 \rightarrow [0,\infty)$ is admissible,
\item $G : [0,\infty)^3 \rightarrow [0,\infty)$ satisfies $G(x,y,z) \leq G(x',y',z')$ if $x,y,z,x',y',z'\in [0,\infty)$ and $x\leq x', y\leq y', z\leq z'$, while:
\begin{equation*}
(x + y)z \leq G(x,y,z) \text{.}
\end{equation*}
\item $H : [0,\infty)^2 \rightarrow [0,\infty)$ satisfies $H(x,y) \leq H(x',y')$ if $x,y,x',y'\in [0,\infty)$ and $x\leq x'$ and $y\leq y'$ while $2 x y \leq H(x,y)$.
\end{enumerate}
\end{definition}

The structure of a {\gQVB} is thus given by the following definition.

\begin{notation}
Let $a\in \A$ where $\A$ is a C*-algebra. We denote $\frac{a+a^\ast}{2}$ by $\Re a$ and $\frac{a-a^\ast}{2i}$ by $\Im a$. Note that $\Re a, \Im a \in \sa{\A}$.
\end{notation}

\begin{definition}\label{qvb-def}
A \emph{\QVB{F}{G}{H}} $(\module{M},\inner{\cdot}{\cdot}{\module{M}},\CDN_{\module{M}},\A,\Lip_\A)$, for some admissible triplet $(F,G,H)$, is given by a {\Qqcms{F}} $(\A,\Lip_\A)$, as well as a left Hilbert module $(\module{M},\inner{\cdot}{\cdot}{\module{M}})$ over $\A$ and a norm $\CDN_{\module{M}}$ defined on a dense $\C$-subspace $\dom{\CDN_{\module{M}}}$ of $\module{M}$ such that:
\begin{enumerate}
\item we have $\|\cdot\|_{\module{M}} \leq \CDN_{\module{M}}$,
\item the set:
\begin{equation*}
\left\{ \omega \in \module{M} : \CDN_{\module{M}}(\omega) \leq 1 \right\}
\end{equation*}
is compact for $\|\cdot\|_{\module{M}}$.
\item for all $a \in \sa{\A}$ and for all $\omega \in \module{M}$, we have:
\begin{equation*}
\CDN_{\module{M}}\left(a\omega\right) \leq G\left(\|a\|_\A,\Lip_\A(a), \CDN_{\module{M}}(\omega)\right)\text{,}
\end{equation*}
which we call the \emph{inner quasi-Leibniz inequality} for $\CDN_{\module{M}}$,
\item for all $\omega,\eta \in \module{M}$, we have:
\begin{equation*}
\max\left\{\Lip_\A\left(\Re\inner{\omega}{\eta}{\module{M}}\right),\Lip_\A\left(\Im\inner{\omega}{\eta}{\module{M}}\right)\right\} \leq H\left(\CDN_{\module{M}}(\omega), \CDN_{\module{M}}(\eta) \right)\text{,}
\end{equation*}
which we call the \emph{modular quasi-Leibniz inequality} for $\CDN_{\module{M}}$.
\end{enumerate}
The {\Qqcms{F}} $(\A,\Lip_\A)$ is called the \emph{base quantum space} of $\Omega$ and is denoted by $\basespace{\Omega}$. The norm $\CDN_{\module{M}}$ will be called the \emph{D-norm} of $\Omega$.
\end{definition}

If $\Omega$ is a {\QVB{F}{G}{H}} then its D-norm is said to be of type $(G,H)$. In general, we say that $\Omega$ is a {\gQVB} when it is a {\QVB{F}{G}{H}} for some admissible triple $(F,G,H)$. We make a simple observation which also applies, in a simpler form, to our work in \cite{Latremoliere15} as we shall discuss after the proof of Theorem (\ref{main-thm}).
\begin{notation}
Let $(F_1,G_1,H_1)$ and $(F_2,G_2,H_2)$ be two admissible triple. If $F = \max\{F_1,F_2\}$, $G = \max\{G_1,G_2\}$ and $H = \max\{H_1,H_2\}$, then $(F,G,H)$ is also an admissible triple, denoted by $(F_1,G_1,H_1)\vee(F_2,G_2,H_2)$. In particular, $F$ is admissible. Therefore, given any pair of {\gQqcms s} or {\gQVB s}, one may always assume that they share the same quasi-Leibniz properties.
\end{notation}

The classical picture which inspired our Definition (\ref{qvb-def}) of a {\gQVB} is provided by Riemannian geometry and locally trivial complex vector bundles. Our idea is that the choice of a metric connection for a hermitian metric on a complex vector bundle is in fact a part of the metric data of the associated module --- and gives rise to a prototypical D-norm.

\begin{example}\label{Riemannian-ex}
Let $M$ be a compact connected differentiable manifold of dimension $n$. As a matter of convention, we assume in this example that vector bundle is meant for locally trivial vector bundle, and that all our vector bundles have complex vector spaces as fibers; in particular, the tangent and cotangent bundles are complexified (by taking their tensor product with the trivial bundle $M\times \C$).

A natural derivation, namely the exterior differential, acts on the dense *-sub\-algebra $C^1(M)$ of $C^1$, $\C$-valued functions over $M$, inside the C*-algebra $C(M)$ of $\C$-valued continuous functions over $M$. This derivation is valued in the $C(M)$-$C(M)$-bimodule $\Omega_1$ of continuous sections of the cotangent bundle $T_\C^\ast M$ of $M$. We recall that for all pair $f ,g$ of $C^1$ functions on $M$, we have $d(f g) = f \wedge dg + df \wedge g = f dg + df \cdot g$.

We are interested in metric structures, and thus we naturally endow $M$ with some Riemannian metric $g$. Formally, a metric $g^V$ on a vector bundle $V$ is a smooth section of the vector bundle over $M$ of sesquilinear functionals over each fiber of $V$ such that for all $x\in\M$, the map $g^V_x$ over the fiber $V_x$ of $V$ at $x$ is in fact a hermitian inner product. In particular, a Riemannian metric is given as a metric over the cotangent bundle $T_\C^\ast M$ (or equivalently over the tangent bundle $T_\C M$ of $M$). 

Now, if $\Gamma V$ is the $C(M)$-left module of continuous sections of a vector bundle $V$ over $M$, then setting, for all $\omega,\eta\in \Gamma V$:
\begin{equation*}
x\in M \mapsto g^V_x(\omega_x,\eta_x)
\end{equation*}
defines a $C(X)$-valued inner product on $\Gamma V$, which we still denote by slight abuse of notation by $g^V$. Thus, $(\Gamma,g^V)$ is a left Hilbert module with norm, for all $\omega\in \Gamma V$:
\begin{equation*}
\|\omega\|_{\Gamma V} = \sup_{x\in M} \sqrt{g^V_x(\omega_x,\omega_x)} \text{.}
\end{equation*}
Notably, $\|\cdot\|_{\Gamma_V}$ is a module norm, i.e. for all $f\in C(X)$ and $\omega\in \Gamma V$ we can trivially check that $\|f \omega\|_{\Gamma_V} \leq\|f\|_{C(X)} \|\omega\|_{\Gamma_V}$.

We note that we will simply write $g^V_x(\omega,\eta)$ for $g^V_x(\omega_x,\eta_x)$ in the rest of this example, whenever $x\in M$, and $\omega$ and $\eta$ are in $\Gamma V$.

Let us focus for a moment on the case where $V$ is the cotangent bundle $T_\C^\ast M$ of $M$, and $g$ some Riemannian metric for $M$. We note that the right action of $C(M)$ on $\Omega_1$ is by so-called adjoinable operators, and in fact $(\Omega_1,g)$ is a Hilbert C*-bimodule over $C(M)$. Consequently, if we define:
\begin{equation*}
\Lip : f\in C^1(M) \mapsto \|d f\|_{\Omega_1}
\end{equation*}
then $\Lip$ is a seminorm defined on a dense subalgebra of $C(M)$, taking the value $0$ exactly on the constant functions over $M$ since $M$ is connected, and satisfying the Leibniz inequality:
\begin{equation*}
\Lip(f g) \leq \|f\|_{C(X)} \Lip(g) + \Lip(f) \|g\|_{C(X)}\text{.}
\end{equation*}
Thanks to the Arz{\'e}la-Ascoli theorem and since $M$ is compact, we note that:
\begin{equation*}
\mathfrak{BL}_1 = \left\{ f\in C(M) : \Lip(f) \leq 1, \|f\|_{C(X)}\leq 1 \right\}
\end{equation*}
is totally bounded in $C(X)$. Moreover, the closure of $\mathfrak{BL}_1$ consists of the $1$-Lipschitz functions with respect to the Riemannian path distance induced by $g$ on $M$, and the Lipschitz seminorm for this distance is the Minkowsky gauge functional of the closure of $\mathfrak{BL}_1$. As it agrees with $\Lip$ on $C^1(M)$, we denote the Lipschitz seminorm simply by $\Lip$ as well. Its closed unit ball is  now compact (as it is closed since $\Lip$ is lower semicontinuous), and it is trivial to check that is satisfies the Leibniz inequality. Thus $(C(X),\Lip)$ is a {\Lqcms}.

Our purpose is to introduce data on modules which will allow us to define a Gromov-Hausdorff distance between them, and thus we now return to our discussion of the metric structure of a generic vector bundle $V$ over $M$. Given a metric $g^V$ on $V$, a very important fact of Riemannian geometry is the existence of a metric connection, i.e. a connection $\nabla$ on $V$ with the property that for all tangent vector fields $X$ of $M$, and for all $\omega,\eta\in \Gamma V$, we have:
\begin{equation*}
d_X g^V(\omega,\eta) = g^V (\nabla_X \omega,\eta) + g^V(\omega,\nabla_X \eta) \text{.}
\end{equation*}
If we require the connection to be torsion free when $V = T_\C^\ast M$, then the connection $\nabla$ is unique and called the L{\'e}vi-Civita connection; we will however work on general vector bundles and not require any condition on the torsion of our metric connections. Instead, we look at the connection as part of the metric information of our vector bundle $V$.

Thus, let us fix a complex bundle $V$ over $M$ with a metric $g^V$ and let $\nabla$ be a $g^V$-metric connection on $V$. Let $\Gamma V$ be the $C(M)$-module of continuous sections of $V$ over $M$ and $\Gamma^1 V$ the $C^1(M)$-module of differentiable sections over $M$.

We already have endowed $T^\ast M$ with a Riemannian metric $g$, and thus we also have a metric on the tangent bundle $T M$ by (fiber-wise) duality; we denote this metric by $g_\ast$. The connection $\nabla$ defined, for all differentiable section $\omega$ of $V$, the linear map:
\begin{equation*}
\nabla\omega : X \in TM \mapsto \nabla_X \omega \in \Gamma V\text{.}
\end{equation*}
We define for all differentiable section $\omega$ the norm:
\begin{equation*}
\opnorm{\nabla\omega}{}{} = \sup_{x\in M} \sup\left\{ \sqrt{g^V_x(\nabla_X \omega, \nabla_X \omega)} : X \in TM, g_{\ast x}(X,X) = 1 \right\}\text{,}
\end{equation*}
i.e. the operator norm of $\nabla\omega$ for the underlying inner products valued in $C(M)$ on the module of vector fields and the module of sections of $V$.

For all differentiable $\omega \in \Gamma V$, we set:
\begin{equation*}
\CDN(\omega) = \max\left\{ \|\omega\|_{\Gamma V}, \opnorm{\nabla \omega}{}{} \right\} \text{.}
\end{equation*}

We now explicit some formulas which we will need. Since $M$ is compact and $V$ is locally trivial, there exists a finite atlas $\mathscr{U}$ such that for any chart $(U,\psi) \in \mathscr{U}$, we also have a local frame for $V$ over $U$, i.e. $k$ functions $\{e_1^U,\ldots,e_k^U\}$ such that for all $x\in U$, the set $\{e_1^U(X), \ldots, e_k^U(x)\}$ is a basis for the fiber $V_x$.

We moreover let $\mathscr{V}$ be an open cover of $M$ with the property that for all $O \in \mathscr{V}$, there exists $(U,\psi) \in \mathscr{U}$ such that the closure $\mathrm{cl}(O) \subseteq U$. This can always be achieved since $M$ is compact.

Let us fix $(U,\psi) \in \mathscr{U}$. Let $\{ \partial_1,\ldots,\partial_n \}$ be some set of tangent vector fields on $U$ such that for all $x\in U$, the set $\{\partial_1(x),\ldots,\partial_n(x)\}$ is a basis for $T_x M$.

We also note that $\nabla$ restricts to a metric connection for $(V,g^V)$ restricted to $U$, as a vector bundle over $U$. We shall tacitly identify $\nabla$ with its restriction.

For any $\omega \in \Gamma_V$, we now write $\omega = \sum_{j = 1}^n \omega_j e^U_j$ for $\omega_1,\ldots,\omega_k \in C^1(M)$. Now, if we write, for all $p,r$ in $\{1,\ldots, k\}$ and $q\in \{1,\ldots,n\}$:
\begin{equation*}
\nabla_{\partial_q} e_p = \sum_{r=1}^k \Gamma_{p q}^r e_r
\end{equation*}
noting that $e_1,\ldots,e_k$ are smooth so that the above expression makes sense, then, for all smooth $\omega$:
\begin{equation*}
\nabla_{\partial_q} \omega  = \sum_{r=1}^k \left(\partial_q \omega_r + \sum_{p=1}^k \Gamma_{p q}^r \omega_{p}\right)e_r \text{.}
\end{equation*}

In particular, for any $x\in U$, $q\in\{1,\ldots,n\}$ and $j\in\{1,\ldots,k\}$, we thus have:
\begin{equation*}
\left| \left( \partial_q \omega \right)_j (x) \right| \leq \left| \left( \nabla_{\partial_q} \omega (x) \right)_j  - \sum_{p=1}^k  \Gamma_{p q}^r \omega_{p} \right|\text{.}
\end{equation*}

We now need a few estimates. We first note that by construction, if for all $x\in U$ we set:
\begin{equation*}
G_x = \begin{pmatrix} g^V_x(e^U_1,e^U_1) & \hdots & g_x^V(e_k^U,e^U_1) \\
\vdots  & & \vdots \\
g_x^V(e^U_1, e^U_k) & \hdots & g_x^U(e^U_k, e^U_k)\end{pmatrix}
\end{equation*}
then $G: x\in U \mapsto G_x$ is a continuous function valued in the $k\times k$ (positive symmetric invertible) matrices. We endow the algebra of $k\times k$ matrices with the norm $\opnorm{\cdot}{}{k}$ induced by the usual inner product on $\C^d$.

Moreover by construction, if we set:
\begin{equation*}
\inner{\sum_{j=1}^k \omega_j e_j}{\sum_{r=1}^k \eta_k e_k}{x} = \sum_{j=1}^k \omega_j(x) \overline{\eta_j(x)}
\end{equation*}
then we have $g^V_x(\omega,\eta) = \inner{G \omega}{\eta}{x}$. Now, for all $x\in U$, we have:
\begin{equation*}
\begin{split}
\max\left\{ |\omega_j(x)| : j\in\{1,\ldots,k\} \right\} &\leq \sqrt{\inner{\omega}{\omega}{x}} \\ 
&\leq \sqrt{\opnorm{G_x^{-1}}{}{k} \inner{G \omega}{\omega}{x}} \\
&= \sqrt{\opnorm{G_x^{-1}}{}{k}} \sqrt{g^V_x(\omega,\omega)} \text{.}
\end{split}
\end{equation*}

We now pick any $O \subseteq \mathscr{V}$ such that the closure $\mathrm{cl}(O)$ lies inside our chosen $U$. In this case, $G$, and therefore $G^{-1}$, are continuous on the compact $\mathrm{cl}(O)$ and thus, $x\in \mathrm{O} \mapsto \opnorm{G_x^{-1}}{}{k}$ is bounded below and above; since the bounds are reached and $G_x$ is never null, we conclude that there exists $w > 0$ such that for all $x\in \mathrm{cl}(O)$ we have:
\begin{equation*}
\max\left\{ |\omega_j(x)| : j\in\{1,\ldots,k\} \right\} \leq w\sqrt{\inner{G\omega}{\eta}{x}} = w\sqrt{g^V_x(\omega,\omega)}\text{.}
\end{equation*}

In particular, we note that if $\CDN(\omega)\leq 1$ and $p \in \{1,\ldots,k\}$ then:
\begin{equation}\label{bounded-norm-eq}
\|\omega_p \|_{C(\mathrm{cl}(O))} \leq w \text{.}
\end{equation}

We also note that the functions $x\in \mathrm{cl}(O) \mapsto \sqrt{g_{\ast x}(\partial_q,\partial_q)}$ are continuous on a compact as well for all $q\in \{1,\ldots,d\}$, and thus we can choose $K > 0$ such that:
\begin{equation*}
\sup\left\{ \sqrt{g_{\ast x}(\partial_q,\partial_q)} : x \in \mathrm{cl}(O), q\in \{1,\ldots, n\} \right \} \leq K \text{.}
\end{equation*}

Last, we note that since the Christoffel symbols of our connection in our local frame are continuous as well, there exists $K_2 > 0$ such that:
\begin{equation*}
\sup\left\{ \left|\Gamma_{p q}^r\right| : p,r \in \{1,\ldots,k\}, q\in\{1,\ldots,n\}, x\in\mathrm{cl}(O) \right\} \leq K_2\text{.}
\end{equation*}

We have thus for all $\omega\in \Gamma V$ with $\CDN(\omega)\leq 1$, $x\in\mathrm{cl}(O)$ and $q\in\{1,\ldots,k\}$:
\begin{equation*}
\begin{split}
|\nabla_{\partial_q} \omega_j (x)| &\leq \max\left\{ |\nabla_{\partial_q}\omega_j(x)| : j\in\{1,\ldots,k\} \right\} \\
&\leq w \sqrt{g^V_x(\nabla_{\partial_q}\omega,\nabla_{\partial_q}\omega)} \\
&\leq w \opnorm{\nabla \omega}{}{} \sqrt{g_{\ast x}(\partial_q,\partial_q)} \\
&\leq w K  \text{.}
\end{split}
\end{equation*}
Therefore for all $q \in \{1,\ldots,n\}$ and $j\in\{1,\ldots,k\}$, and for all $\omega\in\Gamma V$ with $\CDN(\omega)\leq 1$, we estimate:

\begin{equation}\label{bounded-partial-der-eq}
\begin{split}
\left\| (\partial_q \omega)_j \right\|_{C(\mathrm{cl}(O))} &= \sup_{x\in \mathrm{cl}(O)} |\partial_q \omega (x)| \\
&\leq \sup_{x\in \mathrm{cl}(O)} \left| \left(\nabla_{\partial_q}\omega\right)_j(x)\right| \\
&\quad + k \sup\left\{ \left|\Gamma_{p q}^r (x)\right| : p,r \in \{1,\ldots,k\}, q\in \{1,\ldots,n\}, x\in\mathrm{cl}(O) \right\} \\
&\quad \times \max_{p=1}^k \|\omega_p\|_{C(\mathrm{cl}(O))} \\
&\leq w K + w k K_2 \text{.} 
\end{split}
\end{equation}

Let:
\begin{equation*}
\modlip{1}{\CDN,O} = \left\{ \text{the restriction of $\omega_j$ to $\mathrm{cl}(O)$} : \CDN(\omega)\leq 1, j \in \{1,\ldots,k \} \right\}\text{.}
\end{equation*}
We deduce from Equations (\ref{bounded-partial-der-eq}) and (\ref{bounded-norm-eq}) that $\modlip{1}{\CDN,O}$ is an equicontinuous family (in fact, a collection of $(K w + w k K_2)$-Lipschitz functions). 

We may now apply Arz{\'e}la-Ascoli theorem to the set $\modlip{1}{\CDN,O}$, viewed as an equicontinuous set of functions on the compact $\mathrm{O}$ and valued in a fixed compact in $\C$. Thus, $\modlip{1}{\CDN,O}$ is totally bounded for the uniform norm $\|\cdot\|_{\mathrm{cl}(O)}$.

Now, we also observe that since $G$ is bounded above as well, there exists $w_2 > 0$ such that for all $\omega\in\Gamma_V$ and $x\in \mathrm{cl}(O)$:
\begin{equation*}
g^V_x(\omega,\omega) \leq \opnorm{G_x}{}{k} \inner{\omega}{\omega}{x} \leq k\opnorm{G_x}{}{k} \max\{|\omega_1(x)|,\ldots,|\omega_k(x)|\}\text{.}
\end{equation*}

Therefore, we conclude that $\modlip{1}{\CDN}$ is totally bounded for the seminorm:
\begin{equation*}
\left\|\omega\right\|_{\Gamma V, O} = \omega\in \Gamma V \mapsto \sup_{x\in \mathrm{cl}(O)} \sqrt{g^V_x(\omega,\omega)} \text{.}
\end{equation*}

Last, we note that:
\begin{equation*}
\left\|\omega\right\|_{\Gamma V} = \sup_{x\in M} \sqrt{g^V_x(\omega,\omega)} = \max_{O\in\mathscr{V}} \|\omega\|_{\Gamma V, O}
\end{equation*}
from which it is easy to deduce that $\modlip{1}{\CDN}$ is totally bounded for $\|\cdot\|_{\Gamma V}$.

Our reasoning proves that for all $O \in \mathscr{V}$, the set of all the restrictions of elements $\omega\in\Gamma V$ with $\CDN(\omega)\leq 1$ is equicontinuous and obviously bounded. Since $\mathscr{V}$ is a finite cover of $M$, we then conclude that $\{\omega\in\Gamma V : \CDN(\omega)\leq 1\}$ is equicontinuous on the compact $M$ and valued in the common compact thus by Arz{\'e}la-Ascoli, we conclude that $\{\omega\in\Gamma V : \CDN(\omega)\leq 1\}$ is compact for the supremum norm.

Furthermore, since $\nabla$ is a connection, we check that:
\begin{equation*}
\CDN(f \omega) \leq \Lip(f) \|\omega\|_{\Gamma V} + \|f\|_{C(M)} \CDN(\omega)
\end{equation*}
for all $f\in C^1(X)$ and smooth $\omega\in\Gamma V$ (and using the fact that $\|\cdot\|_{\Gamma V}$ is a module norm), while since $\nabla$ is a metric connection, we also check that:
\begin{equation*}
\Lip(g^V(\omega,\eta)) \leq \CDN(\omega)\|\eta\|_{\Gamma V} + \|\omega\|_{\Gamma V} \CDN(\eta)
\end{equation*}
for all smooth $\omega,\eta\in\Gamma V$. 

As we did with $\Lip$, we extend $\CDN$ by defining $\CDN_V$ as the Minkowsky functional of the norm closure of $\modlip{1}{\CDN}$, which is compact. Thus $(\Gamma V, g^V, \CDN_V, C(M), \Lip)$ is a {\gQVB}.

We note that it is not clear in general what $\ker\nabla = \{ \omega \in \Gamma^1 V : \nabla \omega = 0\}$ might be; it is a key reason why we actually define our D-norms to dominate the underlying norm: by taking the maximum of the module norm and the norm of a connection, we remove the question of what the kernel should be and we can make a clear requirement of compactness for the closed unit ball of  our D-norm.
\end{example}

Example (\ref{Riemannian-ex}) justifies the following terminology:
\begin{definition}
A {\QVB{F}{G}{H}} is \emph{Leibniz} when:
\begin{enumerate}
\item $F = (x,y,l_x,l_y) \in [0,\infty)^4 \mapsto x l_y + y l_x$,
\item $G = (x,l,d) \in [0,\infty)^3 \mapsto (x + l) d$,
\item $H = (x,y) \in [0,\infty)^2 \mapsto 2 x y \text{.}$
\end{enumerate}
\end{definition}

We note that a the admissible triple for a Leibniz {\gQVB} is chosen to be the lower allowed bounds in Definition (\ref{admissible-triple-def}).

We now turn to examples of {\gQVB s} over general {\gQqcms s}. We begin with the observation that Definition (\ref{qvb-def}) contains, in the statement of the inner quasi-Leibniz inequality, a canonical extension of the L-seminorm to a dense *-subalgebra of the entire base quantum space. This extension possess properties which will prove helpful for our next few examples. 

\begin{lemma}\label{extension-of-Leibniz-lemma}
Let $(\A,\Lip)$ be a {\Qqcms{F}} for some admissible function $F$. The seminorm:
\begin{equation*}
\mathrm{M} : a\in\A\mapsto\max\{ \Lip(\Re a), \Lip(\Im a) \}
\end{equation*}
satisfies:
\begin{equation*}
\mathrm{M} (a b) \leq 8 F (\|a\|_\A, \|b\|_\A, \mathrm{M}(a), \mathrm{M}(b) ) \text{.}
\end{equation*}
Moreover the domain of $\mathrm{M}$ is a dense *-subalgebra of $\A$, while $\mathrm{M}$ satisfies $\mathrm{M}(a^\ast) = \mathrm{M}(a)$ for all $a\in\A$ and $\{a\in\A : \mathrm{M}(a) = 0\} = \C\unit_\A$.
\end{lemma}

\begin{proof}
We first observe that $\mathrm{M}$ is a seminorm (which may assume the value $\infty$, though it is finite on $\dom{\Lip} + i\dom{\Lip}$; moreover it is easy to check that $\mathrm{M}(a) = 0$ if and only if $a\in\C \unit_\A$). Moreover, $\mathrm{M}$ restricted to $\sa{\A}$ is of course $\Lip$. It is similarly straightforward to note that $\mathrm{M}(a^\ast) = \mathrm{M}(a)$ for all $a\in\A$.

Let $a,b \in \A$. We note, since $ab = \Jordan{a}{b} + i \Lie{a}{b}$ and $\mathrm{M}$ is a norm:
\begin{equation}\label{extension-of-Leibniz-eq1}
\mathrm{M}(ab)\leq \mathrm{M}(\Jordan{a}{b}) + \mathrm{M}(\Lie{a}{b}) \text{.}
\end{equation}
We then have:
\begin{equation*}
\begin{split}
\mathrm{M}(\Jordan{a}{b}) &\leq \mathrm{M}(\Jordan{\Re(a)}{\Re(b)}) + \mathrm{M}(\Jordan{\Re(a)}{\Im(b)}) \\
&\quad + \mathrm{M}(\Jordan{\Im(a)}{\Re(b)}) + \mathrm{M}(\Jordan{\Im(a)}{\Im(b)}) \\
&= \Lip(\Jordan{\Re(a)}{\Re(b)}) + \Lip(\Jordan{\Re(a)}{\Im(b)}) \\
&\quad + \Lip(\Jordan{\Im(a)}{\Re(b)}) + \Lip(\Jordan{\Im(a)}{\Im(b)}) \\
&\leq F(\|\Re(a)\|_\A,\|\Re(b)\|_\A,\Lip(\Re(a)), \Lip(\Re(b))) \\
&\quad + F(\|\Re(a)\|_\A,\|\Im(b)\|_\A,\Lip(\Re(a)), \Lip(\Im(b))) \\
&\quad + F(\|\Im(a)\|_\A,\|\Re(b)\|_\A,\Lip(\Im(a)), \Lip(\Re(b))) \\
&\quad + F(\|\Im(a)\|_\A,\|\Im(b)\|_\A,\Lip(\Im(a)), \Lip(\Im(b))) \\
&\leq 4F(\|a\|_\A,\|b\|_\A,\mathrm{M}(a), \mathrm{M}(b)) \text{.}
\end{split}
\end{equation*}
A similar computation allows us to conclude:
\begin{equation*}
\mathrm{M}(\Lie{a}{b}) \leq 4F(\|a\|_\A,\|b\|_\A,\mathrm{M}(a), \mathrm{M}(b))
\end{equation*}
and thus, using Inequality (\ref{extension-of-Leibniz-eq1}), our lemma is proven.
\end{proof}

\begin{remark}\label{extension-of-Leibniz-rmk}
Let $(X,\mathrm{d})$ be a compact metric space and $\Lip$ be the seminorm induced on $C(X)$ by $\mathrm{d}$. While $\Lip$ is the Lipschitz seminorm for functions from $X$ valued in $\C$ with its usual hermitian norm, we note that $\mathrm{M} = \max\{\Lip\circ\Re,\Lip\circ\Im\}$ is the Lipschitz seminorm of functions form $X$ to $\C$ endowed with the norm $\|x+iy\| = \max\{|x|,|y|\}$ for all $x,y\in\R^2$.

This observation is interesting because Blaschke theorem is valid for the seminorm $\mathrm{M}$: when $\C$ is endowed with $\|\cdot\|$ rather than its standard hermitian norm, a $k$-Lipschitz function on some subset of $X$ to $\C$ can be extended a $k$-Lipschitz function over $X$. Indeed, it is easy to check that for any two {\gQqcms s} $(\A,\Lip_\A)$ and $(\B,\Lip_\B)$, a *-morphism $\pi : \A \rightarrow \B$ is a quantum isometry if and only if $\max\{\Lip_\B\circ\Re,\Lip_\B\circ\Im\}$ is the quotient of $\max\{\Lip_\A\circ\Re,\Lip_\A\circ\Im\}$ and the notion of full quantum isometry extends similarly. Thus Lemma (\ref{extension-of-Leibniz-lemma}) provides a rather canonical way to extend L-seminorms while keeping all notions of isometries unchallenged.
\end{remark}

We first note that every {\gQqcms} defines a canonical {\gQVB} over itself. This observation implies that the modular propinquity will provide another metric on {\gQqcms s}, though we will prove that it is equivalent to the quantum Gromov-Hausdorff propinquity.

\begin{example}\label{algebra-as-module-ex}
Let $(\A,\Lip)$ be a {\Qqcms{F}}. The C*-algebra $\A$ is of course a left module over itself, using the multiplication of $\A$ on the left. The C*-algebra $\A$ is naturally a left $\A$-Hilbert module by setting for all $a,b \in \A$:
\begin{equation*}
\inner{a}{b}{\A} =  a b^\ast \text{.}
\end{equation*}
Note that for every state $\varphi$ of $\A$, the completion of the pre-Hilbert space $\A$ endowed with $\varphi\circ\inner{\cdot}{\cdot}{\A}$ provides the Gel'fand-Naimark-Segal representation associated with $\varphi$

The norm of $a\in\A$ is:
\begin{equation*}
\sqrt{\|a a^\ast\|_\A} = \|a\|_\A\text{,}
\end{equation*}
and thus the C*-norm $\|\cdot\|_\A$ and the $C^\ast$-Hilbert norm $\|\cdot\|_{\module{A}}$ agree. In particular, $(\A,\inner{\cdot}{\cdot}{\A})$ is complete for its norm.

We can now set $\CDN_\A ( a ) = \max\{\Lip(\Re a), \Lip(\Im a), \|a\|_\A\}$ for all $a\in\A$. It is easy to check, using Lemma (\ref{extension-of-Leibniz-lemma}), that:
\begin{equation*}
\Omega(\A) = (\A,\inner{\cdot}{\cdot}{\A},\CDN_{\A},\A,\Lip)
\end{equation*}
is a {\QVB{F}{8F}{8F}}.
\end{example}

We extend Example (\ref{algebra-as-module-ex}) to free modules. Free modules are of course basic examples, but they are also important since every finitely generated projective modules lies inside a free module; thus the construction in the next example would provide D-norms to many finitely generated projective modules under appropriate conditions. This being said, our main example \cite{Latremoliere17a} of non free, finitely generated projective modules in this paper --- Heisenberg modules over quantum $2$-tori --- will come with a D-norm from a connection, akin to Example (\ref{Riemannian-ex}) though involving very different techniques. The following example is thus a natural default source of D-norms, while our work may accommodate different D-norms if the context calls for it.

\begin{example}\label{free-module-ex}
Let $(\A,\Lip_\A)$ be a {\Qqcms{F}} for some admissible function $F$. Let $d\in\N\setminus\{0\}$. Let $\module{M} = \A^d$. The map:
\begin{equation*}
\inner{\begin{pmatrix} a_1 \\ \vdots \\ a_d \end{pmatrix}}{\begin{pmatrix} b_1 \\ \vdots \\ b_d \end{pmatrix}}{\module{M}} = \sum_{j=1}^d a_j b_j^\ast \text{,}
\end{equation*}
for all $\begin{pmatrix} a_1 \\ \vdots \\ a_d \end{pmatrix}, \begin{pmatrix} b_1 \\ \vdots \\ b_d \end{pmatrix} \in \A^d$, is an $\A$-inner product, for which $\module{M}$ is complete.

We set $\mathrm{M}_\A(a) = \max\{\Lip_\A(\Re a), \Lip_\A(\Im a)\}$ for all $a\in\A$, and then::
\begin{equation*}
\Lip_{\A}^d\begin{pmatrix} a_1 \\ \vdots \\ a_d \end{pmatrix} = \max\left\{ \mathrm{M}(a_j) : j \in \{1,\ldots,d\} \right\}
\end{equation*}
for all $a_1,\ldots,a_d \in \A$. $\Lip_{\module{M}}$ is a form of L-seminorm for modules. We note that $\Lip_{\A}^d(x) = 0$ if and only if $x\in \C^d$. 

We now define $\CDN_{\A}^d = \max\left\{\|\cdot\|_{\module{M}},\Lip_{\A}^d\right\}$. 

To begin with, we note that for all $a,b_1,\ldots,b_d \in \A$:
\begin{equation*}
\begin{split}
\Lip_{\A}^d\left(a \begin{pmatrix} b_1 \\ \vdots \\ b_d \end{pmatrix}\right) &= \max\left\{ \mathrm{M}(a b_j) : j\in\{1,\ldots,d\} \right\}\\
&= \max\left\{ 8 F(\|a\|_\A,\|b_j\|_\A,\mathrm{M}_\A(a),\mathrm{M}(b_j)) : j \in \{1,\ldots,d\} \right\}\\
&\leq 8 F\left(\|a\|_\A,\left\|\begin{pmatrix} b_1 \\ \vdots \\ b_d \end{pmatrix}\right\|_{\module{M}},\mathrm{M}(a),\Lip_{\A}^d\begin{pmatrix} b_1 \\ \vdots \\ b_d \end{pmatrix} \right)\\
&\leq G\left(\|a\|_\A,\mathrm{M}(a),\Lip_{\A}^d\begin{pmatrix} b_1 \\ \vdots \\ b_d \end{pmatrix} \right) \text{,}
\end{split}
\end{equation*}
where $G (x,y,z) = 8 F(x,y,z,y)$ for all $x,y,z\geq 0$.

Since $\|\cdot\|_{\module{M}}$ is a $C^\ast$-Hilbert norm and $\mathrm{M}(a) = \Lip_\A(a)$ if $a=a^\ast$, we conclude that if $a\in\sa{\A}$ then:
\begin{equation*}
\CDN_\A^d\left(a \begin{pmatrix} b_1 \\ \vdots \\ b_d \end{pmatrix}\right) \leq G\left(\|a\|_\A,\Lip_\A(a),\CDN_\A^d\begin{pmatrix} b_1 \\ \vdots \\ b_d \end{pmatrix}\right)\text{.}
\end{equation*}

Moreover, again using Lemma (\ref{extension-of-Leibniz-lemma}), we also have, for all $a_1,\ldots,a_d,b_1,\ldots,b_d \in \module{M}$, that:
\begin{equation*}
\begin{split}
\Lip_\A\left(\Re\inner{\begin{pmatrix} a_1 \\ \vdots \\ a_d \end{pmatrix}}{\begin{pmatrix} b_1 \\ \vdots \\ b_d \end{pmatrix}}{\module{M}}\right) &= \Lip_\A\left(\Re\left(\sum_{j=1}^d a_j b_j\right) \right) \\
&\leq \mathrm{M}\left(\sum_{j=1}^d a_j b_j \right) \\
&\leq 8 \sum_{j=1}^d F(\|a_j\|_\A,\|b_j\|_\A,\mathrm{M}(a_j),\mathrm{M}(b_j)) \\
&\leq 8 d F\left( \left\|\begin{pmatrix} a_1 \\ \vdots \\ a_d \end{pmatrix}\right\|_{\module{M}}, \left\|\begin{pmatrix} b_1 \\ \vdots \\ b_d \end{pmatrix}\right\|_{\module{M}}, \Lip_{\A}^d\begin{pmatrix} a_1 \\ \vdots \\ a_d \end{pmatrix},  \Lip_{\A}^d\begin{pmatrix} b_1 \\ \vdots \\ b_d \end{pmatrix} \right) \\
&\leq d F\left( \CDN_\A^d\begin{pmatrix} a_1 \\ \vdots \\ a_d \end{pmatrix}, \CDN_\A^d\begin{pmatrix} b_1 \\ \vdots \\ b_d \end{pmatrix}, \CDN_\A^d\begin{pmatrix} a_1 \\ \vdots \\ a_d \end{pmatrix},\CDN_\A^d\begin{pmatrix} b_1 \\ \vdots \\ b_d \end{pmatrix}  \right) \\
&= H\left(\CDN_\A^d\begin{pmatrix} b_1 \\ \vdots \\ b_d \end{pmatrix},\CDN_\A^d\begin{pmatrix} b_1 \\ \vdots \\ b_d \end{pmatrix}\right)\text{,}
\end{split}
\end{equation*}
where $H(x,y) = 8 d F(x,y,x,y)$ for all $x,y \geq 0$.

If is immediate that $(F,G,H)$ is an admissible triplet. Thus Conditions (1), (3) and (4) of Definition (\ref{qvb-def}) are met.

Last, let $(a_1^n,\ldots,a_d^n) \in \module{M}^\N$ be a sequence such $\CDN_\A^d\begin{pmatrix} a_1^n \\ \vdots \\ a_d^n \end{pmatrix} \leq 1$ for all $n\in\N$.

Thus $(\Re a_1^n)_{n\in\N}$ lies in the compact $\{a\in\dom{\Lip_\A} : \Lip_\A(a)\leq 1, \|a\|_\A\leq 1\}$; we thus may extract a $\|\cdot\|_\A$-convergent subsequence $(\Re a_1^{f_1(n)})_{n\in\N}$ with limit $a_1 \in\sa{\A}$ such that $\Lip_\A(a_1)\leq 1$ and $\|a_1\|_\A\leq 1$ (we used the fact that $\Lip_\A$ is lower semicontinuous with respect to $\|\cdot\|_\A$). For the same reason, we can then extract convergent subsequences $(\Re a_2^{f_1\circ f_2(n)})_{n\in\N}$ of $(\Re a_2^{f_1(n)})_{n\in\N}$ with limit $\Re a_2\in\sa{\A}$, \ldots, $(\Re a_d^{f_1\circ f_2 \circ \cdots \circ f_d(n)})_{n\in\N}$ from $(\Re a_d^{f_1\circ f_2 \circ\cdots \circ f_{d-1}(n)})_{n\in\N}$ with limit $a_d\in\sa{\A}$; moreover $\max\{\Lip_\A(a_j),\|a_j\|_\A : j\in\{1,\ldots,n\} \} \leq 1$. 

If $g : n\in\N\mapsto f_1\circ f_2 \circ \cdots \circ f_d(n)$, then $\begin{pmatrix} \Re a_1^{g(n)} \\ \vdots \\ \Re a_d^{g(n)} \end{pmatrix}_{n\in\N}$ converges to $\begin{pmatrix} a_1 \\ \vdots \\ a_d \end{pmatrix}$, where by construction $\CDN_\A^d\begin{pmatrix} a_1 \\ \vdots \\ a_d \end{pmatrix}\leq 1$.

Just as easily, we can prove that there exists $b_1,\ldots,b_d  \in \sa{\A}$ and a function $h : \N\rightarrow\N$ strictly increasing such that:
\begin{equation*}
\lim_{n\rightarrow\infty} \begin{pmatrix} \Im a_1^{g(h(n))} \\ \vdots \\ \Im a_d^{g(h(n))} \end{pmatrix} = \begin{pmatrix} b_1 \\ \vdots \\ b_d \end{pmatrix}
\end{equation*}
and therefore:
\begin{equation*}
\lim_{n\rightarrow\infty} \begin{pmatrix} a_1^{g(h(n))} \\ \vdots \\ a_d^{g(h(n))} \end{pmatrix} = \begin{pmatrix} a_1 + i b_1 \\ \vdots \\ a_d + i b_d \end{pmatrix}\text{.}
\end{equation*}
By construction, $\CDN\begin{pmatrix} a_1 + i b_1 \\ \vdots \\ a_d + i b_d \end{pmatrix} \leq 1$.

Thus $(\module{M},\inner{\cdot}{\cdot}{\module{M}},\CDN_\A^d,\A,\Lip_A)$ is a {\QVB{F}{G}{H}}.
\end{example}

We make a few additional comments on our Definition (\ref{qvb-def}). Condition (3) will be used to prove that distance zero for the modular propinquity will in particular give rise to a module morphism. Condition (4) connects the metric structures of the D-norm and the L-seminorms. Condition (1) is just a normalization condition: indeed, the unit sphere for a D-norm is compact for the $C^\ast$-Hilbert norm and thus the norm attains a maximum on it; thus the $C^\ast$-Hilbert norm is always less than some constant multiple of the D-norm, thanks to Condition (2). Last, Condition (2) provides us with the compact set we will use to start the construction of a Gromov-Hausdorff distance for modules.

A consequence of Condition (4) is an additional implicit structure in {\gQVB s}:

\begin{remark}
Definition (\ref{qvb-def}) implies that, given a {\gQVB} $(\module{M},\inner{\cdot}{\cdot}{\module{M}},\CDN_{\module{M}},\A,\Lip_\A)$, the space $\dom{\CDN_{\module{M}}}$ is a left module over the Jordan-Lie algebra $\dom{\Lip_\A}$, and that the inner product $\inner{\cdot}{\cdot}{\module{M}}$ restricts to an $\dom{\Lip_\A}$-valued inner product on $\dom{\CDN_{\module{M}}}$.
\end{remark}

Condition (1) (as well as Condition (2)) implies that a D-norm is lower semi-continuous with respect to the $C^\ast$-Hilbert norm, which implies:

\begin{remark}
Let $(\module{M},\inner{\cdot}{\cdot}{\module{M}},\CDN_{\module{M}},\A,\Lip_\A)$ be a {\gQVB}. The $\C$-vector space $(\dom{\CDN_{\module{M}}},\CDN_{\module{M}})$ is a Banach space. Indeed, $\CDN_{\module{M}}$ is lower semi-continuous with respect to $\|\cdot\|_\module{M}$ since its unit ball is compact, hence closed, for $\|\cdot\|_{\module{M}}$; moreover $\CDN_{\module{M}}$ dominates $\|\cdot\|_{\module{M}}$.

The proof that the lower semi-continuity of $\CDN_{\module{M}}$ implies that $(\module{M},\CDN_{\module{M}})$ is a Banach space is identical to the proof that $(\dom{\Lip_\A},\max\{\|\cdot\|_\A,\Lip_\A\})$ is a Banach space, as found in \cite{Latremoliere16b}.
\end{remark}

The category of {\gQVB s}, whose objects are introduced in Definition (\ref{qvb-def}), is constructed using the following natural notion of morphism.

\begin{definition}\label{qvb-morphism-def}
Let:
\begin{equation*}
\Omega_\A = (\module{M},\inner{\cdot}{\cdot}{\module{M}}, \CDN_{\module{M}},\A,\Lip_\A)\text{ and }\Omega_\B = (\module{N},\inner{\cdot}{\cdot}{\module{N}}, \CDN_{\module{N}},\B,\Lip_\B)
\end{equation*}
be two {\gQVB s}. A \emph{morphism} $(\Theta,\theta)$ from $\Omega_\A$ to $\Omega_\B$ is a unital module morphism $(\Theta,\theta)$ such that:
\begin{enumerate}
\item $\theta$ is continuous from $(\dom{\Lip_\A},\Lip_\A)$ to $(\dom{\Lip_\B},\Lip_\B)$, i.e. there exists $C > 0$ such that $\Lip_\B\circ\theta \leq C\Lip_\A$ on $\dom{\Lip_\A}$,
\item $\Theta$ is continuous from $(\dom{\CDN_{\module{M}}},\CDN_{\module{M}})$ to $(\dom{\CDN_{\module{N}}},\CDN_{\module{N}})$, i.e. there exists $M > 0$ such that for all $\omega\in\module{M}$ we have $\CDN_{\module{N}}(\Theta(\omega)) \leq M \CDN_{\module{M}}(\omega)$.
\end{enumerate}
Such a \emph{morphisms} is an \emph{epimorphism} when both $\theta$ and $\Theta$ are surjective, and a monomorphism when both $\theta$ and $\Theta$ are both monomorphisms.

A \emph{isomorphism} is thus given by a morphism $(\Theta,\theta)$ where $\theta$ is a *-isomorphism, $\Theta$ is a bijection and $(\Theta^{-1},\theta^{-1})$ is a morphism from $\Omega_\B$ onto $\Omega_\A$.
\end{definition}

As is customary with categories of metric spaces, there are several appropriate of isomorphisms. Inside the general category described via Definitions (\ref{qvb-def}) and (\ref{qvb-morphism-def}), an isomorphism would be a generalization of a bi-Lipschitz map. For our purpose, a stronger notion of isomorphism will be employed, akin to a notion of isometry. We first define the notion of a full quantum isometry, which is rather self-evident:

\begin{definition}\label{quantum-modular-isometry-def}
Let:
\begin{equation*}
\Omega_\A = (\module{M},\inner{\cdot}{\cdot}{\module{M}}, \CDN_{\module{M}},\A,\Lip_\A)\text{ and }\Omega_\B = (\module{N},\inner{\cdot}{\cdot}{\module{N}}, \CDN_{\module{N}},\B,\Lip_\B)
\end{equation*}
be two {\gQVB s}. A \emph{full quantum isometry} $(\Theta,\theta)$ from $\Omega_\A$ to $\Omega_\B$ is a {\gQVB} isomorphism from $\Omega_\A$ to $\Omega_\B$ such that:
\begin{enumerate}
\item $\Lip_\B\circ\theta = \Lip_\A$ on $\dom{\Lip_\A}$,
\item $\CDN_{\module{N}}\circ\Theta = \CDN_{\module{M}}$ on $\dom{\CDN_{\module{M}}}$.
\end{enumerate}
\end{definition}

It is easy to check that the category of {\gQVB s} with quantum isometries as morphisms is a subcategory of the category whose morphisms are given by Definition (\ref{qvb-morphism-def}).

The more delicate question for us regards the notion of a quantum isometry for {\gQVB s}. As we discussed when introducing quantum isometries for {\gQqcms s}, isometries between L-seminorms rely on working with self-adjoint elements only. We did note in Remark (\ref{extension-of-Leibniz-rmk}) that we can extend L-seminorms to bypass this issue, though the situation for module requires some idea.

We propose to circumvent this issue by bringing the problem down to the base quantum spaces. Indeed, we take advantage of the observation that the inner quasi-Leibniz inequality implies that for any {\gQVB} $(\module{M},\inner{\cdot}{\cdot}{\module{M}},\CDN,\A,\Lip)$ and any $\omega,\eta\in\module{M}$, the elements $\Re\inner{\omega}{\eta}{\module{M}}$ and $\Im\inner{\omega}{\eta}{\module{M}}$ lies in the domain $\dom{\Lip}$ of the L-seminorm $\Lip$. Thus, the tools developed for the Gromov-Hausdorff propinquity can be brought to bare to the study of {\gQVB s}.

Now, while the Cauchy-Schwarz inequality for the $C^\ast$-Hilbert norm of a left pre-Hilbert module $(\module{M},\inner{\cdot}{\cdot}{\module{M}})$ implies that:
\begin{equation}\label{module-norm-CS-eq}
\|\omega\|_{\module{M}} = \sup\{\left|\inner{\omega}{\eta}{\module{M}}\right| : \eta\in\module{M}, \|\eta\|_{\module{M}} \leq 1\}\text{,}
\end{equation}
we want to work only with elements bounded for the D-norms. We now study the metric on the domain of $D$-norm resulting from working with Expression (\ref{module-norm-CS-eq}), with the closed unit ball for the $C^\ast$-Hilbert norm replaced by the closed unit ball for the D-norm. We begin with a notation we shall use throughout this paper.

\begin{notation}
Let $\Omega = \left(\module{M},\inner{\cdot}{\cdot}{\module{M}}, \CDN_{\module{M}}, \A, \Lip_\A\right)$ be a {\QVB{F}{G}{H}} for some admissible $(F,G,H)$. The closed ball of center $0$ and radius $r \geq 0$ in $(\dom{\CDN_{\module{M}}}, \CDN_{\module{M}})$ is denoted by:
\begin{equation*}
\modlip{r}{\Omega} = \left\{ \omega \in \dom{\CDN_{\module{M}}} : \CDN_{\module{M}}(\omega) \leq r \right\}\text{.}
\end{equation*}
By Definition (\ref{qvb-def}), the set $\modlip{r}{\Omega}$ is norm compact.
\end{notation}

We call the initial topology for a set $\mathcal{F}$ of functions defined on a given set $E$ and valued in a topological space, the smallest topology on $E$ for which all the members of $\mathcal{F}$ are continuous.

\begin{definition}
Let $(\module{M},\inner{\cdot}{\cdot}{\module{M}})$ be a left Hilbert $\A$-module. The \emph{$A$-weak topology} on $\module{M}$ is the initial topology for the set of maps:
\begin{equation*}
\left\{ \inner{\cdot}{\omega}{\module{M}} : \omega \in \module{M} \right\}\text{.}
\end{equation*}
\end{definition}

Thus, a net $(\omega_j)_{j\in J}$ converges to some $\omega$ in a left Hilbert $\A$-module $(\module{M},\inner{\cdot}{\cdot}{\module{M}})$ when for all $\eta\in\module{M}$, the net $\left(\inner{\omega_j}{\eta}{\module{M}}\right)_{j\in J}$ converges to $\inner{\omega}{\eta}{\module{M}}$ in $\A$.

Thus, in particular, for any {\gQVB} $\Omega$, the set $\modlip{1}{\Omega}$ is now endowed with three topologies: the norm topology from the D-norm, the norm topology from the $C^\ast$-Hilbert norm inherited from the inner product, and the $\A$-weak topology where $\A$ is the base space of $\Omega$. Definition (\ref{qvb-def}) assures us that the latter two agree on $\modlip{1}{\Omega}$.

\begin{lemma}\label{Aweak-norm-agree-lemma}
Let $\Omega = (\module{M},\inner{\cdot}{\cdot}{\module{M}},\CDN_{\module{M}},\A,\Lip)$ be a {\gQVB}. The $\A$-weak topology and the norm topology, induced by $\inner{\cdot}{\cdot}{\module{M}}$, agree on $\modlip{r}{\Omega}$ for all $r \geq 0$.
\end{lemma}

\begin{proof}
The $\A$-weak topology is weaker than the topology induced by $\|\cdot\|_{\module{M}}$, yet Hausdorff. On the other hand, $\modlip{r}{\Omega}$ is compact for $\|\cdot\|_{\module{M}}$. Therefore, the $\A$-weak topology and the topology from $\|\cdot\|_{\module{M}}$ agree on $\modlip{r}{\Omega}$.
\end{proof}

Our reason to introduce the $\A$-weak topology is that it is naturally metrized by a metric defined from $\modlip{1}{\Omega}$. 

\begin{definition}\label{modular-Kantorovich-def}
Let $\Omega = (\module{M},\inner{\cdot}{\cdot}{\module{M}},\CDN,\A,\Lip_\A)$ be a {\gQVB}. The \emph{modular \MongeKant} $\KantorovichMod{\Omega}$ associated with $\Omega$ is the metric $\KantorovichMod{\Omega}$, defined for $\omega,\eta\in\module{M}$ by:
\begin{equation*}
\KantorovichMod{\Omega}(\omega,\eta) = \sup\left\{ \|\inner{\omega}{\xi}{\module{M}} - \inner{\eta}{\xi}{\module{M}}\|_\A : \xi \in \module{M}, \CDN(\xi) \leq 1  \right\}\text{.}
\end{equation*}
\end{definition}
We note that the {\MongeKant} on the module of a {\gQVB} is indeed a metric since the $\C$-linear span of the closed unit ball for the D-norm is dense in the module itself by assumption. We now prove the main property of this metric for us:

\begin{proposition}\label{modular-Kantorovich-prop}
Let $\Omega = (\module{M},\inner{\cdot}{\cdot}{\module{M}},\CDN,\A,\Lip_\A)$ be a {\gQVB}. For all $r > 0$, the \emph{\MongeKant} $\KantorovichMod{\Omega}$ associated with $\Omega$ metrizes the $\A$-weak* topology on $\modlip{r}{\Omega}$, and therefore it metrizes the norm topology on $\modlip{r}{\Omega}$.
\end{proposition}

\begin{proof}
Let $(\omega_j)_{j \in J}$ be a net in $\modlip{r}{\Omega}$, indexed by $(J,\succ)$ and converging to $\omega$ in the $\A$-weak topology. Let $\varepsilon > 0$. Since $\modlip{1}{\Omega}$ is compact by assumption, there exists a finite subset $F\subseteq \modlip{1}{\Omega}$ which is $\frac{\varepsilon}{3}$-dense in $\modlip{r}{\Omega}$.

There exists $j_0 \in J$ such that for all $j\succ j_0$ we have $\|\inner{\omega_j}{\xi}{\module{M}} - \inner{\omega}{\xi}{\module{M}}\|_\A \leq \frac{\varepsilon}{3}$ for all $\xi\in F$ since $F$ is finite and $J$ is directed. It follows that $\|\inner{\omega_j}{\xi}{\module{M}} - \inner{\omega}{\xi}{\module{M}}\|_\A \leq \varepsilon$. Thus $\KantorovichMod{\Omega}(\omega_j,\omega) \leq \varepsilon$. 

Conversely, if $(\omega_n)_{n\in\N}$ is a sequence in $\modlip{r}{\Omega}$ converging to some $\omega \in \module{M}$ for $\KantorovichMod{\Omega}$, then since $\modlip{1}{\Omega}$ is total, we conclude that $\left(\inner{\omega_n}{\xi}{\module{M}}\right)_{n\in\N}$ converges to $\inner{\omega}{\xi}{\module{M}}$ for all $\xi \in \module{M}$. Thus $\omega$ is the $\A$-weak limit of $(\omega_n)_{n\in\N}$.

Now, since the $\A$-weak topology and the norm topology agree on $\modlip{r}{\Omega}$, and $\modlip{r}{\Omega}$ is compact in norm by assumption, we conclude that $\omega \in \modlip{r}{\Omega}$ as well. This concludes our proof.
\end{proof}

We conclude this section with an observation. Let $\Omega = (\module{M},\inner{\cdot}{\cdot}{\module{M}},\CDN,\A,\Lip)$ be a {\gQVB}. Let $\B = \{ a + i b : a,b \in \dom{\Lip}\}$ endowed with the norm $\|\cdot\|_\B = \max\{\|\cdot\|_\A, \Lip\circ\Re, \Lip\circ\Im\}$. The norm $\|\cdot\|_\B$ is lower semicontinuous with respect to $\|\cdot\|_\A$ and thus one can prove that $(\B,\|\cdot\|_\B)$ is a Banach algebra (the fact that it is a subalgebra of $\A$ follows from the fact $\dom{\Lip}$ is a Jordan-Lie subalgebra of $\sa{\A}$). We note that $\dom{\CDN}$ is a $\B$-left module thanks to Definition (\ref{qvb-def}).

We also noted that $(\dom{\CDN},\CDN)$ is a Banach space as well. Let us call a \emph{current} of $\Omega$ a continuous $\B$-module map from $(\dom{\CDN},\CDN)$  to $(\B,\|\cdot\|_\B)$. Let $\mathscr{C}(\Omega)$ be the space of all currents of $\Omega$ and let $\mathscr{C}_r(\Omega)$ be the closed ball of radius $r > 0$ centered at $0$ for the operator norm on $\mathscr{C}(\Omega)$.

Let us call the locally convex topology induced by the seminorms:
\begin{equation*}
T \in \mathscr{C}(\Omega) \mapsto \|T(\omega)\|_\A
\end{equation*}
on $\mathscr{C}(\Omega)$ for all $\omega\in\dom{\CDN}$ the \emph{$\A$-weak* topology}. 

Let $\mathfrak{L} = \{ a\in\B : \|a\|_\B \leq 1\}$. By assumption on the L-seminorm $\Lip$, the set $\mathfrak{L}$ is compact in $(\A,\|\cdot\|_\A)$. If we let $\Xi = \prod_{\omega\in\dom{\CDN}} r\CDN(\omega) \mathfrak{L}$, then $\Xi$ is compact in the product topology by the Tychonoff theorem.

By construction, if $T \in \mathscr{C}_r(\Omega)$ then $\Theta(T) = (T(\omega))_{\omega\in\dom{\CDN}} \in\Xi$. It is straightforward to check that $\Theta$ is a continuous injection from the $\A$-weak* topology to the product topology, whose range is given by:
\begin{equation*}
\Xi^r = \bigcap_{b\in\B, \omega,\eta\in \dom{\CDN}} \left(\pi_{b\omega + \eta} - b\pi_{\omega} - \pi_\eta\right)^{-1}(\{0\})\text{,}
\end{equation*}
where $\pi_\omega : (b_\eta)_{\eta\in\dom{\CDN}} \in \Xi \mapsto b_\omega$ for all $\omega \in \dom{\CDN}$. Of course, by definition of the product topology on $\Xi$, the maps $\pi_\omega$ are continuous for all $\omega\in\dom{\CDN}$ and thus $\Xi^r$ is closed in $\Xi$, hence compact. It is easy to see that $\Theta$ is an homeomorphism and thus $\mathscr{C}_r(\Omega)$ is actually compact for the $\A$-weak* topology. These facts did not involve the fact that $\modlip{1}{\Omega}$ is compact for $\|\cdot\|_{\module{M}}$.

Suppose now that $\mathrm{S}$ is a seminorm on $\module{M}$ satisfying the inner quasi-Leibniz inequality:
\begin{equation*}
\Lip(\inner{\omega}{\eta}{\module{M}}) \leq H(\mathrm{S}(\omega),\mathrm{S}(\eta))
\end{equation*}
and $\mathrm{S} \geq \|\cdot\|_{\module{M}}$. The latter equation implies that the closed unit ball for $\mathrm{S}$ is also closed in the norm $\|\cdot\|_{\module{M}}$. 

The inner quasi-Leibniz inequality implies in turn that if $\mathrm{S}(\omega)\leq 1$ for $\omega\in\module{M}$ then $\inner{\cdot}{\omega}{\module{M}}$ is a $H(1,1)$ current, and thus belongs to some compact set for the $\A$-weak* topology. It is easy to check that the map $\omega\in\module{M} \mapsto \inner{\cdot}{\omega}{\module{M}}$ is an homeomorphism from $\module{M}$ with the $\A$-weak topology to its range, with the $\A$-weak* topology. Thus we conclude that $\{\omega\in\module{M} : \mathrm{S}(\omega)\leq 1\}$ is totally bounded for the $\A$-weak topology.

This does not however make $\mathrm{S}$ a D-norm, even if it satisfies some form of modular quasi-Leibniz inequality. Indeed, while norm closed, it is unclear whether the closed unit ball of $\mathrm{S}$ is also $\A$-weak closed. Moreover, even it is was, we could not deduce that the closed unit ball for $\mathrm{S}$ is norm compact, rather than weakly compact. Thus, it does not appear to be sufficient to assume the inner and modular quasi-Leibniz inequalities and the dominance over the $C^\ast$-Hilbert norm to construct a D-norm, and the compactness of the closed ball of a D-norm requires some additional work.

We now turn to the construction of the modular propinquity. The basic ingredient is a notion of a modular bridge which extends the notion of a bridge used as a noncommutative encoding of the idea of an isometric embedding in the construction of the quantum Gromov-Hausdorff propinquity. 

\section{Modular Bridges}

Bridges between C*-algebras provide a mean to define a particular type of isometric embedding for a any pair of {\gQqcms s}, from which the quantum Gromov-Hausdorff propinquity is built in \cite{Latremoliere13}. An advantage of bridges is that the quantum propinquity is defined directly from numerical quantities defined using a bridge rather than through the associated isometric embeddings, and thus they are natural to use a foundation for our modular propinquity --- bypassing the need for a notion of isometric embeddings of {\gQVB s}. We present our idea on how to extend the notion of bridges to {\gQVB s} in this section. While our presentation will at times refer to \cite{Latremoliere13}, we will strive to make it reasonably self-contained.

Bridges involve an element of a unital C*-algebra called a pivot, which allows us to select a particular set of states. We require that this set is not empty. The following definition extends the notion of a state defined on some self-adjoint element \cite[Exercise 4.6.16]{Kadison91},\cite{Kadison97} to general elements.

\begin{definition}[{\cite[Definition 3.1]{Latremoliere13}}]\label{level-set-def}
The \emph{$1$-level set} $\StateSpace_1(\D|x)$ of an element $x \in \sa{\D}$ of a unital C*-algebra $\D$ is:
\begin{equation*}
\StateSpace_1(\D|x) = \left\{ \varphi \in \StateSpace(\D) \middle\vert 
\begin{array}{l}
\varphi\left((1-x)^\ast (1-x)\right) = 0 \\ 
\varphi\left((1-x)(1-x)^\ast\right) = 0
\end{array}\right\}\text{.}
\end{equation*}
\end{definition}

We make the following remark:
\begin{remark}
If $x\in \D$ for some unital C*-algebra $\D$, and if $\StateSpace_1(\D|x) \not= \emptyset$, then $\|x\|_\D \geq 1$. Indeed, if $\varphi \in \StateSpace(\D|x)$ then $\varphi(x) = \varphi(x^\ast) = 1$. Thus $\|\Re(x)\|_\D \geq 1$ and thus $\|x\|_\D \geq \|\Re x\|_\D \geq 1$.

In particular, if $\|x\|_\D\leq 1$ and $\StateSpace(\D|x) \not= \emptyset$ then $\|x\|_\D = 1$.
\end{remark}

We will use Definition (\ref{level-set-def}) via the following lemma:

\begin{lemma}[{\cite[Lemma 3.4]{Latremoliere13}}]
If $\D$ is a unital C*-algebra and $x\in \D$, then:
\begin{equation*}
\begin{split}
\StateSpace(\D|x) &= \left\{ \varphi \in \StateSpace(\D) : \forall d \in \D \quad \varphi(d x) = \varphi(x d) = \varphi(d) \right\} \\
&= \left\{ \varphi \in \StateSpace(\D) : \forall d \in \D \quad \varphi(d x^\ast) = \varphi(x^\ast d) = \varphi(d) \right\} \text{.}
\end{split}
\end{equation*}
\end{lemma}

\begin{proof}
This follows from the Cauchy-Schwarz inequality.
\end{proof}

We first extend the notion of a bridge between {\gQqcms s} to a modular bridge between {\gQVB s}. While a bridge between to {\gQqcms s} does not involve any metric information in its definition --- the quantum metric information is used to associate numerical quantities to the bridge --- a modular bridge between two {\gQVB s} involve the D-norms. Nonetheless, a modular bridge retains the simplicity of a bridge, as it only adds two families of elements from modules.

\begin{definition}\label{modular-bridge-def}
Let:
\begin{equation*}
\Omega_\A = (\module{M},\inner{\cdot}{\cdot}{\module{M}},\CDN_{\module{M}},\A,\Lip_\A)\text{ and }\Omega_\B = (\module{N}, \inner{\cdot}{\cdot}{\module{N}}, \CDN_{\module{N}}, \B, \Lip_\B)
\end{equation*}
be two {\gQVB s}.

An \emph{modular bridge}:
\begin{equation*}
\bridge{\gamma} = \left(\Omega_\A,\Omega_\B,\D,x,\pi_\A,\pi_\B,(\omega_j)_{j\in J},(\eta_j)_{j \in J} \right)
\end{equation*}
from $\Omega_\A$ to $\Omega_\B$ is given:
\begin{enumerate}
\item a unital C*-algebra $\D$,
\item an element $x \in \D$ with $\StateSpace_1(\D|x) \not= \emptyset$ and $\|x\|_\D = 1$,
\item $\pi_\A : \A \hookrightarrow \D$ and $\pi_\B : \B \hookrightarrow \D$ are two unital *-monomorphisms,
\item $J$ is some nonempty set,
\item $(\omega_j)_{j\in J}$ is a family of elements in $\modlip{1}{\Omega_\A}$, i.e. $\max\{\CDN_{\module{M}}(\omega_j) : j\in J \} \leq 1$,
\item $(\eta_j)_{j\in J}$ is a family of elements in $\modlip{1}{\Omega_\B}$, i.e. $\max\{\CDN_{\module{N}}(\eta_j) : j\in J \} \leq 1$.
\end{enumerate}
\end{definition}

\begin{notation}
Let $\gamma = (\Omega_\A,\Omega_\B,\D,x,\pi,\rho,(\omega_j)_{j\in J},(\eta_j)_{j\in J})$ be a modular bridge. We will use the following notations and terminology throughout this paper.
\begin{enumerate}
\item The \emph{domain} $\dom{\gamma}$ of $\gamma$ is $\Omega_\A$.
\item The \emph{co-domain} $\codom{\gamma}$ of $\gamma$ is $\Omega_\B$.
\item The element $x$ is called \emph{the pivot of $\gamma$} and is denoted by $\bridgepivot{\gamma}$.
\item The family $(\omega_j)_{j\in J}$ is the \emph{family of anchors} of $\gamma$, denoted by $\bridgeanchors{\gamma}$.
\item The family $(\eta_j)_{j\in J}$ is the \emph{family of co-anchors} of $\gamma$, denoted by $\bridgecoanchors{\gamma}$.
\end{enumerate}
\end{notation}

\begin{notation}
Let $\Omega_\A$ and $\Omega_\B$ be two {\gQVB s}. The set of all modular bridges from a $\Omega_\A$ to $\Omega_\B$ is denoted by $\bridgeset{\Omega_\A}{\Omega_\B}{}$.
\end{notation}

We note that since modular bridges are defined as tuples, the order of their component matter and thus they have a domain and a codomain, though in fact they are quite a symmetric concept. We will remark later that all the quantities defined from modular bridges are in fact symmetric in the domain and the codomain.

We also remark that we include the domain and the codomain of a modular bridge in its very definition. This choice will in fact simplify our notations later on, by removing the need to explicit the quantum metric data as in \cite{Latremoliere13} for various quantities associated to modular bridges.

Last, we note that unlike \cite[Definition 3.6]{Latremoliere13}, we require the pivot of modular bridges are of norm (at most) $1$. This requirement will be essential in the proof of Proposition (\ref{bridge-reach-prop}), which in turn underlies the construction of the modular propinquity.

\bigskip

The basic ingredients to compute the modular propinquity between modules require a lot of notations. We clarify our exposition by grouping some of these notations into a single set of hypothesis which we will use repeatedly in the following definitions and theorems.

\begin{hypothesis}\label{modular-bridge-hyp}
Let:
\begin{equation*}
\Omega_\A = (\module{M},\inner{\cdot}{\cdot}{\module{M}},\CDN_{\module{M}},\A,\Lip_\A)\text{ and }\Omega_\B = (\module{N},\inner{\cdot}{\cdot}{\module{N}},\CDN_{\module{N}},\B,\Lip_\B)
\end{equation*}
be two {\QVB{F}{G}{H}s}. 

Let $J$ be some nonempty set and let:
\begin{equation*}
\gamma = (\Omega_\A,\Omega_\B,\D,x,\pi,\rho,(\omega_j)_{j\in J},(\eta_j)_{j\in J})
\end{equation*}
be a modular bridge from $\Omega_\A$ to $\Omega_\B$.
\end{hypothesis}

The modular propinquity is computed from natural numerical quantities obtained from modular bridges and the quantum metric information encoded in {\gQVB s}. The first quantities we will use are in fact the numerical values introduced in \cite{Latremoliere13} for the canonical bridge from $\basespace{\dom{\gamma}}$ to $\basespace{\codom{\gamma}}$ associated to any modular bridge $\gamma$:

\begin{definition}\label{basic-bridge-def}
Let Hypothesis (\ref{modular-bridge-hyp}) be given. The \emph{basic bridge} $\gamma_\flat$ from $\A$ to $\B$ is given by:
\begin{equation*}
\gamma_\flat = ( \D, x, \pi_\A, \pi_\B ) \text{.}
\end{equation*}
\end{definition}

It is straightforward that Definition (\ref{basic-bridge-def}) gives a bridge in the sense of \cite[Definition 3.6]{Latremoliere13}. Thus, we can compute the reach and height of a basic bridge. We adjust our terminology to fit the setting of this paper in the following definitions of the height and basic reach of a modular bridge.

We start by recalling from \cite[Definition 3.10]{Latremoliere13} that a bridge defines an important seminorm:

\begin{definition}[{\cite[Definition 3.10]{Latremoliere13}}]\label{bridge-seminorm-def}
Let Hypothesis (\ref{modular-bridge-hyp}) be given. The \emph{bridge seminorm} $\bridgenorm{\gamma}{\cdot,\cdot}$ of the modular bridge $\gamma$ is the bridge seminorm of the basic bridge $\gamma_\flat$, i.e. the seminorm on $\A\oplus\B$ defined for all $a\in\A$ and $b\in\B$ by:
\begin{equation*}
\bridgenorm{\gamma}{a,b} = \left\| \pi_\A\left(a\right) x - x \pi_\B\left(b\right) \right\|_{\D}\text{.}
\end{equation*}
\end{definition}

The bridge seminorm allows us to quantify how far apart two {\gQqcms s} are from the perspective of a given bridge.

\begin{definition}[{\cite[Definition 3.14]{Latremoliere13}}]\label{basic-reach-def}
Let Hypothesis (\ref{modular-bridge-hyp}) be given. The \emph{basic reach} $\bridgebasicreach{\gamma}$ of the modular bridge $\gamma$ is the reach of the basic bridge $\gamma_\flat$ with respect to $(\Lip_\A,\Lip_\B)$, i.e.
\begin{equation*}
\max\left\{ 
\begin{array}{lll}
\sup_{a\in\sa{\A},\Lip_\A(a)\leq 1} &\inf_{b\in\sa{\B},\Lip_\B(b)\leq 1} &\bridgenorm{\gamma}{a,b}
\\
\sup_{b\in\sa{\B},\Lip_\B(b)\leq 1} &\inf_{a\in\sa{\A},\Lip_\A(a)\leq 1} &\bridgenorm{\gamma}{a,b}
\end{array}
\right\} \text{.}
\end{equation*}
\end{definition}

We provide an alternative expression for the basic reach of a modular bridge. Indeed, the basic reach is where we actually take the Hausdorff distance between {\gQqcms s} in an appropriate sense. We shall use the following notation for the Hausdorff distance on a pseudo-metric space.

\begin{notation}\label{Hausdorff-distance-def}
Let $X$ be a set and $\mathrm{d}$ be a pseudo-metric on $X$. For any nonempty subset $A\subseteq X$ and for any $x\in X$, we set:
\begin{equation*}
\mathrm{d}(x,A) = \inf\{\mathrm{d}(x,y) : y \in A\} \text{.}
\end{equation*}

For any two nonempty sets $A, B \subseteq X$, we then define, following \cite{Hausdorff}:
\begin{equation*}
\Haus{\mathrm{d}}(A,B) = \sup\left\{ \mathrm{d}(x,B), \mathrm{d}(y,A) : x\in A, y\in B \right\}\text{.}
\end{equation*}
\end{notation}

We thus observe, using the notations of Hypothesis (\ref{modular-bridge-hyp}) and of Definition (\ref{basic-bridge-def}), that:
\begin{multline}\label{basic-reach-eq1}
\bridgebasicreach{\gamma} = \Haus{\bridgenorm{\gamma}{\cdot,\cdot}}\left( \left\{ (a,0) \in \A\oplus\B : a\in\sa{\A},\Lip_\A(a)\leq 1 \right\}, \right. \\
\left. \left\{ (0,b) : b\in\sa{\B},\Lip_\B(b) \leq 1 \right\} \right)\text{.}
\end{multline}

The motivation to use the bridge seminorm, i.e. to involve the pivot, in Equation (\ref{basic-reach-eq1}), in place of the norm $\|\cdot\|_\D$ of $\D$, is that the pivot allows us to ``cut-off'' elements and thus may be used as a noncommutative substitute for truncation. This fact is explained and illustrated in \cite{Latremoliere13c}. 

The cost of replacing the norm of $\D$ by the bridge seminorm in Equation (\ref{basic-reach-eq1}) is measured by the next quantity associated with a modular bridge.

\begin{definition}[{\cite[Definition 3.16]{Latremoliere13}}]\label{bridge-height-def}
Let Hypothesis (\ref{modular-bridge-hyp}) be given. The \emph{height} $\bridgeheight{\gamma}$ of the modular bridge $\gamma$ is the height of the basic bridge $\gamma_\flat$ with respect to $(\Lip_\A,\Lip_\B)$, i.e.:
\begin{multline*}
\max\left\{ \Haus{\Kantorovich{\Lip_\A}}\left(\StateSpace(\A), \pi_\A^\ast\left( \StateSpace(\D|x)  \right) \right), \Haus{\Kantorovich{\Lip_\B}}\left(\StateSpace(\B), \pi_\B^\ast\left( \StateSpace(\D|x)  \right) \right) \right\}\text{.}
\end{multline*}
\end{definition}

The height of a bridge involves computation in each of the domain and co-domain of the bridge, but not in between them. Its definition is what justifies that pivot must have nonempty $1$-level set.

We now turn to the new quantities which we define for modular bridges, which naturally relate to the module structure. The first of these numerical values, called the reach of a modular bridge, is derived from a new natural seminorm defined by a modular bridge. We continue to choose our terminology from the lexical field of bridges.

\begin{definition}\label{deck-seminorm-def}
Let Hypothesis (\ref{modular-bridge-hyp}) be given. The \emph{deck seminorm} $\decknorm{\gamma}{\cdot,\cdot}$ is the seminorm on $\module{M}\oplus\module{N}$ defined for all $\omega\in\module{M}$ and $\eta\in\module{N}$ by:
\begin{multline*}
\decknorm{\gamma}{\omega,\eta} = \max \left\{ \bridgenorm{\gamma}{\inner{\omega}{\omega_k}{\module{M}}, \inner{\eta}{\eta_k}{\module{N}}}, \right. \\
\left. \bridgenorm{\gamma}{\inner{\omega_k}{\omega}{\module{M}}, \inner{\eta_k}{\eta}{\module{N}}} : k\in J \right\} \text{.}
\end{multline*}
\end{definition}

We continue using the notations of Definition (\ref{deck-seminorm-def}). We emphasize that when working with $\decknorm{\gamma}{\cdot,\cdot}$, we only require the structure of vector space on $\module{M}\oplus\module{N}$. We also record the deck seminorm does not involve any explicit need to embed $\module{M}$ and $\module{N}$ is some left Hilbert module. We rely instead on the well understood idea behind noncommutative isometric embeddings of quantum metric spaces and avoid the need to introduce a similar, non-obvious notion for modules. 

Furthermore, the deck seminorm is defined with a symmetry in mind, which will prove useful in the notion of the inverse of a bridge defined at the end of this section.

The reach of a modular bridge requires the definition of two additional quantities besides the basic reach. The first quantity, the modular reach, regards the pairing of anchors and co-anchors. We underscore that, in the construction of the deck seminorm, we match anchors and co-anchors with the same index in the modular bridge. Therefore, when constructing of a modular bridge, we must make an astute choice, with the idea that each pair of anchor and co-anchor are expected to be ``close'' in a sense quantified, ultimately, by the modular reach, via the deck seminorm.

\begin{definition}\label{modular-reach-def}
Let Hypothesis (\ref{modular-bridge-hyp}) be given. The \emph{modular reach} $\bridgemodularreach{\gamma}$ is the nonnegative number:
\begin{equation*}
\max\left\{ \decknorm{\gamma}{\omega_j, \eta_j} : j\in J \right\} \text{.}
\end{equation*}
\end{definition}

The last quantity needed to define the reach of a modular bridge is the imprint. The modular bridge only involves anchors and co-anchors, and the cost of this choice, rather than taking some Hausdorff distance between unit balls for D-norms, is measured by the following quantity:

\begin{definition}\label{bridge-imprint-def}
Let Hypothesis (\ref{modular-bridge-hyp}) be given. The \emph{imprint} $\bridgeimprint{\gamma}$ of the modular bridge $\bridge{\gamma}$ is:
\begin{multline*}
\max\left\{ \Haus{\KantorovichMod{\CDN_{\module{M}}}}\left( \left\{ \omega_j : j \in J \right\}, \modlip{1}{\Omega_\A}\right), \Haus{\KantorovichMod{\CDN_{\module{N}}}}\left( \left\{ \eta_j : j \in J \right\}, \modlip{1}{\Omega_\B}\right)\right\}\text{.}
\end{multline*}
\end{definition}

We now define the reach of a modular bridge as a synthetic valued which adequately combines the basic reach, the modular reach, and the imprint. Our definition is immediately followed with a proposition which, we hope, will clarify the meaning of the reach of a modular bridge --- and which will prove crucial for our work in allowing for the definition of target sets for modular bridges, to come shortly.

\begin{definition}\label{bridge-reach-def}
Let Hypothesis (\ref{modular-bridge-hyp}) be given. The \emph{reach} $\bridgereach{\gamma}$ of the modular bridge $\gamma$ is the nonnegative value:
\begin{equation*}
\bridgereach{\gamma} = \max\left\{ \bridgebasicreach{\gamma}, \bridgemodularreach{\gamma} + \bridgeimprint{\gamma} \right\}\text{.}
\end{equation*}
\end{definition}

\begin{proposition}\label{bridge-reach-prop}
Let Hypothesis (\ref{modular-bridge-hyp}) be given. If $\omega\in\module{M}$ with $\CDN_{\module{M}}(\omega)\leq 1$ and if $j \in J$ is chosen so that $\KantorovichMod{\Omega_\A}(\omega,\omega_j) \leq \bridgeimprint{\gamma}$, then $\decknorm{\gamma}{\omega,\eta_j} \leq \bridgereach{\gamma}$. The result also holds if $\Omega_\A$ and $\Omega_\B$ are switched.

We then have:
\begin{equation}\label{bridge-reach-eq1}
\max\left\{ \begin{array}{lll}
\sup_{a\in\sa{\A},\Lip_\A(a)\leq 1} &\inf_{b\in\sa{\B},\Lip_\B(b)\leq 1} &\bridgenorm{\gamma}{a,b}
\\
\sup_{b\in\sa{\B},\Lip_\B(b)\leq 1} &\inf_{a\in\sa{\A},\Lip_\A(a)\leq 1} &\bridgenorm{\gamma}{a,b}\\
\sup_{\omega\in\module{M},\CDN_{\module{M}}(\omega)\leq 1} &\inf_{\eta\in\module{N},\CDN_{\module{N}}(\eta)\leq 1} &\decknorm{\gamma}{\omega,\eta}   \\
\sup_{\eta\in\module{N},\CDN_{\module{N}}(\eta)\leq 1} &\inf_{\omega\in\module{M},\CDN_{\module{M}}(\omega)\leq 1} &\decknorm{\gamma}{\omega,\eta} 
\end{array}
\right\} \leq \bridgereach{\gamma} \text{.}
\end{equation}
\end{proposition}

\begin{proof}
By Definition (\ref{basic-reach-def}) and Definition (\ref{bridge-reach-def}), it is sufficient to prove that:
\begin{equation*}
\max\left\{ \begin{array}{lll}
\sup_{\omega\in\module{M},\CDN_{\module{M}}(\omega)\leq 1} &\inf_{\eta\in\module{N},\CDN_{\module{N}}(\eta)\leq 1} &\decknorm{\gamma}{\omega,\eta}   \\
\sup_{\eta\in\module{N},\CDN_{\module{N}}(\eta)\leq 1} &\inf_{\omega\in\module{M},\CDN_{\module{M}}(\omega)\leq 1} &\decknorm{\gamma}{\omega,\eta} 
\end{array}
\right\} \leq \bridgereach{\gamma} \text{.}
\end{equation*}

Let $\omega\in\CDN_{\module{M}}$ with $\CDN_{\module{M}} (\omega) \leq 1$. By Definition (\ref{bridge-imprint-def}), there exists $j \in J$ such that:
\begin{equation*}
\KantorovichMod{\Omega_\A}(\omega,\omega_j) \leq \bridgeimprint{\gamma} \text{.}
\end{equation*}

Now, by Definition (\ref{modular-reach-def}), we have $\decknorm{\gamma}{\omega_j, \eta_j} \leq \bridgemodularreach{\gamma} \text{.}$ Thus, for any $k\in J$, we compute:
\begin{multline*}
\left\|\pi_\A\left( \inner{\omega}{\omega_k}{\module{M}} \right) x - x \pi_\B \left( \inner{\eta_j}{\eta_k}{\module{N}} \right)\right\|_{\D} \\
\begin{split}
&\leq \left\|\pi_\A\left(\inner{\omega}{\omega_k}{\module{M}} - \inner{\omega_j}{\omega_k}{\module{M}}\right)x\right\|_\D \\
&\quad + \left\| \pi_\A\left( \inner{\omega_j}{\omega_k}{\module{M}}\right) x - x \pi_\B\left( \inner{\eta_j}{\eta_k}{\module{N}}\right) \right\|_{\D} \\
&\leq \|x\|_{\D} \KantorovichMod{\Omega_\A}(\omega,\omega_j) + \decknorm{\gamma}{\omega_j,\eta_j}\\
&\leq \bridgeimprint{\gamma} + \bridgemodularreach{\gamma} \leq \bridgereach{\gamma} \text{.}
\end{split}
\end{multline*}

We now observe that since the involution of $\D$ is isometric and $k \in J$:
\begin{multline*}
\left\|\pi_\A\left(\inner{\omega_k}{\omega}{\module{M}}\right) x - x \pi_\B\left(\inner{\eta_k}{\eta_j}{\module{N}}\right)\right\|_\D \\ 
= \left\| x^\ast \pi_\A\left(\inner{\omega}{\omega_j}{\module{M}}\right) - \pi_\B\left(\inner{\eta_j}{\eta_k}{\module{N}}\right) x^\ast \right\|_\D\text{.}
\end{multline*}

Since $\|x^\ast\|_\D \leq 1$, a similar computation then proves that for all $k\in J$:
\begin{equation*}
\left\|\pi_\A\left(\inner{\omega_k}{\omega_j}{\module{M}}\right) x - x \pi_\B\left(\inner{\eta_k}{\eta_j}{\module{N}}\right)\right\|_\D \leq \bridgereach{\gamma} \text{.}
\end{equation*}

Thus, as desired $\decknorm{\gamma}{\omega,\eta_j} \leq \bridgereach{\gamma}\text{.}$ In particular, we have shown:
\begin{equation*}
\sup\nolimits_{\omega\in\modlip{1}{\Omega_\A}} \inf\nolimits_{\eta\in\modlip{1}{\Omega_\B}} \decknorm{\gamma}{\omega,\eta} \leq \bridgereach{\gamma} \text{.}
\end{equation*}

A similar computation shows that:
\begin{equation*}
\sup\nolimits_{\eta\in\modlip{1}{\Omega_\B}} \inf\nolimits_{\omega\in\modlip{1}{\Omega_\A}} \decknorm{\gamma}{\omega,\eta} \leq \bridgereach{\gamma} \text{.}
\end{equation*}

This concludes our proof.
\end{proof}

We now pause for a few remarks regarding our Definition (\ref{bridge-reach-def}) of a reach for the modular bridge. First, unlike in \cite{Latremoliere13}, we required in Definition (\ref{modular-bridge-def}) that pivots have norm at most one. The result in Proposition (\ref{bridge-reach-prop}) is where this additional assumption is needed.

Proposition (\ref{bridge-reach-prop}) suggests a competing candidate for the notion of a bridge, given by the left-hand side of Inequality (\ref{bridge-reach-eq1}). This alternate candidate is given as the maximum of Expression (\ref{basic-reach-eq1}) and the following similar expression for modules:
\begin{multline}\label{modular-reach-eq1}
\Haus{\decknorm{\gamma}{\cdot,\cdot}}\left( \left\{ (\omega,0) \in \module{M}\oplus\module{N} : \CDN_{\module{M}}(\omega) \leq 1 \right\}, \right. \\
\left. \left\{ (0,\eta) \in \module{M}\oplus\module{N} : \CDN_{\module{N}}(\eta) \leq 1 \right\} \right) \text{.}
\end{multline}

This formulation would more closely resemble the definition of the basic reach. Our preference for Definition (\ref{bridge-reach-def}) rather than the maximum of Expressions (\ref{basic-reach-eq1}) and (\ref{modular-reach-eq1}) is at the core of our idea for the construction of the modular propinquity. Indeed, Definition (\ref{basic-reach-def}) employs the match between anchors and co-anchors. This pairing is essential, because it also appears in the Definition (\ref{deck-seminorm-def}) of the deck seminorm and actually, it is the approach we use to construct a seminorm from a couple of \emph{sesquilinear} maps.

Indeed, we also could have introduced anchors and co-anchors in the construction of the quantum propinquity. Namely, an ``anchored'' bridge from $(\A,\Lip_\A)$ to $(\B,\Lip_\B)$ could be of the form $\gamma = (\D,x,\pi_\A,\pi_\B,(a_j)_{j\in J}, (b_j)_{j\in J})$ with $a_j\in\alglip{1}{\Lip_\A}$ and $b_j\in\alglip{1}{\Lip_\B}$ for all $j \in J$. We then could define the ``anchored'' reach as we just did for modular bridge, i.e. as the maximum of $\max\{\bridgenorm{\gamma}{a_j,b_j} : j\in J\}$ and of a kind of imprint, i.e. $\max\{\Haus{\|\cdot\|_\A}(\{a_j:j\in J\}, \alglip{1}{\Lip_\A}), \Haus{\|\cdot\|_\A}(\{a_j:j\in J\}, \alglip{1}{\Lip_\A}) \}$. The length of an anchored bridge would then be the maximum of its anchored reach and its height, defined in the usual manner.

Yet such a definition of a bridge reach would not change our construction of the quantum propinquity. Indeed, Proposition (\ref{bridge-reach-prop}) could be adapted to prove that the reach of the bridge $(\D,x,\pi_\A,\pi_\B)$ is lesser or equal than the anchored reach of $\gamma$. It is also easy to check that given a bridge $(\D,x,\pi_\A,\pi_\B)$ in the sense of \cite[Definition 3.6]{Latremoliere13}, there always is a mean to construct an anchor bridge with the same length. We refer briefly to \cite{Latremoliere13} for various notions which we will extend in a moment to modular bridges, and the reader may skip the following few details as they are just a side observation. Using the notion of target sets introduced in \cite[Definition 5.1]{Latremoliere13}, we can, for all $a\in \alglip{1}{\Lip_\A}$, choose some $b_a \in \targetsetbridge{\gamma}{a}{1}$, and similarly by symmetry, for all $b\in\alglip{1}{\Lip_\B}$, choose $a_b \in \targetsetbridge{\gamma^{-1}}{b}{1}$. With these notations, if $J = \alglip{1}{\Lip_\A}\coprod \alglip{1}{\Lip_\B}$, and if we set $a_a = a$ and $b_b = b$ for all $a\in \alglip{1}{\Lip_\A}$ and $b\in\alglip{1}{\Lip_\B}$, then:
\begin{equation*}
\left(\D,x,\pi_\A,\pi_\B,(a_j)_{j\in J}, (b_j)_{j \in J}\right)
\end{equation*}
is an anchored bridge with the same length than the bridge $(\D,x,\pi_\A,\pi_\B)$. Thus there is no need for anchors and co-anchors in the construction of the quantum propinquity.

If such is the case, then why did we introduce anchors in our current work? The reason lies with the fact that the bridge seminorm is indeed a seminorm because the maps $\pi_\A$ and $\pi_\B$ are linear. However the inner products $\inner{\cdot}{\cdot}{\module{M}}$ and $\inner{\cdot}{\cdot}{\module{N}}$ are sesquilinear, and thus, to construct our deck seminorm as indeed a seminorm, we discovered the idea of employing pairs of anchors-co-anchors. While this idea would not change anything for the quantum propinquity, it becomes essential for the modular propinquity.

The length of a modular bridge is the synthetic numerical value which summarizes all the information contained in the basic reach, modular reach, height and imprint of the modular bridge, and from which the modular propinquity will be computed.

\begin{definition}\label{bridge-length-def}
Let Hypothesis (\ref{modular-bridge-hyp}) be given. The length $\bridgelength{\gamma}$ of the modular bridge $\bridge{\gamma}$ is the maximum of its reach, its height and its imprint:
\begin{equation*}
\bridgelength{\gamma} = \max\left\{ \bridgeheight{\gamma}, \bridgereach{\gamma} \right\}\text{.}
\end{equation*}
\end{definition}

We note that a modular bridge always has a finite length.

\begin{lemma}\label{finite-length-lemma}
If $\gamma$ is a modular bridge then $\bridgelength{\gamma} < \infty$.
\end{lemma}

\begin{proof}
The imprint and the height of a modular bridge are both defined as the Hausdorff distance between two compact sets and thus are finite.

Now, if $\omega\in\modlip{1}{\Omega_\A}$ and $\eta\in\modlip{1}{\Omega_\B}$ then since $\|x\|_\D\leq 1$, we have:
\begin{equation*}
\decknorm{\gamma}{\omega,\eta} = \max_{j \in J} \left\| \pi_\A\left(\inner{\omega}{\omega_k}{\module{M}}\right) x - x \pi_\A\left(\inner{\omega}{\omega_k}{\module{M}}\right) \right\|_\D \leq 2 \text{.}
\end{equation*}

Thus $\bridgemodularreach{\gamma} \leq 2$. The reach of $\gamma$ is thus the maximum of the (finite) basic bridge reach and the sum of the (finite) imprint and the (finite) modular reach. Thus $\bridgereach{\gamma} < \infty$ and thus our proposition is proven.
\end{proof}

\bigskip

Modular bridges are a type of morphism between {\gQVB s} --- though we shall address the question of composition for modular bridges in our next section with the introduction of modular treks. In the rest of this section, we formalize the idea that bridges posses some properties akin to some form of multi-valued morphism. These properties are the essential reason behind the fact that, if the modular propinquity is null between two {\gQVB s}, then they are fully quantum isometric.

A modular bridge from $\Omega_\A$ to $\Omega_\B$, with $\Omega_\A$ and $\Omega_\B$ two {\gQVB s}, defines maps  from the domain of the L-seminorm of $\Omega_\A$ to the power set of domain of the L-seminorm of $\Omega_\B$.

\begin{definition}\label{targetset-def}
Let Hypothesis (\ref{modular-bridge-hyp}) be given. For any $a\in\dom{\Lip_\A}$ and $l \geq \Lip_\A(a)$, we define the \emph{$l$-target set} $\targetsetbridge{\gamma}{a}{l}$ of $a$ for $\bridge{\gamma}$ as:
\begin{equation*}
\targetsetbridge{\gamma}{a}{l} = \left\{ b\in\dom{\Lip_\B} \middle\vert \begin{array}{l}
\Lip_\B(b) \leq l\\
\bridgenorm{\gamma}{a,b} \leq l \bridgereach{\gamma_\flat}\text{.}
\end{array}\right\}\text{.}
\end{equation*}
\end{definition}

Definition (\ref{targetset-def}) ensures that $\targetsetbridge{\gamma}{a}{l} = \targetsetbridge{\gamma_\flat}{a}{l}$, where the right hand side is defined in \cite[Definition 5.1]{Latremoliere13}. It actually would not matter in our subsequent work if instead, we had used $\bridgereach{\gamma}$ in place of $\bridgereach{\gamma_\flat}$ in Definition (\ref{targetset-def}). On the other hand, thanks to our choice, we can invoke our work in \cite{Latremoliere13} to immediately conclude:

\begin{proposition}\label{base-targetset-prop}
Let Hypothesis (\ref{modular-bridge-hyp}) be given. For all $a,a' \in \dom{\Lip_\A}$ and $l \geq \max\{\Lip_\A(a),\Lip_\A(a')\}$, if $b \in \targetsetbridge{\gamma}{a}{l}$ and $b'\in\targetsetbridge{\gamma}{a'}{l}$ then:
\begin{enumerate}
\item $\targetsetbridge{\gamma}{a}{l}$ is a nonempty compact subset of $\sa{\B}$,
\item $\|b\|_\B \leq \|a\|_\A + 2 l \bridgelength{\gamma}$,
\item for all $t \in \R$ we have $b + t b' \in \targetsetbridge{\gamma}{a + t a'}{(1+|t|)l}$,
\item $\|b - b'\|_\B \leq \|a-a'\|_\A + 4 l \bridgelength{\gamma}$,
\item We have:
\begin{equation*}
\Jordan{b}{b'} \in \targetsetbridge{\gamma}{\Jordan{a}{a'}}{F(\|a\|_\A + 2 l \bridgelength{\gamma},\|a'\|_\A + 2 l \bridgelength{\gamma}, l, l) }
\end{equation*}
and
\begin{equation*}
\Lie{b}{b'}\in\targetsetbridge{\gamma}{\Lie{b}{b'}}{F(\|a\|_\A + 2 l \bridgelength{\gamma},\|a'\|_\A + 2 l \bridgelength{\gamma}, l, l)}\text{.}
\end{equation*}
\end{enumerate}
In particular, for all $a\in\dom{\Lip_\A}$ and $l \geq \Lip_\A(a)$, we have:
\begin{equation*}
\diam{\targetsetbridge{\gamma}{a}{l}}{} \leq 4 l \bridgelength{\gamma} \text{.}
\end{equation*}
\end{proposition}

\begin{proof}
Assertion (1) is \cite[Lemma 5.2]{Latremoliere13} and since it is a closed subset of the norm compact $\alglip{l}{\Lip_\B}$. Assertion (2) follows from \cite[Proposition 5.3]{Latremoliere13}. Assertion (3) follows from \cite[Proposition 5.4]{Latremoliere13}. Assertion (4) follows from Assertion (2) and Assertion (3). Assertion (5) is established by noting:
\begin{equation*}
\Lip_\B(\Jordan{b}{b'}) \leq F(\|b\|_\B,\|b'\|_\B, l, l) \leq F(\|a\|_\A + 2\bridgelength{\gamma}, \|a'\|_\A + 2 l \bridgelength{\gamma}, l, l) \text{,}
\end{equation*}
and similarly for the Lie product. Setting $a=a'$ gives us the given estimate on the diameter of $\targetsetbridge{\gamma}{a}{l}$.
\end{proof}

We now define the target set for elements in the domain of a D-norm.

\begin{definition}\label{modular-targetset-def}
Let Hypothesis (\ref{modular-bridge-hyp}) be given. For any $\omega \in \module{M}$ and $l \geq \CDN_{\module{M}}(\omega)$, we define the \emph{$l$-modular target set} of $\omega$ for $\bridge{\gamma}$ as:
\begin{equation*}
\targetsetbridge{\gamma}{\omega}{l} = \left\{ \eta \in \module{N} \middle\vert \begin{array}{l}
\CDN_{\module{N}}(\eta)\leq l\\
\decknorm{\gamma}{\omega,\eta} \leq l \bridgereach{\gamma}
\end{array} \right\}\text{.}
\end{equation*}
\end{definition}

We begin by observing that modular target sets are compact and non-empty.

\begin{proposition}\label{compact-targetset-prop}
Let Hypothesis (\ref{modular-bridge-hyp}) be given. For any $\omega\in\dom{\CDN_{\module{M}}}$ and $l \geq \CDN_{\module{M}}(\omega)$, the set $\targetsetbridge{\gamma}{\omega}{l}$ is a nonempty compact for $\|\cdot\|_{\module{N}}$ (or equivalently for $\KantorovichMod{\Omega_\B}$).
\end{proposition}

\begin{proof}
By Proposition (\ref{bridge-reach-prop}), for all $\omega\in\modlip{1}{\Omega_\A}$ there exists $\eta\in\modlip{1}{\Omega_\B}$ such that $\decknorm{\gamma}{\omega,\eta} \leq \bridgereach{\gamma}$. Thus, if $\CDN_{\module{M}}(\omega) \leq l$ for some $\omega\in\module{M}$, it follows from homogeneity that there exists $\eta\in \modlip{l}{\Omega_\B}$ such that $\decknorm{\gamma}{\omega,\eta} \leq l \bridgereach{\gamma}$ since $\decknorm{\gamma}{\cdot}$ is a seminorm. Therefore, $\targetsetbridge{\gamma}{\omega}{l} \not=\emptyset$.

By construction $\targetsetbridge{\gamma}{\omega}{l}$ is a subset of the compact set $\modlip{l}{\Omega_\B}$ (for $\|\cdot\|_{\module{N}}$ or for $\KantorovichMod{\Omega_\B}$, as both give the same topology on $\modlip{l}{\Omega_\B}$). Thus it is sufficient to prove that $\targetsetbridge{\gamma}{\omega}{l}$ is closed.

Let $(\eta_n)_{n\in\N}$ be a sequence in $\targetsetbridge{\gamma}{\omega}{l}$, converging to some $\eta$ for $\|\cdot\|_{\module{N}}$. Since $\CDN_{\module{N}}$ is lower semi-continuous with respect to $\|\cdot\|_{\module{N}}$, we have $\CDN_{\module{N}}(\eta)\leq l$. 

Moreover, by continuity, $\decknorm{\gamma}{\omega,\eta} \leq l \bridgereach{\gamma}$ since $\decknorm{\gamma}{\omega,\eta} \leq l \bridgereach{\gamma}$. This proves that $\eta\in\targetsetbridge{\gamma}{\omega}{l}$ as desired.
\end{proof}

The fundamental property of modular target sets for a modular bridge $\bridge{\gamma}$ is that their diameter in the modular {\MongeKant} is controlled by the length of $\bridge{\gamma}$ --- and, in contrast with target sets for basic bridges, not their diameter in the $C^\ast$-Hilbert norm. We begin with a well-known lemma included for convenience.

\begin{lemma}\label{norm-bound-lemma}
If $\A$ is a C*-algebra, $a\in\A$, and there exists $M \geq 0$ such that for all $\varphi\in\StateSpace(\A)$ we have:
\begin{equation*}
\max\left\{|\varphi(\Re(a))|,|\varphi(\Im(a))|\right\} \leq M\text{,}
\end{equation*}
then:
\begin{equation*}
\|a\|_\A \leq \sqrt{2} M \text{.}
\end{equation*}
\end{lemma}

\begin{proof}
Let $b\in\sa{\A}$ then the functional calculus implies that $\|b\|_\A = \sup\{|\varphi(b)| : \varphi\in\StateSpace(\A) \}$. Thus for all $a\in\A$, we compute:
\begin{equation*}
\begin{split}
\|a\|_\A^2 &= \|a a^\ast\|_\A \\
&= \|(\Re(a) + i\Im(a)) (\Re(a) + i\Im(a))^\ast\|_\A \\
&= \|\Re(a)^2 + \Im(a)^2\|_\A \\
&\leq \|\Re(a)\|^2 + \|\Im(a)\|^2 \\
&= \left(\sup\{ |\varphi(\Re(a)) : \varphi\in\StateSpace(\A)\}\right)^2 + \left(\sup\{ |\varphi(\Im(a)) : \varphi\in\StateSpace(\A)\}\right)^2 \\
&\leq 2 M^2 \text{.}
\end{split}
\end{equation*}
This concludes our lemma.
\end{proof}

\begin{proposition}\label{diameter-prop}
Let Hypothesis (\ref{modular-bridge-hyp}) be given. 

If $\omega, \omega' \in\module{M}$, $l \geq \max\left\{ \CDN_{\module{M}}(\omega), \CDN_{\module{M}}(\omega') \right\}$, $\eta\in\targetsetbridge{\gamma}{\omega}{l}$ and  $\eta'\in\targetsetbridge{\gamma}{\omega'}{l}$, then:
\begin{equation*}
\KantorovichMod{\CDN_{\module{N}}}(\eta,\eta') \leq  \sqrt{2} \left( \KantorovichMod{\CDN_{\module{M}}}(\omega,\omega') + (4 l + H(2l, 1)) \bridgelength{\gamma}\right) \text{.}
\end{equation*}

In particular:
\begin{equation*}
\diam{\targetsetbridge{\gamma}{\omega}{l}}{\KantorovichMod{\Lip_{\module{N}}}} \leq \sqrt{2} (4 l + H(2l, 1)) \bridgelength{\gamma}\text{.}
\end{equation*}
\end{proposition}

\begin{proof}
Let $\theta = \eta - \eta'$ and $\zeta = \omega - \omega'$. Note that:
\begin{equation*}
\max\left\{ \CDN_{\module{N}}(\theta), \CDN_{\module{M}}(\zeta)  \right\} \leq 2 l \text{.}
\end{equation*}

Let $\varphi \in \StateSpace(\B)$ and let $\nu \in \modlip{1}{\Omega_\A}$. 

There exists $j \in J$ such that $\KantorovichMod{\CDN_{\module{N}}}(\nu,\eta_j)\leq \bridgeimprint{\gamma}$ by Definition (\ref{bridge-imprint-def}). 

By Definition (\ref{bridge-height-def}), there exists $\psi \in \StateSpace(\D|x)$ with $\Kantorovich{\Lip_\B}(\varphi,\psi\circ\pi_\B) \leq \bridgeheight{\gamma}$.

Note that $\CDN_{\module{N}}(\eta_j) \leq 1$, and therefore, using the inner quasi-Leibniz inequality, we have:
\begin{equation*}
\max\left\{ \Lip_\B\left(\Re\inner{\theta}{\eta_j}{\module{N}}\right), \Lip\left(\Im\inner{\theta}{\eta_j}{\module{N}}\right) \right\} \leq H(2l,1) \text{.}
\end{equation*}
 We also note that since $\psi$ is a state, we have:
\begin{equation*}
|\psi(\Re(d))| = |\Re(\psi(d))| \leq |\psi(d)| \text{ and, similarly: }|\psi(\Im(d))|\leq |\psi(d)|
\end{equation*}
for all $d\in\D$.

Now, letting $m = H(2l,1)$:
\begin{equation*}
\begin{split}
\left|\varphi\left(\Re\inner{\theta}{\nu}{\module{N}}\right)\right| &\leq 2 l \bridgelength{\gamma} + \left|\varphi\left(\Re\inner{\theta}{\eta_j}{\module{N}}\right)\right| \text{ by Def. (\ref{bridge-imprint-def}),}\\
&\leq (2 l + m) \bridgelength{\gamma} + \left|\psi\circ\pi_\B \left(\Re\inner{\theta}{\eta_j}{\module{N}} \right)\right| \text{ by choice of $\psi$,}\\
&\leq (2 l + m) \bridgelength{\gamma} + \left|\psi\circ\pi_\B \left(\inner{\theta}{\eta_j}{\module{N}} \right)\right| \\
&\leq (2 l + m) \bridgelength{\gamma} + \left|\psi\left(\pi_\B(\inner{\theta}{\eta_j}{\module{N}} )x\right)\right| \text{ by Def. (\ref{level-set-def}),}\\
&\leq (4 l + m) \bridgelength{\gamma} + \left|\psi\left(x \pi_\A(\inner{\zeta}{\omega_j}{\module{M}} )\right)\right| \text{ by Prop. (\ref{bridge-reach-prop}),}\\
&\leq (4 l + m) \bridgelength{\gamma} + \left|\psi\circ\pi_\A\left(\inner{\zeta}{\omega_j}{\module{M}}\right)\right| \text{ by Def. (\ref{level-set-def}),}\\
&\leq (4 l + m)\bridgelength{\gamma} + \left\|\inner{\zeta}{\omega_j}{\module{M}} \right\|_\A \\
&\leq (4 l + m) \bridgelength{\gamma} + \left\|\inner{\omega}{\omega_j}{\module{M}} - \inner{\omega'}{\omega_j}{\module{M}} \right\|_\A \\
&\leq (4 l + m) \bridgelength{\gamma} + \KantorovichMod{\Lip_{\omega}}(\omega,\omega') \text{.}
\end{split}
\end{equation*}

The same computation holds for $\Re$ replaced with $\Im$, and thus we record:
\begin{equation*}
\left|\varphi\left(\Im\inner{\theta}{\nu}{\module{N}}\right)\right| \leq (4 l + m)\bridgelength{\gamma} + \KantorovichMod{\Lip_{\omega}}(\omega,\omega') \text{,}
\end{equation*}
and therefore by Lemma (\ref{norm-bound-lemma}), we conclude:
\begin{equation*}
\left\|\inner{\theta}{\nu}{\module{N}}\right\|_\B \leq \sqrt{2}\left( (4 l + m)\bridgelength{\gamma} + \KantorovichMod{\CDN}(\omega,\omega') \right)\text{.}
\end{equation*}

Thus:
\begin{equation*}
\begin{split}
\KantorovichMod{\CDN_{\module{N}}}(\eta,\eta') &= \sup\left\{\left\|\inner{\theta}{\nu}{\module{N}} \right\|_\B : \nu\in\modlip{1}{\Omega_\B} \right\} \\
&\leq \sqrt{2}\left( (4 l + m)\bridgelength{\gamma} + \KantorovichMod{\CDN}(\omega,\omega') \right)\text{.}
\end{split}
\end{equation*}
Our proof is thus complete. The assertion on the diameter is obtained simply by letting $\omega=\omega'$.
\end{proof}

Our first relation between target sets on modules and the module algebraic structure concerns linearity, as expressed in the following proposition. 

\begin{proposition}\label{linear-prop}
Let Hypothesis (\ref{modular-bridge-hyp}) be given. 
If:
\begin{enumerate}
\item $\omega, \omega' \in \module{M}$,
\item $l \geq \CDN_{\module{M}}(\omega)$ and $l' \geq \CDN_{\module{M}}(\omega')$, 
\item $\eta \in \targetsetbridge{\gamma}{l}{\omega}$ and $\eta'\in\targetsetbridge{\gamma}{\omega'}{l'}$,
\item $t\in \R$,
\end{enumerate} 
then:
\begin{equation*}
\eta + t\eta' \in \targetsetbridge{\gamma}{\omega + t\omega'}{l + |t| l'} \text{.}
\end{equation*}
\end{proposition}

\begin{proof}
Since $\CDN_{\module{N}}(\eta)\leq l$ and $\CDN_{\module{N}}(\eta')\leq l'$, we have $\CDN_{\module{N}}(\eta + t \eta') \leq l + |t|l'$. 

On the other hand, we note that since $\decknorm{\gamma}{\cdot}$ is a seminorm $\module{M}\oplus\module{N}$, we conclude that:
\begin{equation*}
\begin{split}
\decknorm{\gamma}{\omega+t\omega',\eta + t\eta'} &\leq \decknorm{\gamma}{\omega,\eta} + |t|\decknorm{\gamma}{\omega',\eta'}\\
&\leq (l + |t|l') \bridgereach{\gamma} \text{.}
\end{split}
\end{equation*}
This completes our proof.
\end{proof}

We now prove that target sets also behave predictably with respect to the left action on the module. This proposition is where the modular quasi-Leibniz inequality plays its role.

\begin{proposition}\label{product-prop}
Let Hypothesis (\ref{modular-bridge-hyp}) be given. Let $a\in\dom{\Lip_\A}$, $\omega \in \dom{\CDN_{\module{M}}}$, and $l \geq \CDN_{\module{M}}(\omega)$ and $l' \geq \Lip_\A(a)$. Let $b \in \targetsetbridge{\gamma}{a}{l'}$ and $\eta\in\targetsetbridge{\gamma}{\omega}{l}$. Then:
\begin{equation*}
b\eta \in \targetsetbridge{\gamma}{a\omega}{G(\|a\|_\A + 2\bridgelength{\gamma}, l ,l')} \text{.}
\end{equation*}
\end{proposition}

\begin{proof}
We begin with the observation that:
\begin{equation*}
\begin{split}
\CDN_{\module{N}}(b\eta) &\leq G\left(\|b\|_{\B}, \Lip_\B(b), \CDN_{\module{M}}(\eta) \right) \\
&\leq G(\|a\|_\A + 2 l \bridgelength{\gamma}, l', l) \text{,}
\end{split}
\end{equation*}
using Proposition (\ref{base-targetset-prop}).

We also note that for any $j\in J$:
\begin{multline*}
\left\|\pi_\A\left(\inner{a\omega}{\omega_j}{\module{M}}\right) x - x \pi_\B\left(\inner{b\eta}{\eta_j}{\module{N}}\right)\right\|_{\D} \\
\begin{aligned}
&= \left\|\pi_\A(a)\pi_\A\left(\inner{\omega}{\omega_j}{\module{M}}\right) x - x \pi_\B(b)\pi_\B\left(\inner{\eta}{\eta_j}{\module{N}}\right)\right\|_{\D} \\
&\leq \left\|\pi_\A(a) \pi_\A\left(\inner{\omega}{\omega_j}{\module{M}}\right) x - \pi_\A(a) x \pi_\B\left(\inner{\eta}{\eta_j}{\module{M}}\right)\right\|_{\D} \\
&\quad + \left\|\pi_\A(a) x \pi_\B\left(\inner{\eta}{\eta_j}{\module{M}}\right) - x \pi_\B(b) \pi_\B\left(\inner{\eta}{\eta_j}{\module{N}}\right) \right\|_{\D}\\
&\leq \|a\|_{\A}\left\|\pi_\A\left(\inner{\omega}{\omega_j}{\module{M}}\right)x - x\pi_\B\left(\inner{\eta}{\eta_j}{\module{M}}\right)\right\|_{\D} \\
&\quad + \|\pi_\A(a)x - x \pi_\B(b)\|_{\D} \left\|\pi_\B\left(\inner{\eta}{\eta_j}{\module{N}}\right)\right\|_{\D}\\
&\leq \|a\|_\A \decknorm{\gamma}{\omega,\eta} + \bridgenorm{\gamma}{a,b} \|\eta\|_{\module{N}}\\
&\leq \|a\|_\A l \bridgereach{\gamma} + l' \bridgereach{\gamma} \CDN_{\module{N}}(\eta) \\
&\leq \bridgelength{\gamma}\left(\|a\|_\A l + l' l \right)\\
&\leq \bridgelength{\gamma} G(\|a\|_\A, l, l') \text{ by Def. (\ref{admissible-triple-def}), }\\
&\leq \bridgelength{\gamma}{G(\|a\|_\A + 2 l \bridgelength{\gamma}, l, l')} \text{.}
\end{aligned}
\end{multline*}

A similar computation proves that:
\begin{multline*}
\left\|\pi_\B\left(\inner{\omega_j}{a\omega}{\module{M}}\right) x - x \pi_\A\left(\inner{\eta_j}{b\eta}{\module{N}}\right)\right\|_{\D}\\
\begin{split}
&= \left\|\pi_\B\left(\inner{b\eta}{\eta_j}{\module{N}}\right) x^\ast - x^\ast \pi_\A\left(\inner{a\omega}{\omega_j}{\module{M}}\right)\right\|_{\D} \\
&\leq \bridgelength{\gamma}{G(\|a\|_\A + 2 l \bridgelength{\gamma}, l, l')} \text{.}
\end{split}
\end{multline*}
Therefore, $\decknorm{\gamma}{a\omega,b\eta} \leq \bridgelength{\gamma}{G(\|a\|_\A + 2 l \bridgelength{\gamma}, l, l')}$ since $j \in J$ is arbitrary.

Thus $b\eta \in \targetsetbridge{\gamma}{\omega}{G(\|a\|_\A + 2 l \bridgelength{\gamma}, l, l')}$.
\end{proof}

We now relate modular bridges and the inner products on modules, which illustrate the role of the inner quasi-Leibniz inner inequality.

\begin{proposition}\label{inner-morphism-prop}
Let Hypothesis (\ref{modular-bridge-hyp}) be given. Let $\omega \in \module{M}$ and $l \geq\CDN_{\module{M}}(\omega)$. If $\eta \in \targetsetbridge{\gamma}{\omega}{l}$ and $b\in\targetsetbridge{\gamma}{\inner{\omega}{\omega}{\module{M}}}{H(l,l)}$ then:
\begin{equation*}
\left\| b - \inner{\eta}{\eta}{\module{N}} \right\|_\B \leq ( 8 l \sqrt{2} + H( 2l,2l ) + 2 H( l,l ) + 2\sqrt{2} H( 2l, 1) )\bridgelength{\gamma}\text{.}
\end{equation*}
\end{proposition}

\begin{proof}
If $l = 0$ then $\omega = 0$, $\eta = 0$ and $b = 0$ thus the proposition is trivial. Let us assume $l > 0$.

Let $\omega\in\dom{\CDN_{\module{M}}}$ and $l \geq \CDN_{\module{M}}(\omega)$. Let $\eta \in \targetsetbridge{\Gamma}{\omega}{l}$. We note that:
\begin{equation*}
\max\left\{ \Lip_\A\left(\inner{\omega}{\omega}{\Omega_\A}\right), \Lip_\B\left(\inner{\eta}{\eta}{\Omega_\B}\right) \right\} \leq H( l,l )\text{,}
\end{equation*}
noting $\inner{\omega}{\omega}{\Omega_\A}$ and $\inner{\eta}{\eta}{\Omega_\B}$ are self-adjoint.

Let $b \in \targetsetbridge{\gamma}{\inner{\omega}{\omega}{\module{M}}}{H(l,l)}$.

By Definition (\ref{bridge-imprint-def}) of $\bridgeimprint{\gamma}$, there exists $j \in J$ such that:
\begin{equation*}
\KantorovichMod{\Omega_\A}(\omega, l \omega_j) \leq l \bridgeimprint{\gamma}\text{.}
\end{equation*}
It follows that $\inner{\omega}{\omega - l\omega_j}{\module{M}} = l\inner{l^{-1} \omega}{\omega - l\omega_j}{\module{M}} \leq l \KantorovichMod{\Omega_\A}(\omega, l\omega_j) \leq l^2 \bridgeimprint{\gamma}$.

Moreover, by Proposition (\ref{bridge-reach-prop}), we have:
\begin{equation*}
\decknorm{\gamma}{\omega, l\eta_j} \leq l \bridgereach{\gamma} \leq l\bridgelength{\gamma} \text{.}
\end{equation*}

We then have, since $\|x\|_\D \leq 1$:
\begin{multline*}
\left\|\pi_\A\left(\inner{\omega}{\omega}{\module{M}}\right) x-x \pi_\B\left( \inner{l \eta_j}{l \eta_j}{\module{N}}\right) \right\|_\D\\
\begin{split}
&\leq l^2 \bridgeimprint{\gamma} + \left\|\pi_\A\left(\inner{\omega}{l \omega_j}{\module{M}}\right) x - x \pi_\B\left(\inner{l \eta_j}{l \eta_j}{\module{N}}\right) \right\|_\D \\
&\leq l^2 \bridgeimprint{\gamma} + l \left\|\pi_\A\left(\inner{\omega}{\omega_j}{\module{M}}\right) x - x \pi_\B\left(\inner{l \eta_j}{\eta_j}{\module{N}}\right) \right\|_\D \\
&\leq l^2 \bridgelength{\gamma} + l \decknorm{\gamma}{\omega, l\eta_j} \\
&\leq 2l^2 \bridgelength{\gamma} \text{.}
\end{split}
\end{multline*}

Now, since $l\eta_j \in \targetsetbridge{\gamma}{\omega}{l}$ (again Proposition (\ref{bridge-reach-prop})), we have by Proposition (\ref{diameter-prop}):
\begin{equation*}
\KantorovichMod{\Omega_\B}(\eta,l \eta_j) \leq \sqrt{2}\left(4 l + H( 2l, 1) \right)\bridgelength{\gamma}\text{.}
\end{equation*}

Let $\varphi \in \StateSpace(\B)$. By Definition (\ref{bridge-height-def}), there exists $\psi \in \StateSpace(\D)$ such that $\KantorovichMod{\Omega_\B}(\varphi,\psi\circ\pi_\B) \leq \bridgeheight{\gamma}$. We then have:

\begin{multline*}
|\varphi(b - \inner{\eta}{\eta}{\module{M}})| \\
\begin{split}
&\leq H( 2l,2l )\bridgeheight{\gamma} + \left|\psi\circ\pi_\B(b - \inner{\eta}{\eta}{\module{N}})\right| \\
&\leq H( 2l,2l )\bridgeheight{\gamma} + \left|\psi\circ\pi_\B(b - \inner{l \eta_j}{l\eta_j}{\module{N}})\right| + \left|\psi\circ\pi_\B(\inner{l \eta_j}{l \eta_j}{\module{N}} - \inner{\eta}{\eta}{\module{N}})\right|\\
&\leq H( 2l, 2l )\bridgeheight{\gamma} + \left|\psi\left(x \pi_\B\left(b - \inner{l \eta_j}{l\eta_j}{\module{M}}\right)\right)\right| + \left\|\inner{l \eta_j}{l \eta_j}{\module{N}} - \inner{\eta}{\eta}{\module{N}}\right\|_\B \\
&\leq H( 2l,2l )\bridgeheight{\gamma} + \left|\psi\left(x \pi_\B\left(b - \inner{l \eta_j}{l\eta_j}{\module{M}}\right)\right)\right| + \\
&\quad + \left\|\inner{\eta}{\eta-l\eta_j}{\module{N}} + \inner{\eta- l\eta_j}{l \eta_j}{\module{N}} \right\|_\B  \\
&\leq H( 2l,2l )\bridgeheight{\gamma} + \left|\psi\left(x \pi_\B\left(b - \inner{l \eta_j}{l\eta_j}{\module{M}}\right)\right)\right| + \\
&\quad + l\KantorovichMod{\Omega_\B}(\eta,l\eta_j) + l\KantorovichMod{\Omega_\B}(\eta ,l\eta_j)  \\
&\leq H( 2l,2l )\bridgeheight{\gamma} + \left|\psi\left(x \pi_\B\left(b - \inner{l \eta_j}{l\eta_j}{\module{M}}\right)\right) \right. \\
&\quad \left. - \psi(\pi_\A\left(\inner{\omega}{\omega}{\module{M}}\right)x) + \psi(\pi_\A\left(\inner{\omega}{\omega}{\module{M}}\right)x)\right| \\
&\quad + 2 l \sqrt{2}  \left(4 l + H( 2l, 1) \right)\bridgelength{\gamma} \\
&\leq H( 2l,2l )\bridgeheight{\gamma} + \left|\psi\left(x \pi_\B(b) - \pi_\A\left(\inner{\omega}{\omega}{\module{M}}\right)x\right)\right| \\
&\quad + \left|\psi\left(x \pi_\B\left(\inner{l \eta_j}{l\eta_j}{\module{M}}\right)\right) - \psi\left(\pi_\A\left(\inner{\omega}{\omega}{\module{M}}\right)x\right)\right| \\
&\quad   + 2 l \sqrt{2}  \left(4 l + H( 2l, 1) \right)\bridgelength{\gamma} \\
&\leq H( 2l,2l )\bridgelength{\gamma} + H(l,l) \bridgelength{\gamma} + H (l,l )\bridgelength{\gamma} + 2 \sqrt{2} l \left(4 l + H( 2l, 1) \right)\bridgelength{\gamma} \\
&\leq ( 8 l \sqrt{2} + H( 2l,2l ) + 2 H( l,l ) + 2\sqrt{2} H( 2l, 1) )\bridgelength{\gamma} \text{.}
\end{split}
\end{multline*}

This concludes our proposition since $b - \inner{\eta}{\eta}{\module{N}}$ is self-adjoint in $\B$.
\end{proof}

We now check that modular bridges are essentially symmetric objects. We shall avoid the term inverse as we shall see that for modular bridges, in contrast to bridges, the following notion is not quite an inverse in the sense of morphisms.

\begin{definition}\label{inverse-bridge-def}
The \emph{reverse bridge} of a bridge deck $\gamma = (\D,x,\pi,\rho,(\omega)_{j\in J},(\eta_j)_{j \in J})$ is $\gamma^\ast = (\D,x^\ast,\rho,\pi,(\eta_j)_{j\in J},(\omega_j)_{j\in J})$. 
\end{definition}

\begin{lemma}\label{inverse-bridge-lemma}
If $\gamma \in \bridgeset{\Omega_\A}{\Omega_\B}{}$ for any two {\gQVB s} $\Omega_\A$ and $\Omega_\B$, then $\gamma^\ast \in \bridgeset{\Omega_\B}{\Omega_\A}{}$ and $\bridgelength{\gamma^\ast} = \bridgelength{\gamma}$.
\end{lemma}

\begin{proof}
We use the notations of Hypothesis (\ref{modular-bridge-hyp}). We note that for all $\omega\in\module{M}$ and $\eta\in\module{N}$ we have:
\begin{equation*}
\decknorm{\gamma^\ast}{\eta,\omega} = \decknorm{\gamma}{\omega,\eta}
\end{equation*}
by construction. This observation justifies the particular symmetry in Definition (\ref{deck-seminorm-def}).

Moreover for all $a\in\sa{\A}$ and $b\in\sa{\B}$, we have:
\begin{equation*}
\begin{split}
\bridgenorm{\gamma^\ast}{b,a} &= \|\pi_\B(b)x^\ast - x^\ast \pi_\A(a)\|_\D \\
&= \|(\pi_\A(a) x -  x \pi_\B(b))^\ast\|_\D = \bridgenorm{\gamma}{a,b} \text{.}
\end{split}
\end{equation*}

Thus $\bridgereach{\gamma} = \bridgereach{\gamma^\ast}$. The other claims of our lemma are self-evident.
\end{proof}
We remark that for any modular bridge $\gamma$ we have $\bridge{\gamma_\flat^{-1}} = \left(\bridge{\gamma^\ast}\right)_\flat$, by \cite[Proposition 4.7]{Latremoliere13}.

We will observe in the next section that we do not need the full generality afforded to us by Definition (\ref{modular-bridge-def}), as we could limit ourselves to working only with finite families of anchors (and hence of co-anchors). The reason for this observation is the following lemma.

\begin{lemma}\label{finite-anchor-set-lemma}
Let Hypothesis (\ref{modular-bridge-hyp}) be given. For any $\varepsilon > 0$, there exists a modular bridge $\gamma_\varepsilon$ from $\Omega_\A$ to $\Omega_\B$ such that:
\begin{enumerate}
\item $\bridgelength{\gamma_\varepsilon} \leq \bridgelength{\gamma} + \varepsilon$,
\item $\bridgeanchors{\gamma}$ and $\bridgecoanchors{\gamma}$ are finite families.
\end{enumerate}
\end{lemma}

\begin{proof}
Let $\varepsilon > 0$. Since:
\begin{equation*}
\modlip{1}{\Omega_\A} = \bigcup_{\omega\in\modlip{1}{\Omega_\A}} \module{M}(\omega,\varepsilon)
\end{equation*}
by Definition (\ref{bridge-imprint-def}), and since $\modlip{1}{\Omega_\A}$ is compact, there exists a finite set $J_1 \subseteq \modlip{1}{\Omega_\A}$ such that:
\begin{equation*}
\modlip{1}{\Omega_\A} = \bigcup_{\omega \in J_1} \module{M}(\omega, \varepsilon)\text{.}
\end{equation*}

Similarly, there exists a finite subset $J_2$ of $\modlip{1}{\Omega_\B}$ such that:
\begin{equation*}
\modlip{1}{\Omega_\B} = \bigcup_{\eta \in J_1} \module{N}(\eta, \varepsilon)\text{.}
\end{equation*}

Let $J_3 = J_1 \coprod J_2$ be the disjoint union of $J_1$ and $J_2$, itself a finite set.

If $j \in J_1$ then we write $\omega_j = j$ and we choose $\eta_j \in \targetsetbridge{\gamma}{j}{1}$. If $j\in J_2$ then we write $\eta_j = j$ and we choose $\omega_j \in \targetsetbridge{\gamma^\ast}{j}{1}$. These choices are possible since by Proposition (\ref{compact-targetset-prop}), the target sets involved are all nonempty (and as customary in functional analysis, we work within ZFC).

Let $\gamma_\varepsilon = (\Omega_\A,\Omega_\B,\D,x,\pi_\A,\pi_\B,(\omega_j)_{j\in J_3}, (\eta_j)_{j \in J_3})$.

We now make a few simple observations. We have $\decknorm{\gamma}{\omega_j,\eta_j} \leq \bridgereach{\gamma}$ for all $j\in J_3$, and thus $\bridgemodularreach{\gamma_\varepsilon} \leq \bridgereach{\gamma}$. On the other hand, by construction, $\bridgeimprint{\gamma_\varepsilon} \leq \varepsilon$. Last, we obviously have $\gamma_\circ = (\gamma_\varepsilon)_\circ$ by construction.

Therefore, $\bridgereach{\gamma_\varepsilon} \leq \bridgereach{\gamma} + \varepsilon$. This concludes our proof since $\bridgeheight{\gamma}=\bridgeheight{\gamma_\varepsilon}$.
\end{proof}

The generality of Definition (\ref{modular-bridge-def}) is however useful to describe modular bridges as morphisms in a category, as we shall do now. Indeed, the we are now ready to introduce the category of {\gQVB s} with modular treks, which generalize modular bridges and which carry a notion of length, from which the modular propinquity is computed.

\section{The modular propinquity}

The modular propinquity is constructed using certain morphisms for {\gQVB s}, called modular treks, which extend the notion of modular bridges to allow for the definition of composition. A modular trek, informally, is a finite path made of modular bridges whose codomains match the domain of the next modular bridge in the trek. It is immediate to define the length of a modular trek as the sum of the lengths of its constituent modular bridges. The length of a modular bridge, and by extension of a modular trek, replaces the notion of distortion for a correspondence sometimes used to define the Gromov-Hausdorff distance \cite{burago01}. The modular propinquity between any two {\gQVB s} $\Omega_\A$ and $\Omega_\B$ is the infimum of the lengths of any modular trek between $\Omega_\A$ and $\Omega_\B$. Concatenation of treks provide a notion of composition which translates to the fact that the modular propinquity satisfies the triangle inequality. Symmetry of the modular propinquity follow from the fact that treks are always reversible, in a sense to be made precise below. We will handle the more complicated coincidence axiom in the next section.

Since modular treks involve choices of modular bridges, just as with treks in the construction of the dual Gromov-Hausdorff propinquity \cite{Latremoliere13b}, we have much freedom in defining the modular propinquity to best suits a given context. Indeed, we may reduce the class of allowed modular bridges which may appear in a given modular trek by imposing additional constraints, such as asking the L-seminorms involved to be defined on a dense domain in the entire C*-algebra, additional Leibniz conditions such as the strong Leibniz property, or other additional requirements on D-norms, pivots, anchors or co-anchors (requirements on anchors and co-anchors should be symmetric to ensure that we obtain a metric). This flexibility proved helpful with the dual propinquity and will likely be as well for the modular propinquity.

Let us thus define modular treks formally:

\begin{definition}\label{modular-trek-definition}
Let $\mathcal{B}$ be a nonempty class of modular bridges. A \emph{modular $\mathcal{B}$-trek} $\Gamma = \left(\bridge{\gamma^j}\right)_{j \in \{1,\ldots,n\}}$ is given by $n\in\N\setminus\{0\}$ modular bridges $\gamma^0, \ldots, \gamma^n$ such that:
\begin{equation*}
\dom{\bridge{\gamma^{j+1}}} = \codom{\bridge{\gamma^j}} \text{ for all $j\in\{1,\ldots,n-1\}$.}
\end{equation*}
The \emph{domain} $\dom{\Gamma}$ of $\Gamma$ is $\dom{\bridge{\gamma^1}}$ and the \emph{codomain} $\codom{\Gamma}$ of $\Gamma$ is $\codom{\bridge{\gamma^n}}$.
\end{definition}

A modular trek is a modular $\mathcal{B}$-trek for some nonempty class $\mathcal{B}$ of modular bridges. 

We associate the following natural notion of length to modular treks:

\begin{definition}\label{treklength-def}
The \emph{length} of a modular trek $\Gamma = (\bridge{\gamma^j})_{j \in \{1,\ldots,n\}}$ is:
\begin{equation*}
\treklength{\Gamma} = \sum_{j=1}^n \bridgelength{\gamma}\text{.}
\end{equation*}
\end{definition}

Before introducing the modular propinquity, we first assemble the conditions needed on a class of modular bridges to allow for the construction of an actual metric in the following definition, which extends on \cite[Definition 3.10]{Latremoliere13b}.

\begin{definition}\label{compatible-bridge-class-def}
Let $(F,G,H)$ be an admissible triple. Let $\mathcal{C}$ be a nonempty class of {\QVB{F}{G}{H}s}. A class $\mathcal{B}$ of modular bridges is \emph{compatible} with $\mathcal{C}$ when:
\begin{enumerate}
\item for all $\bridge{\gamma}\in\mathcal{T}$, we have $\dom{\bridge{\gamma}},\codom{\bridge{\gamma}}\in\mathcal{C}$,
\item for all $\Omega_\A, \Omega_\B \in \mathcal{C}$, there exists a modular $\mathcal{B}$-trek from $\Omega_\A$ to $\Omega_\B$,
\item for all $\gamma \in \mathcal{T}$, we have $\gamma^\ast \in \mathcal{T}$,
\item for all $\Omega_\A$ and $\Omega_\B$ in $\mathcal{C}$, if there exists a full quantum isometry $\Theta : \Omega_\A\rightarrow\Omega_\A$ then for all $\varepsilon > 0$, there exists a modular $\mathcal{B}$-trek $\Gamma_\varepsilon$ from $\Omega_\A$ to $\Omega_\B$ with $\treklength{\Gamma_\varepsilon} < \varepsilon$.
\end{enumerate}
\end{definition}

\begin{example}
Let $(F,G,H)$ be an admissible triple. Let $\mathcal{C}$ be the class of all {\QVB{F}{G}{H}s} and let $\mathcal{B}$ be the class of all bridges between elements of $\mathcal{C}$. Note that a modular $\mathcal{B}$-trek consists of modular bridges which only involve {\QVB{F}{G}{H}s}. We check that $\mathcal{B}$ is compatible with $\mathcal{C}$.

Assertions (1) and (3) of Definition (\ref{compatible-bridge-class-def}) are trivial in this case.

Assertion (2) gives us a chance to observe that a bridge gives rise to a modular bridge. Let us use the notations of Hypothesis (\ref{modular-bridge-hyp}), with the additional assumption that $\Omega_\A,\Omega_\B \in \mathcal{C}$.

Let $(\D,x,\pi_\A,\pi_\B)$ be a bridge from $\A$ to $\B$ with $\|x\|_\D\leq 1$. If we pick any $\omega\in\modlip{1}{\Omega_\A}$ and $\eta\in\modlip{1}{\Omega_\B}$, then $(\Omega_\A,\Omega_\B,\D,x,\pi_\A,\pi_\B,\omega,\eta)$ is a modular bridge in $\mathcal{B}$ from $\Omega_\A$ to $\Omega_\B$  (identifying family of a single element with the element itself).

Now by \cite[Proposition 4.6]{Latremoliere13}, there does exist a bridge from $\A$ to $\B$ with a (self-adjoint) pivot of norm $1$. Thus, Assertion (2) holds as well.

Last, keeping the same notations, assume that $(\Theta,\theta)$ is a full quantum isometry from $\Omega_\A$ to $\Omega_\B$. We simply define the following one-bridge trek:
\begin{equation*}
\left(\Omega_\A,\Omega_\B,\B,\unit_\B,\theta,\mathrm{id}_\B, ( a )_{a \in \modlip{1}{\Omega_\A}}, (\Theta(a))_{a \in \modlip{1}{\Omega_\A}} \right)
\end{equation*}
where $\mathrm{id}_\B$ is the identity of $\B$. A straightforward computation shows that $\bridgelength{\gamma} = 0$.
\end{example}

We will find the following notation helpful.

\begin{notation}
Let $\mathcal{C}$ be nonempty class of {\gQVB s} and let $\mathcal{B}$ be a compatible class of modular bridges. Let $\Omega_\A$ and $\Omega_\B$ be chosen in $\mathcal{C}$. The class of all modular $\mathcal{B}$-treks from $\Omega_\A$ to $\Omega_\B$ is denoted by:
\begin{equation*}
\trekset{\mathcal{B}}{\Omega_\A}{\Omega_\B}\text{.}
\end{equation*}
\end{notation}

We are now ready to introduce the main definition of this work.
\begin{definition}\label{modular-propinquity-def}
Let $\mathcal{C}$ be a nonempty class of {\QVB{F}{G}{H}s} for some admissible triple $(F,G,H)$ and let $\mathcal{B}$ be a class of modular bridges compatible with $\mathcal{C}$. The \emph{modular Gromov-Hausdorff $\mathcal{B}$-propinquity} between two {\gQVB s} $\Omega_\A$ and $\Omega_\B$ in $\mathcal{C}$ is:
\begin{equation*}
\boxed{
\modpropinquity{\mathcal{B}}\left(\Omega_\A,\Omega_\B\right) = \inf\left\{ \treklength{\Gamma} : \Gamma \in \trekset{\mathcal{B}}{\Omega_\A}{\Omega_\B} \right\}\text{.}}
\end{equation*}
\end{definition}

\begin{notation}
If $\mathcal{C}$ is the class of all Leibniz {\gQVB s} and $\mathcal{B}$ is the class of all modular bridges, then $\modpropinquity{\mathcal{B}}$ is simply denoted $\modpropinquity{}$.
\end{notation}

We now proceed to prove that the modular propinquity is a metric up to full quantum isometry, for any compatible class of modular bridges. In the process, we will show that modular treks are morphisms in some category of {\gQVB s}. The coincidence axiom is by far the most involved property to establish, and will be the subject of the next section.

We begin by observing that the modular propinquity is always finite, and it dominates the quantum Gromov-Hausdorff propinquity. We begin with the natural definition of a basic trek.

\begin{definition}\label{basic-trek-def}
If $\Gamma = (\gamma^j)_{j\in \{1,\ldots,n\}}$ is a modular trek, then:
\begin{equation*}
\Gamma_\flat = \left(\basespace{\dom{\gamma^j}}, \gamma^j_\flat, \basespace{\codom{\gamma^j}} : j\in\{1,\ldots,n\} \right)
\end{equation*}
is a trek from $\basespace{\dom{\Gamma}}$ to $\basespace{\codom{\Gamma}}$.
\end{definition}

\begin{remark}
In \cite[Definition 3.20]{Latremoliere13}, treks explicitly included domains and codomains of bridges while bridges did not in \cite[Definition 3.6]{Latremoliere13}. We have made a different choice of notation, and thus our treks need not include the domain and codomain information already contained in modular bridges.
\end{remark}

\begin{proposition}\label{finite-modular-propinquity-prop}
Let $\mathcal{C}$ be a nonempty class of {\QVB{F}{G}{H}s} for some admissible triple $(F,G,H)$ and let $\mathcal{B}$ be a class of modular bridges compatible with $\mathcal{C}$. If:
\begin{equation*}
\Omega_\A = (\module{M}_\A,\inner{\cdot}{\cdot}{\A},\CDN_\A,\A,\Lip_\A)\text{ and }\Omega_\B = (\module{M}_\B, \inner{\cdot}{\cdot}{\B}, \CDN_\B, \B, \Lip_\B)
\end{equation*}
are two {\gQVB s} in $\mathcal{C}$, and if $\Gamma \in \trekset{\mathcal{B}}{\Omega_\A}{\Omega_\B}$, then:
\begin{equation*}
\treklength{\Gamma_\flat} \leq \treklength{\Gamma}\text{,}
\end{equation*}
and thus:
\begin{equation*}
\propinquity{\mathcal{C}}((\A,\Lip_\A),(\B,\Lip_\B)) \leq \modpropinquity{\mathcal{B}}{\mathcal{T}}(\Omega_\A,\Omega_\B) < \infty \text{.}
\end{equation*}
\end{proposition}

\begin{proof}
By Definition (\ref{basic-bridge-def}), if $\bridge{\gamma} \in \mathcal{B}$ then $\bridge{\gamma_\flat} \in \bridgeset{\A}{\B}{}$. Moreover $\bridgelength{\gamma_\flat}\leq\bridgelength{\gamma}$ since $\bridgeheight{\gamma_\flat} = \bridgeheight{\gamma}$ by Definition (\ref{bridge-height-def}) while $\bridgereach{\gamma_\flat} = \bridgebasicreach{\gamma} \leq \bridgereach{\gamma}$ by Definition (\ref{basic-reach-def}).

Now, if $\Gamma = (\bridge{\gamma^j})_{j\in\{1,\ldots,n\}} \in \trekset{\mathcal{B}}{\Omega_\A}{\Omega_\B}$ then $\gamma_\flat = (\bridge{\gamma_\flat^j})_{j\in\{1,\ldots,n\}}$ is a trek from $(\A,\Lip_\A)$ to $(\B,\Lip_\B)$ and:
\begin{equation*}
\treklength{\gamma_\flat} = \sum_{j=1}^n\bridgelength{\gamma^j_\circ} \leq \sum_{j=1}^n \bridgelength{\gamma^j} = \treklength{\Gamma}\text{.}
\end{equation*}
This proves that by definition:
\begin{equation*}
\propinquity((\A,\Lip_\A),(\B,\Lip_\B)) \leq \modpropinquity{\mathcal{B}}(\Omega_\A,\Omega_\B)\text{.}
\end{equation*}

The modular propinquity is finite since there exists at least one modular trek from $\Omega_\A$ to $\Omega_\B$ in $\mathcal{B}$ by Definition (\ref{compatible-bridge-class-def}). Now, a modular trek always has finite length, since modular bridges always have finite length by Lemma (\ref{finite-length-lemma}).
\end{proof}

We now prove that the modular propinquity is symmetric in its arguments and satisfies the triangle inequality. These facts rely on the fact that treks can be reversed and composed.

\begin{definition}
The \emph{reverse} of a modular trek $\Gamma = (\bridge{\gamma_j})_{j\in\{1,\ldots,n\}}$ is the modular trek $\Gamma^\ast = \left(\bridge{\gamma_{n+1-j}^{\ast}}\right)_{j\in\{1,\ldots,n\}}$.
\end{definition}

\begin{lemma}\label{reverse-trek-lemma}
Let $\mathcal{C}$ be a nonempty class of {\QVB{F}{G}{H}s}, where $(F,G,H)$ is an admissible triple, and let $\mathcal{B}$ be a class of modular bridges compatible with $\mathcal{C}$. If $\Gamma$ is a modular $\mathcal{B}$-trek then $\Gamma^\ast$ is a modular $\mathcal{B}$-trek from $\codom{\Gamma}$ to $\dom{\Gamma}$; moreover $\treklength{\Gamma} = \treklength{\Gamma^\ast}$.
\end{lemma}

\begin{proof}
This statement is immediate since a compatible class of modular bridges is closed by inversion of modular bridges by Definition (\ref{compatible-bridge-class-def}), and by Lemma (\ref{inverse-bridge-lemma}).
\end{proof}

We do not have a direct mean to compose modular bridges --- similarly as the situation with bridges in \cite{Latremoliere13}. However, we can easily compose modular treks.

\begin{definition}\label{composed-modular-trek-def}
Let $\Gamma_1 = \left(\bridge{\gamma_j^1}\right)_{j\in\{1,\ldots,n\}}$ and $\Gamma_2 = \left(\bridge{\gamma_j^2}\right)_{j\in\{1,\ldots,m\}}$ be two modular treks. The \emph{composed modular trek} $\Gamma_1\star\Gamma_2$ is the trek from $\dom{\Gamma_1}$ to $\codom{\Gamma_2}$ given by $\left(\bridge{\gamma_1^1},\ldots,\bridge{\gamma_n^1},\bridge{\gamma_1^2},\ldots,\bridge{\gamma_m^2}\right)$.
\end{definition}

\begin{lemma}\label{composed-trek-lemma}
Let $\mathcal{C}$ be a nonempty class of {\QVB{F}{G}{H}s}, with $(F,G,H)$ is an admissible triple, and let $\mathcal{B}$ be a class of modular bridges compatible with $\mathcal{C}$. If $\Gamma_1$ and $\Gamma_2$ are two modular $\mathcal{B}$-treks, then $\Gamma_1\star\Gamma_2$ is a modular $\mathcal{B}$-trek and:
\begin{equation*}
\treklength{\Gamma_1\star\Gamma_2} = \treklength{\Gamma_1} + \treklength{\Gamma_2}\text{.}
\end{equation*}
\end{lemma}

\begin{proof}
The result follows immediately from the Definition (\ref{treklength-def}) of the length of a modular trek and Definition (\ref{composed-modular-trek-def}).
\end{proof}

\begin{proposition}\label{triangle-inequality-modular-propinquity-prop}
Let $\mathcal{C}$ be a nonempty class of {\QVB{F}{G}{H}s}, with $(F,G,H)$ an admissible triple, and let $\mathcal{B}$ be a class of modular bridges compatible with $\mathcal{C}$. If $\Omega_\A$, $\Omega_\B$, and $\Omega_\D$ are three {\gQVB s} in $\mathcal{C}$, then:
\begin{equation*}
\modpropinquity{\mathcal{B}}\left(\Omega_\A,\Omega_\B\right) \leq \modpropinquity{\mathcal{B}}\left(\Omega_\A,\Omega_\D\right) + \modpropinquity{\mathcal{B}}\left(\Omega_\D,\Omega_\B\right) \text{,}
\end{equation*}
and
\begin{equation*}
\modpropinquity{\mathcal{B}}\left(\Omega_\A,\Omega_\B\right) = \modpropinquity{\mathcal{B}}\left(\Omega_\B,\Omega_\A\right)\text{.}
\end{equation*}
\end{proposition}

\begin{proof}
Let $\varepsilon > 0$. There exists modular treks $\Gamma_1$ and $\Gamma_2$, respectively from $\Omega_\A$ to $\Omega_\B$ and $\Omega_\B$ to $\Omega_\D$, such that:
\begin{equation*}
\treklength{\Gamma_1} \leq \modpropinquity{B}(\Omega_\A,\Omega_\B) + \frac{\varepsilon}{2}\text{ and }\treklength{\Gamma_2} \leq \modpropinquity{B}(\Omega_\B,\Omega_\D) + \frac{\varepsilon}{2}\text{.}
\end{equation*} 
Let $\Gamma = \Gamma_1 \star \Gamma_2$. Then:
\begin{equation*}
\begin{split}
\modpropinquity{B}(\Omega_\A,\Omega_\D) &\leq \treklength{\Gamma} \\
&= \treklength{\Gamma_1} + \treklength{\Gamma_2} \\
&\leq \modpropinquity{B}(\Omega_\A,\Omega_\B) + \modpropinquity{B}(\Omega_\B,\Omega_\D) + \varepsilon\text{.} 
\end{split}
\end{equation*}
As $\varepsilon > 0$ is arbitrary, we conclude that:
\begin{equation*}
\modpropinquity{\mathcal{B}}\left(\Omega_\A,\Omega_\B\right) \leq \modpropinquity{\mathcal{B}}\left(\Omega_\A,\Omega_\D\right) + \modpropinquity{\mathcal{B}}\left(\Omega_\D,\Omega_\B\right) \text{,}
\end{equation*}
as desired.

Symmetry follows from Lemma (\ref{reverse-trek-lemma}). 
\end{proof}

We conclude by observing that the modular propinquity is a pseudo-metric, i.e. in addition to being finite, symmetric and satisfy the triangle inequality, it is null whenever two {\gQVB s} are full quantum isometric.

\begin{proposition}\label{sufficient-zero-modular-propinquity-prop}
Let $\mathcal{C}$ be a nonempty class of {\QVB{F}{G}{H}s}, with $(F,G,H)$ an admissible triple, and let $\mathcal{B}$ be a class of modular bridges compatible with $\mathcal{C}$. Let:
\begin{equation*}
\Omega_\A = (\module{M}_\A,\inner{\cdot}{\cdot}{\A},\CDN_\A,\A,\Lip_\A)\text{ and }\Omega_\B = (\module{M}_\B, \inner{\cdot}{\cdot}{\B}, \CDN_\B, \B, \Lip_\B)
\end{equation*} be two {\gQVB s} in $\mathcal{C}$.

If there exists a full quantum isometry $(\theta,\Theta)$ from $\Omega_\A$ to $\Omega_\B$, then $\modpropinquity{\mathcal{B}}(\Omega_\A,\Omega_\B) = 0$.
\end{proposition}

\begin{proof}
By Definition (\ref{compatible-bridge-class-def}), for all $\varepsilon > 0$, there exists a modular $\mathcal{B}$-trek $\Gamma_\varepsilon$ from $\Omega_\A$ to $\Omega_\B$ such that $\treklength{\Gamma_\varepsilon} < \varepsilon$. Thus $\modpropinquity{\mathcal{B}}(\Omega_\A,\Omega_\B) < \varepsilon$. This proves our result.
\end{proof}

We pause for an observation which formalizes the intuition we have followed when working with treks. If $\Omega = (\module{M},\inner{\cdot}{\cdot}{\module{M}}, \CDN, \A,\Lip_\A)$ is a {\gQVB} then we may define a canonical modular bridge $\mathrm{idbridge}_\Omega$ from $\Omega_\A$ to $\Omega_\B$ by setting:
\begin{equation*}
\mathrm{idbridge}_\Omega = \left(\A,\unit_\A,\mathrm{id}_\A,\mathrm{id}_\A,(\omega)_{\omega\in \modlip{1}{\Omega}}, (\omega)_{\omega\in\modlip{1}{\Omega}}\right) \in \bridgeset{\Omega}{\Omega}{} \text{,}
\end{equation*}
where $\mathrm{id}_\A$ is the identity *-automorphism of $\A$. We immediately that $\bridgelength{\mathrm{idbridge}} = 0$, and it is natural to think of $\mathrm{idbridge}$ as the identity bridge of $\Omega$. 

Identifying modular bridges with modular treks reduced to a single bridge, we thus seem to have gathered many key ingredients for a category: modular treks compose, and we have an identity modular trek for any {\gQVB s}. Moreover, we will extend in the next section the various properties of target sets for modular bridges to modular treks; while not a part of the requirement to define a category, these morphism-like properties certainly push forth the idea that modular treks ought to be considered a type of morphisms of {\gQVB s}.

There are two small issues to deal with to complete this picture. First of all, we must work with modular treks up to a notion of reduction. Indeed, even composition a trek with the identity trek of its domain or co-domain does not lead to the same modular trek with our definitions. It is however easy to define a notion of a \emph{reduced modular trek}, which is a trek with no loop. Formally, if $\Gamma = (\gamma_j)_{j\in \{1,\ldots,n\}}$ is a modular trek, then we shall say that $\Gamma$ is reduced there exists no $j < k \in \{1,\ldots,n\}$ such that $\dom{\gamma_j} = \codom{\gamma_k}$ and $\bridgeanchors{\gamma_j} = \bridgecoanchors{\gamma_k}$. It is trivial to prove that any modular trek can be reduced, i.e. it admits a subfamily which is a reduced trek with the same domain and codomain. We note that a modular trek with a single bridge is by definition reduced.

Now, we can compose two reduced modular treks to a reduced modular trek simply by reducing their composition as defined in Definition (\ref{composed-modular-trek-def}). It is a simple exercise to check that composition of reduced treks thus defined is associative and that the identity treks act as units for the composition.

The second small issue is that our morphism sets for our prospective category are not sets. There are simply too many possible modular treks between any two {\gQVB s}. However, this is a very minor issue. The simplest and often sufficient mean to fix this is to restrict which class of {\gQVB s} we work with in a given context, making sure this class is a set, and then use modular treks formed only with {\gQVB s} in this set. 

When working with treks, rather than modular treks, a similar construction in \cite{Latremoliere13} led to a category with reduced treks as morphisms over the class of {\gQqcms s}, and all reduced treks were isomorphisms --- i.e. invertible. We note that in our current modular version, modular treks may not be invertible, as being invertible requires that the sets of anchors and co-anchors be the entire closed unit balls of the D-norms of their domain and codomain. In particular, there are many single-bridge modular treks which are not an identity bridge. 

Now, the length of a modular trek is larger than the length of its reduction, and thus we could define the modular propinquity with reduced treks only if desired without changing its value. This would introduce unneeded complications, but it is worth noting that we can bring our construction within this framework. Indeed, it really shows that the modular propinquity is constructed via a sort of generalized correspondences in the metric sense.

\bigskip

We now turn to proving that the modular propinquity is indeed a metric up to full quantum isometry. 

\section{Distance Zero}

We continue our study of the morphism-like properties of modular treks. We extend the notion of a target set from modular bridges to modular treks, using the notion of an itinerary. There are two kind of target sets for treks: one defined for elements in modules and one defined for elements in {\gQqcms s}. The latter follows the same ideas as in \cite{Latremoliere13}.

Once more, we will group certain common notations and hypothesis for multiple use in this section.

\begin{hypothesis}\label{modular-trek-hyp}
Let $\mathcal{C}$ be a nonempty class of {\QVB{F}{G}{H}s}, with $(F,G,H)$ be an admissible triple, and let $\mathcal{B}$ be a class of modular bridges compatible with $\mathcal{C}$. Let:
\begin{equation*}
\Omega_\A = (\module{M},\inner{\cdot}{\cdot}{\module{M}}, \CDN_{\module{M}},\A,\Lip_\A)\text{ and }\Omega_\B = (\module{N},\inner{\cdot}{\cdot}{\module{N}}, \CDN_{\module{N}},\B,\Lip_\B)
\end{equation*}
be two {\gQVB s} in $\mathcal{C}$. Let $l \geq 0$.

Let $\Gamma = \left(\bridge{\gamma_j}\right)_{j\in\{1,\ldots,n\}}$ be a modular trek from $\Omega_\A$ to $\Omega_\B$.
\end{hypothesis}

We begin by recalling \cite[Definition 5.7]{Latremoliere13}, adjusted to our context.

\begin{definition}\label{itinerary-def}
Let Hypothesis (\ref{modular-trek-hyp}) be given. Let $l \geq 0$. An $l$-\emph{itinerary} from $a\in\dom{\Lip_\A}$ to $b \in\dom{\Lip_\B}$ along the modular trek $\Gamma$ is an $l$-itinerary from $a$ to $b$ along the basic trek $\Gamma_\flat$, i.e. a family $(d_j)_{j\in\{0,\ldots,n\}}$ such that:
\begin{enumerate}
\item $d_0 = a$,
\item $d_n = b$, 
\item $d_{j+1} \in \targetsetbridge{\gamma_{j+1}}{d_j}{l}$ for all $j\in\{0,\ldots,n-1\}$.
\end{enumerate}

The set of all $l$-itineraries along $\Gamma$ starting at $a$ and ending at $b$ is denoted by:
\begin{equation*}
\itineraries{\Gamma}{a}{b}{l}\text{.}
\end{equation*}
\end{definition}

We now generalize the notion of itinerary to modules.

\begin{definition}\label{mod-itinerary-def}
Let Hypothesis (\ref{modular-trek-hyp}) be given and $l \geq 0$. An $l$-\emph{itinerary} from $\omega\in\dom{\CDN_{\module{M}}}$ to $\eta\in\dom{\CDN_{\module{N}}}$ along the modular trek $\Gamma$ is a family $(\xi_j)_{j\in\{0,\ldots,n\}}$ such that:
\begin{enumerate}
\item $\xi_0 = \omega$,
\item $\xi_n = \eta$, 
\item $\xi_{j+1} \in \targetsetbridge{\gamma_{j+1}}{\xi_j}{l}$ for all $j\in\{0,\ldots,n-1\}$.
\end{enumerate}

The set of all itineraries along $\Gamma$ starting at $\omega$ and ending at $\eta$ is denoted by:
\begin{equation*}
\itineraries{\Gamma}{\omega}{\eta}{l}\text{.}
\end{equation*}

\end{definition}

Itineraries allow us to extend the notion of a target set from modular bridges to modular treks.

\begin{definition}\label{trek-targetset-def}
Let Hypothesis (\ref{modular-trek-hyp}) be given. The \emph{target set} for some $a\in\dom{\Lip_\A}$ and $l \geq \Lip_\A(a)$ along the modular trek $\Gamma$ is:
\begin{equation*}
\targetsettrek{\Gamma}{\omega}{l} = \left\{ b : \itineraries{\Gamma}{a}{b}{l} \not= \emptyset \right\}\text{.}
\end{equation*}
\end{definition}

\begin{definition}\label{trek-mod-targetset-def}
Let Hypothesis (\ref{modular-trek-hyp}) be given. The \emph{target set} for some $\omega\in\dom{\CDN_{\module{M}}}$ and $l \geq \CDN_{\module{M}}$ along the modular trek $\Gamma$ is:
\begin{equation*}
\targetsettrek{\Gamma}{\omega}{l} = \left\{ \eta : \itineraries{\Gamma}{\omega}{\eta}{l} \not= \emptyset \right\}\text{.}
\end{equation*}
\end{definition}

Definition (\ref{targetset-def}) was chosen to ensure that, given a modular trek $\Gamma$, for all $a\in\basespace{\dom{\Gamma}}$, we have $\targetsettrek{\Gamma}{a}{l} = \targetsettrek{\Gamma_\flat}{a}{l}$, thus allowing us to directly invoke \cite{Latremoliere13} to conclude:

\begin{proposition}[{\cite[Propositions 5.11 and 5.12]{Latremoliere13}}]\label{targetset-prop}
Let us assume Hypothesis (\ref{modular-trek-hyp}). Let $a, a' \in \dom{\Lip_\A}$ and let $l \geq \max\{ \Lip_\A(a), \Lip_\A(a') \}$. If $b\in\targetsettrek{\Gamma}{a}{l}$ and $b'\in\targetsettrek{\Gamma}{a'}{l}$ then the following assertions hold:
\begin{enumerate}[1.]
\item $\|b - b'\|_\B \leq \|a - a'\|_\A + 4 l\treklength{\Gamma}  \text{.}$
\item $\diam{\targetsettrek{\gamma_\flat}{a}{l}}{\|\cdot\|_\B} \leq 4 l \treklength{\Gamma_\flat}\text{.}$
\item for all $t\in\R$, we have:
\begin{equation*}
\eta  t \eta' \in \targetsettrek{\Gamma}{a + t a'}{l(1 + |t|)}
\end{equation*}
\item we have:
\begin{equation*}
\Jordan{b}{b'} \in \targetsettrek{\Gamma}{\Jordan{a}{a'}}{F(\|a\|_\A + 2 l \treklength{\Gamma}, \|a'\|_\A + 2 l \treklength{\Gamma}, l, l)}
\end{equation*}
and:
\begin{equation*}
\Lie{b}{b'} \in \targetsettrek{\Gamma}{\Lie{a}{a'}}{F(\|a\|_\A + 2 l \treklength{\Gamma}, \|a'\|_\A + 2 l \treklength{\gamma_\flat}, l, l)}\text{.}
\end{equation*}
\item $\targetsettrek{\Gamma}{a}{l}$ is a nonempty subset of $\alglip{l}{\Lip_\B}$.
\end{enumerate}

\end{proposition}

\begin{proof}
Note that $\gamma_\flat = \left(\bridge{\gamma^j_\circ}\right)_{j\in\{1,\ldots,n\}}$ is a trek from $(\A,\Lip_\A)$ to $(\B,\Lip_\B)$, and $\targetsettrek{\Gamma}{a}{l} = \targetsettrek{\gamma_\flat}{a}{l}$, while $\treklength{\Gamma_\flat} \leq \treklength{\Gamma}$. Thus we may apply our work in \cite{Latremoliere13}.

Alternatively, all the statements in this proposition follow from similar techniques to Proposition (\ref{trek-targetset-morphism-prop}) applied to Proposition (\ref{base-targetset-prop}) and Proposition (\ref{trek-targetset-compact-prop}).
\end{proof}

With our notion of target sets in hand, we now can generalize Propositions (\ref{compact-targetset-prop}), (\ref{diameter-prop}), (\ref{linear-prop}) and (\ref{product-prop}) from modular bridges to modular treks.

\begin{proposition}\label{trek-targetset-morphism-prop}
Let us assume Hypothesis (\ref{modular-trek-hyp}), and let us assume that $\Omega_\A$ and $\Omega_\B$ are {\QVB{F}{G}{H}} for some admissible triple $(F,G,H)$.

Let $\omega, \omega' \in \dom{\CDN_{\module{M}}}$ and let $l \geq \max\{ \CDN_{\module{M}}(\omega), \CDN_{\module{M}}(\omega') \}$. If $\eta\in\targetsettrek{\Gamma}{\omega}{l}$ and $\eta'\in\targetsettrek{\Gamma}{\omega'}{l}$ then the following assertions hold:
\begin{enumerate}[1.]
\item $\KantorovichMod{\Omega_\B}(\eta,\eta') \leq \sqrt{2}\left(\KantorovichMod{\Omega_\A}(\omega,\omega') + (4 l + H(2 l , 1) ) \treklength{\Gamma}\right)\text{.}$
\item $\diam{\targetsettrek{\Gamma}{\omega}{l}}{\Kantorovich{\Omega_\B}} \leq \sqrt{2} ( 4 l + H(2 l, 1) ) \treklength{\Gamma}\text{.}$
\item for all $t\in\C$, we have:
\begin{equation*}
\eta + t \eta' \in \targetsettrek{\Gamma}{\omega + t\omega'}{l(1 + |t|)}
\end{equation*}
\item for all $a\in\dom{\Lip_\A}$ and for all $l' \geq \Lip_\A(a)$, if $b \in \targetsettrek{\Gamma}{a}{l'}$, then we have:
\begin{equation*}
b\eta \in \targetsettrek{\Gamma}{a\omega}{G(\|a\|_\A + 2 l' \treklength{\Gamma},l',l)}\text{.}
\end{equation*}
\item If $b \in \targetsettrek{\Gamma}{\inner{\omega}{\omega}{\module{M}}}{H(l,l)}$ then:
\begin{equation*}
\left\| b - \inner{\eta}{\eta}{\module{N}} \right\|_\B \leq \left(8 l \sqrt{2} + H( 2l, 2l ) + 6 H( l,l ) + 2\sqrt{2} H( 2l, 1 ) \right) \treklength{\Gamma} \text{.}
\end{equation*}
\end{enumerate}

\end{proposition}

\begin{proof}
We write $\Omega_j = \codom{\bridge{\gamma^j}}$ for all $j\in\{1,\ldots,n\}$ and $\Omega_0 = \Omega_\A$.

Let $(\xi_0,\ldots,\xi_n) \in \itineraries{\Gamma}{\omega}{\eta}{l}$ and $(\xi'_0,\ldots,\xi'_n) \in \itineraries{\Gamma}{\omega'}{\eta'}{l}$. Since $\xi_{j+1} \in \targetsetbridge{\gamma_{j+1}}{\xi_j}{l}$, Proposition (\ref{diameter-prop}) gives us:
\begin{equation*}
\KantorovichMod{\Omega_{j+1}}(\xi_{j+1},\xi'_{j+1}) \leq \sqrt{2}\left(\KantorovichMod{\Omega_j}(\xi_j,\xi'_j) + (4 l + H(2 l, 1)) \bridgelength{\gamma_{j+1}} \right) \text{.}
\end{equation*}
Thus by induction, we get:
\begin{equation*}
\begin{split}
\KantorovichMod{\Omega_\B}(\eta,\eta') &\leq \sqrt{2}\left(\Kantorovich{\Omega_\A}(\omega,\omega') + (4 l + H(2 l, 1)) \sum_{j=1}^n \bridgelength{\gamma_j}\right) \\
&= \sqrt{2}\left(\Kantorovich{\Omega_\A}(\omega,\omega') + (4 l + H(2 l, 1)) \treklength{\Gamma} \right) \text{.}
\end{split}
\end{equation*}
If $\omega = \omega'$, then we obtain that $\diam{\targetsettrek{\Gamma}{\omega}{l}}{\Kantorovich{\Omega_\B}} \leq  \sqrt{2}( 4 l + H(2 l, 1)) \treklength{\Gamma}\text{.}$

We also have $\eta_{j+1} + t\eta'_{j+1} \in \targetsetbridge{\gamma_j}{\eta_{j} + t\eta'_j}{l + |t|l}$ by Proposition (\ref{linear-prop}). Thus, by induction, we get that:
\begin{equation*}
\eta+t\eta' \in \targetsettrek{\Gamma}{\omega+t\omega'}{l + |t|l'}\text{.}
\end{equation*}

Let now:
\begin{equation*}
(b_j)_{j=0}^n \in \itineraries{\Gamma}{a}{b}{l}\text{.}
\end{equation*}
For each $j$ we have $b_{j+1}\eta_{j+1} \in \targetsetbridge{\gamma_j}{b_j\eta_j}{G(\|b_j\| + l \bridgelength{\gamma_j},r,l)}$ by Proposition (\ref{product-prop}). Now, as before, we have $\|b_j\| \leq \|a\|_\A + l \sum_{k=0}^j \bridgelength{\gamma_k} \leq \|a\|_\A + l \treklength{\Gamma}$, and since $\bridgelength{\gamma_j} \leq \treklength{\Gamma}$, we have:
\begin{equation*}
d_{j+1}\eta_{j+1} \in \targetsetbridge{\gamma_j}{d_j\eta_j}{G(\|a\|_\A + 2 l \treklength{\Gamma},r,l)}
\end{equation*}
since $G(\cdot,r,l)$ is weakly increasing. This proves in turn that:
\begin{equation*}
b\eta \in \targetsettrek{\Gamma}{a\omega}{G(\|a\|_\A + 2 l \treklength{\Gamma}, r, l)}\text{.}
\end{equation*}

Last, we address the property of target sets for treks and inner products. To ease notations, we set:
\begin{equation*}
C = \left(8 l \sqrt{2} + H( 2l, 2l ) + 2 H( l,l ) + 2\sqrt{2} H( 2l, 1 ) \right)\text{,}
\end{equation*}
which is the constant in Proposition (\ref{inner-morphism-prop}). Moreover, we set $\Omega_j = (\module{M}_j, \inner{\cdot}{\cdot}{j}, \CDN_j, \A_j, \Lip_j)$ for all $j \in \{0,\ldots,n\}$. Moreover, we write $\gamma_j = (\D_j, x_j, \pi_j, \rho_j, \bridgeanchors{\gamma_j}, \bridgecoanchors{\gamma_j})$.

Let $b \in \targetsettrek{\Gamma}{\inner{\omega}{\omega}{\module{M}}}{H(l,l)}$ and $(b_j)_{j=0}^n \in \itineraries{\Gamma}{\inner{\omega}{\omega}{\module{M}}}{b}{H(l,l)}$.

Let us assume that for some $j \in \{1,\ldots,n-1\}$, we have:
\begin{equation}\label{trek-morphism-induction-hyp1}
\left\| b_j - \inner{\xi_j}{\xi_j}{\module{N_j}} \right\|_{\A_{j}} \leq (C + 4 H(l,l)) \sum_{k=1}^j \bridgelength{\gamma_j} \text{.}
\end{equation}

By Definition (\ref{itinerary-def}), we have $b_{j+1} \in \targetsetbridge{\gamma_{j+1}}{b_j}{H(l,l)}$ and $\xi_{j+1} \in \targetsetbridge{\gamma_{j+1}}{\xi_j}{l}$. There is no expectation that $b_{j+1} \in \targetsetbridge{\gamma_{j+1}}{\inner{\xi_j}{\xi_j}{\module{M}_j}}{H(l,l)}$. So we introduce $b_{j+1}' \in \targetsetbridge{\gamma_{j+1}}{\inner{\xi_j}{\xi_j}{\module{M}_j}}{H(l,l)}$. By Proposition (\ref{diameter-prop}), we have:
\begin{equation*}
\left\| b_{j+1} - b_{j+1}'\right\|_{\A_{j+1}} \leq \| b_j - \inner{\xi_j}{\xi_j}{\module{M}_j} \|_{\A_j} + 4 H(l,l) \bridgelength{\gamma_{j+1}} \text{.}
\end{equation*}
On the other hand, by Proposition (\ref{inner-morphism-prop}), we have:
\begin{equation*}
\left\| b_{j+1}' - \inner{\xi_{j+1}}{\xi_{j+1}}{\module{M_{j+1}}} \right\|_{\A_{j+1}} \leq C \bridgelength{\gamma_{j+1}}\text{.}
\end{equation*}

Therefore:
\begin{equation*}
\left\| b_{j+1} - \inner{\xi_{j+1}}{\xi_{j+1}}{\module{M}_{j+1}} \right\|_{\A_{j+1}} \leq (C + 4 H(l,l))\bridgelength{\gamma_{j+1}} \text{,}
\end{equation*}
and thus using our induction hypothesis (\ref{trek-morphism-induction-hyp1}), we get:
\begin{equation*}
\left\| b_{j+1} - \inner{\xi_{j+1}}{\xi_{j+1}}{\module{M}_{j+1}} \right\|_{\A_{j+1}} \leq (C + 4 H(l,l)) \sum_{k=1}^{j+1} \bridgelength{\gamma_k} \text{,}
\end{equation*}
which is our induction hypothesis (\ref{trek-morphism-induction-hyp1}) for $j+1$. 

Now, by Proposition (\ref{inner-morphism-prop}), we have:
\begin{equation*}
\begin{split}
\left\| b_1 - \inner{\xi_1}{\xi_1}{\module{M}_1}\right\|_{\A_1} &\leq C \bridgelength{\gamma_1} \\
&\leq (C + 4 H( l,l ))\bridgelength{\gamma_1} \text{.}
\end{split}
\end{equation*}

Therefore, by induction, we have proven that:
\begin{equation*}
\left\| b - \inner{\eta}{\eta}{\module{N}}\right\|_\B \leq (C + 4 H(l,l)) \treklength{\Gamma} \text{.}
\end{equation*}

This concludes our proof.
\end{proof}

We also prove that target sets of modular treks are compact.
\begin{proposition}\label{trek-targetset-compact-prop}
Let us assume Hypothesis (\ref{modular-trek-hyp}). If $\omega\in\dom{\CDN_{\module{M}}}$ and $l \geq \CDN_{\module{M}}(\omega)$ then $\targetsettrek{\Gamma}{\omega}{l}$ is a nonempty and compact subset of $\modlip{r}{\Omega_\B}$ for $\|\cdot\|_{\module{N}}$ (equivalently for $\KantorovichMod{\Omega_\B}$).
\end{proposition}

\begin{proof}
We write $\Omega_j = \codom{\bridge{\gamma^j}}$ and $\Omega_j = (\module{M}_j,\inner{\cdot}{\cdot}{\module{M}_j},\CDN_{\module{M}_j},\A_j,\Lip_j)$ for all $j\in\{1,\ldots,n\}$.

We first note that a trivial induction prove that $\targetsettrek{\Gamma}{\omega}{l}$ is not empty using Proposition (\ref{compact-targetset-prop}).

By construction, $\targetsettrek{\Gamma}{\omega}{l}$ is a subset of the $\|\cdot\|_{\module{N}}$--compact set $\modlip{l}{\Omega_\B}$. Thus it is sufficient to prove that it is closed for $\|\cdot\|_{\module{N}}$.

Let $(\eta_k)_{k\in\N}$ be a sequence in $\targetsettrek{\Gamma}{\omega}{l}$, converging to some $\eta\in\module{N}$ for $\|\cdot\|_{\module{N}}$.

Now, for each $k\in\N$, let $(\omega,\eta_k^1,\ldots,\eta_k^n)$ be an $l$-itinerary from $\omega$ to $\eta_k$. By Definition (\ref{qvb-def}), each sequence $(\eta_k^j)_{k\in\N}$ lies in the compact set $\left\{\xi\in\module{M}_j : \CDN_{\module{N}_j}(\xi) \leq l \right\}$ for all $j\in\{1,\ldots,n\}$. Thus by a trivial induction, there exists strictly increasing functions $f_j : \N \rightarrow\N$ for $j\in\{1,\ldots,n\}$ such that $(\eta^j_{f_1\circ\cdots\circ f_j(k)})_{k\in\N}$ converges to some $\eta^j \in \module{M}_j$ for $\|\cdot\|_{\module{M}_j}$, for all $j\in\{1,\ldots,n\}$. Let $g : k \in \N \mapsto f_1 \circ f_2 \circ\cdots\circ f_n(k)$, so that $(\eta_{g(k)}^j)_{k\in\N}$ converges to $\eta^j$ for all $j\in\{1,\ldots,n\}$.

Our goal is to prove that $(\omega,\eta^1,\ldots,\eta^{n-1},\eta = \eta^n)$ is an $l$-itinerary along $\Gamma$.

To begin with, $\CDN_{\module{M}_j}(\eta^j) \leq l$ since $\CDN_{\module{M}_j}$ is lower semi-continuous for all $j\in\{1,\ldots,n\}$.

Second of all, by continuity, we also have for all $j\in\{1,\ldots,n\}$:
\begin{equation*}
\decknorm{\gamma_j}{\eta^j, \eta^{j+1}} = \lim_{k\rightarrow\infty} \decknorm{\gamma_j}{\eta^j_k,\eta^{j+1}_k} \leq l \bridgereach{\gamma_j} \text{.}
\end{equation*}
This concludes our proof.
\end{proof}

Proposition (\ref{trek-targetset-compact-prop}) shows that modular trek target sets are in the hyperspace of a compact metric space, namely a closed unit ball for some D-norm: the norm topology and the modular {\MongeKant} topology on these balls are indeed the same and compact. The proof of our main Theorem (\ref{main-thm}) relies on an important property of the topology induced by the Hausdorff distance over the hyperspace of all nonempty closed subsets of a compact space: it only depends on the topological equivalence class of the chosen metric. We recall this well-known fact and include a proof for the convenience of the reader.

\begin{lemma}\label{Vietoris-lemma}
Let $X$ be a compact space with topology $\tau$ and let $\mathcal{F} = \{ U^c : U \in \tau, U\not=X \}$ be the set of all nonempty closed subsets of $X$. The \emph{Vietoris topology} is the smallest topology on $\mathcal{F}$ generated from the topological basis:
\begin{equation*}
\mathscr{O}(U,V_1,\ldots,V_n) = \left\{ F \in \mathcal{F} : F\subseteq U\text{ and }\forall j \in \{1,\ldots,n\} \quad F\cap V_j \not= \emptyset \right\}
\end{equation*}
for all $n\in\N$ and $U,V_1,\ldots,V_n \in \tau$.

If $\mathrm{d}$ is a metric on $X$ which induced $\tau$, then the topology induced by $\Haus{\mathrm{d}}$ is the Vietoris topology.

Consequently, if $\mathrm{d}_1$ and $\mathrm{d}_2$ are two metrics which induce the same topology on $X$ then $\Haus{\mathrm{d}_1}$ and $\Haus{\mathrm{d}_2}$ induce the same topology on $\mathcal{F}$.
\end{lemma}

\begin{proof}
Let $F \in \mathcal{F}$ and $r > 0$. Since $F$ is compact, there exists $x_1,\ldots,x_n \in F$ for some $n\in\N$ such that $F \subseteq \bigcup_{j=1}^n X\left(x_j,\frac{r}{2}\right)$ where the open ball in $(X,\mathrm{d})$ of center any $y\in X$ and radius $r$ is denoted by $X(y,r)$. For all $j\in \{1,\ldots,n\}$, we set $V_j = X\left(x_j,\frac{r}{2}\right)$.

Let $U = \bigcup_{j=1}^n V_j$. Note that by construction, $F \in \mathcal{O}(U,V_1,\ldots,V_n)$. Now let $G \in \mathcal{O}(U,V_1,\ldots,V_n)$. If $x\in G$, then $x\in U$ and thus $x \in V_j$ for some $j\in\{1,\ldots,n\}$, implying that $\mathrm{d}(x,F) < \frac{r}{2}$. If $x\in F$, then $x \in V_j$ for some $j \in \{1,\ldots,n\}$. Since $G\cap V_j \not=\emptyset$, there exists $y \in G\cap V_j$ and by definition of $V_j$, we conclude $\mathrm{d}(x,y) < r$. Hence $\Haus{\mathrm{d}}(F,G) < r$. Thus $\mathcal{O}(U,V_1,\ldots,V_n) \subseteq \mathcal{F}(F,r)$.

Let now $U,V_1,\ldots,V_n \in \tau$ be given with $F \in \mathcal{O}(U,V_1,\ldots,V_n)$. Since $X\setminus U$ is closed and disjoint from $F$, we conclude that there exists $\varepsilon_0 > 0$ such that, for all $x\in F$ and $y \in X\setminus U$, we have $\mathrm{d}(x,y) \geq \varepsilon_0$.

Now, for each $j\in \{1,\ldots,n\}$, there exists $x_j \in F\cap V_j$ and there exists $\varepsilon_j > 0$ such that $X(x_j,\varepsilon_j) \subseteq V_j$. Let $\varepsilon = \min\{\varepsilon_j : j \in \{0,\ldots,n\} \}$.

Let $G \in \mathcal{F}(F,\varepsilon)$. Let $x\in G$. There exists $y \in F$ such that $\mathrm{d}(x,y) < \varepsilon$. Thus $x \in U$ since $\mathrm{d}(x,y) < \varepsilon_0$. Thus $G\subseteq U$.

Let $j\in\{1,\ldots,n\}$. There exists $y \in G$ such that $\mathrm{d}(x_j,y) < \varepsilon \leq \varepsilon_j$, and thus by construction, $y \in X(x_j,\varepsilon_j) \subseteq V_j$ and thus $G\cap V_j \not= \emptyset$. We thus have shown that $G \in \mathcal{O}(U,V_1,\ldots,V_n)$. Thus $\mathcal{F}(F,\varepsilon) \subseteq \mathcal{O}(U,V_1,\ldots,V_n)$.

This proves our lemma.
\end{proof}

We conclude our preliminary statements with a simple, useful lemma which we will use a few times in our proof of our main theorem.

\begin{lemma}\label{Haus-pick-lemma}
Let $(E,\dist)$ be a compact metric space. Let $(A_n)_{n\in\N}$ be a sequence of closed subsets of $E$ converging to some singleton $\{a\}$ for $\Haus{\dist}$.

If $(x_n)_{n\in\N}$ is a sequence in $E$ such that $x_n \in A_n$ for all $n\in\N$, then $(x_n)_{n\in\N}$ converges to $a$.
\end{lemma}

\begin{proof}
Let $\varepsilon > 0$. There exists $N\in\N$ such that for all $n\geq N$, we have:
\begin{equation*}
\Haus{\dist}(A_n,\{a\}) < \varepsilon \text{.}
\end{equation*}
Thus $\dist(x_n,a) < \varepsilon$ for all $n\geq N$.
\end{proof}

We are now ready to prove our main theorem.

\begin{theorem}\label{main-thm}
Let $\mathcal{C}$ be a nonempty class of {\QVB{F}{G}{H}s}, with $(F,G,H)$ an admissible triple, and let $\mathcal{B}$ be a class of modular bridges compatible with $\mathcal{C}$. Let $\Omega_\A = (\module{M},\inner{\cdot}{\cdot}{\module{M}},\A,\Lip_\A)$ and $\Omega_\B = (\module{N},\inner{\cdot}{\cdot}{\module{N}},\B,\Lip_\B)$ be two {\gQVB s} in $\mathcal{C}$. The following two assertions are equivalent:
\begin{enumerate}[I.]
\item $\modpropinquity{\mathcal{B}}(\Omega_\A,\Omega_\B) = 0$,
\item $\Omega_\A$ and $\Omega_\B$ are fully quantum isometric, i.e. there exists a *-isomorphism $\theta : \A \rightarrow \B$ and a linear continuous isomorphism $\Theta : \module{M}\rightarrow\module{N}$ such that:
\begin{enumerate}[1.]
\item $\Lip_\B\circ\theta = \Lip_\A$,
\item $\Theta(a\omega) = \theta(a)\Theta(\omega)$ for all $a\in\sa{\A}$, $\omega\in\module{M}$,
\item $\CDN_{\module{N}}\circ\Theta = \CDN_{\module{M}}$,
\item $\inner{\Theta(\cdot)}{\Theta(\cdot)}{\module{N}} = \theta\circ\inner{\cdot}{\cdot}{\module{M}}$.
\end{enumerate}
\end{enumerate}
\end{theorem}

\begin{proof}
For all $n\in\N$, let $\Gamma_n \in \trekset{\mathcal{T}}{\Omega_\A}{\Omega_\B}$ be given such that $\treklength{\Gamma_n} \leq \frac{1}{n+1}$.

We prove our theorem in a series of claim. 

\begin{claim}\label{single-convergence-claim}
If $\omega\in\dom{\CDN_{\module{M}}}$ and $l \geq \CDN_{\module{M}}(\omega)$, and if $f : \N \rightarrow\N$ is a strictly increasing function, then there exists a strictly increasing function $g : \N \rightarrow \N$ such that the sequence:
\begin{equation*}
\left(\targetsettrek{\Gamma_{f\circ g(n)}}{\omega}{l}\right)_{n\in\N}
\end{equation*}
converges to a singleton  for the Hausdorff distance $\Haus{\|\cdot\|_{\module{M}}}$.
\end{claim}

The sequence $\left(\targetsettrek{\Gamma_{f(n)}}{\omega}{l}\right)_{n\in\N}$ is a sequence of closed subsets of the compact $\modlip{l}{\Omega_\B}$ by Proposition (\ref{trek-targetset-compact-prop}). The hyperspace of all closed nonempty subsets of the compact set $(\modlip{l}{\Omega_\B},\KantorovichMod{\Omega_\B})$ is compact for the Hausdorff distance $\Haus{\KantorovichMod{\Omega_\B}}$.

Thus, $\left(\targetsettrek{\Gamma_{f(n)}}{\omega}{l}\right)_{n\in\N}$ admits a convergent subsequence $\left(\targetsettrek{\Gamma_{f\circ g(n)}}{\omega}{l}\right)_{n\in\N}$ converging for $\Haus{\KantorovichMod{\Omega_\B}}$; let $\alg{L}$ be its limit.

By Assertion (2) of Proposition (\ref{trek-targetset-morphism-prop}), we have $\diam{\alg{L}}{\KantorovichMod{\Omega_\B}} = 0$, i.e. it is a singleton.

Now, on the compact set $\modlip{r}{\Omega_\B}$, both $\KantorovichMod{\Omega_\B}$ and $\|\cdot\|_{\module{N}}$ are topologically equivalent by Proposition (\ref{modular-Kantorovich-prop}). Hence, $\left(\targetsettrek{\Gamma_{f\circ g(n)}}{\omega}{l}\right)_{n\in\N}$ converges to $\alg{L}$ for $\Haus{\|\cdot\|_{\module{N}}}$ by Lemma (\ref{Vietoris-lemma}).

\begin{claim}\label{unique-limit-claim}
Let us simplify our notations for this claim. Let $(A_n)_{n\in\N}$ and $(B_n)_{n\in\N}$ be two sequences of nonempty closed subsets in $\modlip{K}{\Omega_\B}$ for some $K > 0$ such that, for all $n\in\N$, we have $A_n \subseteq B_n$, and moreover:
\begin{equation*}
\lim_{n\rightarrow\infty} \diam{B_n}{\KantorovichMod{\Omega_\B}} = \lim_{n\rightarrow\infty\}} \diam{A_n}{\KantorovichMod{\Omega_\A}} = 0\text{.}
\end{equation*}

Then $(A_n)_{n\in\N}$ converges for $\Haus{\KantorovichMod{\Omega_\B}}$ if and only if $(B_n)_{n\in\N}$ converges for $\Haus{\KantorovichMod{\Omega_\B}}$ (noting the limit must be a singleton and it must be the same for both sequences).
\end{claim}

Assume first that $(A_n)_{n\in\N}$ converges for $\Haus{\KantorovichMod{\Omega_\B}}$ --- the limit being necessarily a singleton $\{\eta\}$, since the diameter of $A_n$ converges to $0$ as $n$ goes to infinity.

Let $\varepsilon > 0$. There exists $N \in \N$ such that for all $n\geq N$ we have $\Haus{\KantorovichMod{\Omega_\B}}(A_n,\{\eta\}) < \frac{\varepsilon}{2}$. There exists $N'\in\N$ such that for all $n\geq N'$, we have $\diam{B_n}{\KantorovichMod{\Omega_\B}} < \frac{\varepsilon}{2}$. Let $n \geq \max\{N,N'\}$. If $\omega\in B_n$ then there exists $\xi \in A_n$ such that $\KantorovichMod{\Omega_\B}(\omega,\xi) < \frac{\varepsilon}{2}$, and then we have $\KantorovichMod{\Omega_\B}(\xi,\eta) < \frac{\varepsilon}{2}$. Thus $\KantorovichMod{\Omega_\B}(\omega,\eta) < \varepsilon$. It then follows that $\Haus{\KantorovichMod{\Omega_\B}}(B_n,\{\eta\}) < \varepsilon$. This proves that $(B_n)_{n\in\N}$ converges to $\{\eta\}$.

Assume second that $(B_n)_{n\in\N}$ converges for $\Haus{\KantorovichMod{\Omega_\B}}$, again necessarily to a singleton $\{\eta\}$. It is then immediate that $\KantorovichMod{\Omega_\B}(\omega,\eta) \leq \Haus{\KantorovichMod{\Omega_\B}}(B_n,\{\eta\})$ for all $\omega\in A_n$ and thus in particular, $(A_n)_{n\in\N}$ converges to $\{\eta\}$ as well.

\begin{claim}\label{any-l-claim}
If $\omega\in\dom{\CDN_{\module{M}}}$ and $l \geq \CDN_{\module{M}}(\omega)$, and if $f : \N \rightarrow\N$ is a strictly increasing function such that
\begin{equation*}
\left(\targetsettrek{\Gamma_{f(n)}}{\omega}{l}\right)_{n\in\N}
\end{equation*}
converges to $\{\eta\}$ for the Hausdorff distance $\Haus{\KantorovichMod{\Omega_\B}}$, then for all $l' \geq \CDN_{\module{M}}(\omega)$, the sequence:
\begin{equation*}
\left(\targetsettrek{\Gamma_{f(n)}}{\omega}{l'}\right)_{n\in\N}
\end{equation*}
converges to $\{\eta\}$ for the Hausdorff distance $\Haus{\KantorovichMod{\Omega_\B}}$.
\end{claim}

We note that for all $l \geq l' \geq \CDN_{\module{M}}(\omega)$, we have:
\begin{equation*}
\targetsettrek{\Gamma_{f(n)}}{\omega}{l'} \subseteq \targetsettrek{\Gamma_{f(n)}}{\omega}{l}
\end{equation*}
for all $n\in\N$. Moreover, $\targetsettrek{\Gamma_{f(n)}}{\omega}{l'}, \targetsettrek{\Gamma_{f(n)}}{\omega}{l}\subseteq \modlip{l}{\Omega_\B}$ for all $n\in\N$. Last:
\begin{equation*}
\lim_{n\rightarrow\infty}\diam{\targetsettrek{\Gamma_{f(n)}}{\omega}{l'}}{\KantorovichMod{\Omega_\B}} = \lim_{n\rightarrow\infty} \diam{\targetsettrek{\Gamma_{f(n)}}{\omega}{l}}{\KantorovichMod{\Omega_\B}} = 0
\end{equation*}
by Assertion (2) of Proposition (\ref{trek-targetset-morphism-prop}). This allows us to conclude our claim using Claim (\ref{unique-limit-claim}).

\begin{claim}\label{diagonal-claim}
There exists $f : \N\rightarrow \N$ strictly increasing such that for all $\omega\in \dom{\CDN_{\module{M}}}$ and for all $l \geq \CDN_{\module{M}}(\omega)$, the sequence:
\begin{equation*}
\left(\targetsettrek{\Gamma_{f(n)}}{\omega}{l}\right)_{n\in\N}
\end{equation*}
converges to a singleton $\theta(\omega)$ for the Hausdorff distance $\Haus{\KantorovichMod{\Omega_\B}}$ (or equivalently for $\Haus{\|\cdot\|_{\module{N}}}$).
\end{claim}

We use a diagonal argument and Claim (\ref{single-convergence-claim}). As a compact metric space, the closed unit ball of $\dom{\CDN_{\module{M}}}$ is separable; however we can be a bit more precise in our case. For each $n\in\N$, let:
\begin{equation}\label{main-thm-trek-notation-eq}
\Gamma_n = \left(\bridge{\gamma_j^n} : j\in\{1,\ldots,K_n\} \right)\text{ for some $K_n \in \N\setminus\{0\}$.}
\end{equation}
Since the imprint of $\bridge{\gamma_1^n}$ is less than $\treklength{\Gamma_n}$, we note that $\bridgeanchors{\gamma_1^n}$ is a finite, $\frac{1}{n+1}$-dense subset of $(\modlip{1}{\Omega_\A}, \KantorovichMod{\Omega_\A})$. 

Let:
\begin{equation*}
\mathscr{S}_1 = \bigcup_{n\in\N} \bridgeanchors{\gamma_1^n}\text{.}
\end{equation*}
By construction, the set $\mathscr{S}$ is dense in $\modlip{1}{\Omega_\A}$ --- as well as countable.

For each $N\in\N$, the set $\mathscr{S}_N = N\cdot \mathscr{S}_1$ is dense in $\modlip{N}{\Omega_\A}$, since we note that the modular {\MongeKant} is homogeneous, namely $\KantorovichMod{\Omega_\A}(\omega, \eta) = N\KantorovichMod{\Omega_\A}(N^{-1}\omega, N^{-1}\eta)$ for all $\omega,\eta\in\module{M}$.

Thus, $\mathscr{S} = \bigcup_{N\in\N}\mathscr{S}_N$ is countable and dense in $\dom{\CDN_{\module{M}}}$. Let us write $\mathscr{S}$ as $\left\{ \omega_n : n \in \N \right\}$.

By Claim (\ref{single-convergence-claim}), there exists $g_0 : \N\rightarrow\N$ strictly increasing, such that the sequence:
\begin{equation*}
\left(\targetsettrek{\Gamma_{g_0(n)}}{\omega_0}{\CDN_{\module{M}}(\omega_0)}\right)_{n\in\N}
\end{equation*}
converges to a singleton $\{\Theta(\omega_0)\}$.

Assume now that for some $k\in\N$, we have built $g_0 : \N\rightarrow\N$, \ldots, $g_k : \N\rightarrow\N$ strictly increasing functions such that for all $j\in\{0,\ldots,k \}$, the sequence:
\begin{equation*}
\left(\targetsettrek{\Gamma_{g_0\circ\ldots\circ g_j(n)}}{\omega_j}{\CDN_{\module{M}}(\omega_j)}\right)_{n\in\N}
\end{equation*}
converges to a singleton $\{\Theta(\omega_j)\}$.

Applying our Claim (\ref{single-convergence-claim}) again, there exists $g_{k+1}$ strictly increasing, such that $\left(\targetsettrek{\Gamma_{g_0\circ\ldots\circ g_{k+1}(n)}}{\omega_{k+1}}{\CDN_{\module{M}}(\omega_{k+1})}\right)_{n\in\N}$ converges. Thus by induction, there exists strictly increasing functions $g_k$ for all $k\in\N$ such that:
\begin{equation*}
\left(\targetsettrek{\Gamma_{g_0\circ\ldots\circ g_j(n)}}{\omega_j}{\CDN_{\module{M}}(\omega_j)}\right)_{n\in\N}
\end{equation*}
converges to a singleton denoted by $\left\{\Theta(\omega_j)\right\}$ for all $j\in\N$.

Since subsequences of converging sequences have the same limit as the original sequence, we conclude that, if we set $f : n\in\N \rightarrow f(n) = g_0\circ\cdots\circ g_n(n) \in \N$, then $f$ is strictly increasing and for all $\omega\in\module{S}$, the sequence:
\begin{equation*}
\left(\targetsettrek{\Gamma_{f(m)}}{\omega}{\CDN_{\module{M}}(\omega)}\right)_{m\in\N}
\end{equation*}
converges to a singleton $\{\Theta(\omega)\}$.

By Claim (\ref{any-l-claim}), we note that for any $\omega\in\mathscr{S}$ and $l \geq \CDN_{\module{M}}(\omega)$, we also have:
\begin{equation*}
\lim_{n\rightarrow\infty} \Haus{\KantorovichMod{\Omega_\B}}\left(\targetsettrek{\Gamma_{f(m)}}{\omega}{l}, \{\Theta(\omega)\} \right) = 0 \text{.}
\end{equation*}

We now move to prove that $\Theta$ can be extended to $\dom{\CDN_{\module{M}}}$. Let $\omega\in\dom{\CDN_{\module{M}}}$. There exists $N\in\N$ such that $\omega\in\modlip{N}{\Omega_\A}$. We may as well assume that $N > 0$.

Let $\varepsilon > 0$. There exists $\omega_\varepsilon \in \module{S}_N$ such that $\KantorovichMod{\Omega_\A}(\omega, \omega_\varepsilon)  < \frac{\varepsilon\sqrt{2}}{12}$. Then by Proposition (\ref{trek-targetset-morphism-prop}), we have, for all $n\in\N$:
\begin{equation*}
\Haus{\KantorovichMod{\Omega_\B}}\left( \targetsettrek{\Gamma_n}{\omega}{N},  \targetsettrek{\Gamma_n}{\omega_\varepsilon}{N} \right) \leq \sqrt{2}\left(\frac{\sqrt{2} \varepsilon}{12} + (4 N + H ( 2 N, 1 ) )\frac{1}{n+1}\right) \text{.}
\end{equation*}

Let $N'\in\N$ be chosen so that $\frac{1}{n+1}\leq\frac{\sqrt{2}\varepsilon}{12(4 N + H(2 N,1))}$ for all $n\geq N'$.

Therefore, for all $n\geq N'$, we have:
\begin{equation*}
\Haus{\KantorovichMod{\Omega_\B}}\left( \targetsettrek{\Gamma_n}{\omega}{N},  \targetsettrek{\Gamma_n}{\omega_\varepsilon}{N} \right) \leq \frac{\varepsilon}{3} \text{.}
\end{equation*}

Since $\targetsettrek{\Gamma_{f(n)}}{\omega_\varepsilon}{N}$ converges for $\Haus{\KantorovichMod{\Omega_\B}}$, it is Cauchy, and thus there exists $N''\in\N$ such that for all $p,q \geq N''$ we have:
\begin{equation*}
\Haus{\KantorovichMod{\Omega_\B}}\left( \targetsettrek{\Gamma_{f(p)}}{\omega_\varepsilon}{N},  \targetsettrek{\Gamma_{f(q)}}{\omega_\varepsilon}{N} \right) \leq \frac{\varepsilon}{3} \text{.}
\end{equation*}
Thus if $p,q \geq \max\{N',N''\}$, we have:
\begin{equation*}
\begin{split}
\Haus{\KantorovichMod{\Omega_\B}}\left( \targetsettrek{\Gamma_{f(p)}}{\omega}{N},  \targetsettrek{\Gamma_{f(q)}}{\omega}{N} \right) &\leq \Haus{\KantorovichMod{\Omega_\B}}\left( \targetsettrek{\Gamma_{f(p)}}{\omega}{N},  \targetsettrek{\Gamma_{f(p)}}{\omega_\varepsilon}{N} \right) \\
&\quad + \Haus{\KantorovichMod{\Omega_\B}}\left( \targetsettrek{\Gamma_{f(p)}}{\omega_\varepsilon}{N},  \targetsettrek{\Gamma_{f(q)}}{\omega_\varepsilon}{N} \right) \\
&\quad + \Haus{\KantorovichMod{\Omega_\B}}\left( \targetsettrek{\Gamma_{f(q)}}{\omega_\varepsilon}{N},  \targetsettrek{\Gamma_{f(q)}}{\omega}{N} \right) \\
&\leq \frac{\varepsilon}{3} + \frac{\varepsilon}{3} + \frac{\varepsilon}{3} = \varepsilon \text{.}
\end{split}
\end{equation*}

Thus the sequence $\left(\targetsettrek{\Gamma_{f(m)}}{\omega}{N}\right)_{m\in\N}$ is Cauchy for $\KantorovichMod{\Omega_\B}$ inside the hyperspace of closed subsets of the compact $\modlip{N}{\Omega_\B}$. and thus converges by completeness. 

Again by Proposition (\ref{trek-targetset-morphism-prop}), the limit of $\left(\targetsettrek{\Gamma_{f(m)}}{\omega}{N}\right)_{m\in\N}$ for $\Haus{\KantorovichMod{\Omega_\B}}$ is a singleton which we denote by $\{\Theta(\omega)\}$. Moreover, by Claim (\ref{any-l-claim}), the sequence $\left(\targetsettrek{\Gamma_{f(m)}}{\omega}{l}\right)_{m\in\N}$ converges in $\Haus{\KantorovichMod{\Omega_\B}}$ to $\{\Theta(\omega)\}$ for any $l \geq\CDN_{\module{M}}(\omega)$. Last, since $\KantorovichMod{\Omega_\B}$ and $\|\cdot\|_{\module{N}}$ are topologically equivalent on $\modlip{K}{\Omega_\B}$ for any $K \geq 0$, the Hausdorff distances $\Haus{\KantorovichMod{\Omega_\B}}$ and $\Haus{\|\cdot\|_{\module{N}}}$ are also topologically equivalent by Lemma (\ref{Vietoris-lemma}), which concludes the proof of our claim.

\begin{claim}\label{CDN-continuity-claim}
For all $\omega\in\dom{\CDN_{\module{M}}}$ we have $\CDN_{\module{N}}(\Theta(\omega)) \leq \CDN_{\module{M}}(\omega)$.
\end{claim}

Let $\omega\in\dom{\CDN_{\module{M}}}$ and let $l = \CDN_{\module{M}}(\omega)$. By Claim (\ref{diagonal-claim}) and Lemma 
(\ref{Haus-pick-lemma}), if we pick $\eta_n \in \targetsettrek{\Gamma_{f(n)}}{\omega}{l}$ for all $n\in\N$, then $\lim_{n\rightarrow\infty} \|\eta_n - \Theta(\omega)\|_{\module{N}}= 0$. Since $\CDN_{\module{N}}$ is lower semi-continuous (as $\modlip{1}{\Omega_\B}$ is compact, hence closed, for the norm $\|\cdot\|_{\module{N}}$), we conclude that $\CDN_{\module{N}}(\Theta(\omega)) \leq l = \CDN_{\module{M}}(\omega)$.

\begin{claim}\label{theta-claim}
There exists a unital *-morphism $\theta : \A \rightarrow \B$ such that $\Lip_\B\circ\theta = \Lip_\A$ and a strictly increasing function $g : \N\rightarrow\N$ such that:
\begin{enumerate}[i.]
\item for all $a\in\dom{\Lip_\A}$ and for all $l \geq \Lip_\A(a)$, the sequence:
\begin{equation*}
\left(\targetsettrek{\Gamma_{g(n)}}{a}{l}\right)_{n\in\N}
\end{equation*}
converges to $\{\theta(a)\}$ for $\Haus{\|\cdot\|_\B}$ ;
\item for all $\omega\in\dom{\CDN_{\module{M}}}$ and any $l \geq \CDN_{\module{M}}(\omega)$, the sequence:
\begin{equation*}
\left(\targetsettrek{\Gamma_{g (n)}}{\omega}{l}\right)_{n\in\N}
\end{equation*}
converges to $\{\Theta(\omega)\}$ for $\Haus{\|\cdot\|_{\module{N}}}$.
\end{enumerate}
Moreover $\Lip_\B\circ\theta \leq \Lip_\A$.
\end{claim}

For all $n\in\N$, let $\Upsilon_n = (\Gamma_{f(n)})_\flat$. We note that $\treklength{\Upsilon_n} \leq \frac{1}{n+1}$ by construction. The construction of $\theta$ follows the same techniques as used in \cite[Theorem 5.13]{Latremoliere13}, which provides us with a *-isomorphism $\theta$ and some strictly increasing function $f_1 :\N\rightarrow\N$ such that, for all $\omega\in\dom{\Lip_\A}$, the sequence \begin{equation*}
\left(\targetsettrek{\Upsilon_{f_1(n)}}{a}{l}\right)_{n\in\N}
\end{equation*}
converges to $\{\theta(a)\}$ for $\Haus{\|\cdot\|_\B}$.

The rest of the claim follows if we set $g = f\circ f_1$.

\begin{claim}\label{isometry-claim}
For all $\omega,\omega'\in\dom{\CDN_{\module{M}}}$ we have:
\begin{equation*}
\theta\circ\inner{\omega}{\omega'}{\module{M}} = \inner{\Theta(\omega)}{\Theta(\omega')}{\module{N}}\text{.}
\end{equation*}
In particular, $\|\Theta(\omega)\|_{\module{N}} = \|\omega\|_{\module{M}}$.
\end{claim}

Let $\omega\in\dom{\CDN_{\module{M}}}$ and $l = \CDN_{\module{M}}(\omega)$. For each $n\in \N$ we pick $\eta_n \in \targetsettrek{\Gamma_{g(n)}}{\omega}{l}$ and $b_n \in \targetsettrek{\Gamma_{g(n)}}{\inner{\omega}{\omega}{\\module{M}}}{H(l,l)}$.

By Lemma (\ref{Haus-pick-lemma}) and Claim (\ref{theta-claim}), we conclude that $\lim_{n\rightarrow\infty} b_n = \theta(\inner{\omega}{\omega}{\module{M}})$ and $\lim_{n\rightarrow\infty} \eta_n = \Theta(\omega)$.

By Proposition (\ref{trek-targetset-morphism-prop}), for all $n\in\N$, we have:
\begin{multline*}
\left\| b_n - \inner{\eta_n}{\eta_n}{\module{N}} \right\|_\B \\ \leq \treklength{\Gamma_{g(n)}} \left( 8 l \sqrt{2} + H( 2l,2l ) + 6 H(l,l) + 2\sqrt{2} H( 2l,1 ) \right) \xrightarrow{n\rightarrow\infty} 0 \text{.}
\end{multline*} 
Therefore:
\begin{equation}\label{isometry-eq}
\begin{split}
\inner{\Theta(\omega)}{\Theta(\omega)}{\module{N}} &= \inner{\eta}{\eta}{\module{N}} \\
&= \lim_{n\rightarrow\infty} \inner{\eta_n}{\eta_n}{\module{N}} \\
&= \lim_{n\rightarrow\infty} b_n = \theta(\inner{\omega}{\omega}{\module{M}}) \text{.}
\end{split}
\end{equation}

Let now $\omega'\in\dom{\CDN_{\module{M}}}$. We note that:
\begin{equation*}
\inner{\omega}{\omega'}{\module{M}} = \frac{1}{4}\sum_{k = 0}^3 i^k \inner{\omega + i^k \omega'}{\omega + i^k}{\module{M}} \text{.}
\end{equation*}

The same polarizing identities hold in $\module{N}$. Thus, Equality (\ref{isometry-eq}) coupled with the above polarizing identities proves our claim.

\begin{claim}\label{morphism-claim}
For all $\omega,\omega'\in\module{M}$, $t\in \R$ and $a\in\A$:
\begin{equation*}
\Theta(\omega + t\omega') = \Theta(\omega) + t\Theta(\omega')
\end{equation*}
and
\begin{equation*}
\Theta(a\omega) = \theta(a)\Theta(\omega)\text{.}
\end{equation*}
Consequently, $\Theta$ is uniformly continuous with from $(\dom{\CDN_{\module{M}}},\|\cdot\|_{\module{M}})$ to  $(\dom{\CDN_{\module{N}}},\|\cdot\|_{\module{N}})$ and thus has a unique extension as a continuous module morphism, denoted in the same manner, from $(\module{M},\|\cdot\|_{\module{M}})$ to $(\module{N},\|\cdot\|_{\module{N}})$.
\end{claim}

Let $a\in\dom{\Lip_\A}$, $\omega\in\dom{\CDN_{\module{M}}}$ and $l \geq \max\{\Lip_\A(a), \CDN_{\module{M}}(\omega)\}$. Let $b_n \in \targetsettrek{\Gamma_{g(n)}}{a}{l}$ and $\eta_n \in\targetsettrek{\Gamma_{g(n)}}{\omega}{l}$

By Proposition (\ref{trek-targetset-morphism-prop}), we have:
\begin{equation*}
b_n \eta_n \in \targetsettrek{\Gamma_n}{a\omega}{G(\|a\|_\A + 2 l \treklength{\Gamma_{g(n)}}, l, l)} \text{.}
\end{equation*}
Thus $(b_n\eta_n)_{n\in\N}$ converges to $\Theta(a\omega)$ by Claim (\ref{diagonal-claim}) and Lemma (\ref{Haus-pick-lemma}). For the same reasons, $(b_n)_{n\in\N}$ converges to $\theta(a)$ and $(\eta_n)_{n\in\N}$ converges to $\Theta(\omega)$. By continuity of the left module action in $\module{N}$ and uniqueness of the limit:
\begin{equation*}
\theta(a) \Theta(\omega) = \Theta(a\omega)\text{.}
\end{equation*}

A similar reasoning applies to prove the linearity of $\Theta$. Let $\omega,\omega' \in \dom{\CDN_{\module{M}}}$ and $t\in\C$. Let $l \geq \max\{\CDN_{\module{M}}(\omega), \CDN_{\module{M}}(\omega')$. For all $n\in\N$, we let $\eta_n \in \targetsettrek{\Gamma_{g(n)}}{\omega}{l}$ and $\eta'\in\targetsettrek{\Gamma_{g(n)}}{\omega'}{l}$. By Proposition (\ref{trek-targetset-morphism-prop}) again, we have:
\begin{equation*}
\eta_n + t\eta'_n \in \targetsettrek{\Gamma_{g(n)}}{\omega+t\omega}{l + |t|l}\text{.}
\end{equation*}
By Lemma (\ref{Haus-pick-lemma}) and Claim (\ref{diagonal-claim}), we conclude that:
\begin{equation*}
\begin{split}
\Theta(\omega + t\omega') &= \lim_{n\rightarrow\infty} (\eta_n + t\eta'_n)\\
&= \lim_{n\rightarrow\infty}\eta_n + t \lim_{n\rightarrow\infty}\eta'_n \\
&= \Theta(\omega) + t\Theta(\omega)\text{.}
\end{split}
\end{equation*}

Now, $\inner{\Theta(\cdot)}{\Theta(\cdot)}{\module{N}} = \theta\circ\inner{\cdot}{\cdot}{\module{M}}$ so $\|\Theta(\cdot)\|_{\module{N}} = \|\cdot\|_{\module{M}}$. Thus $\Theta$, being linear, is continuous and of norm $1$. It thus can be extended to $\module{M}$ by continuity as a linear map. 

By continuity, we have that $\Theta(a\omega) = \theta(a)\Theta(\omega)$ for all $a\in\sa{\A}$ and $\omega\in\module{M}$. By linearity, it follows that $(\Theta,\theta)$ is a module morphism, as desired. This completes our claim.

\begin{claim}\label{iso-claim}
The map $\Theta$ is a continuous module isomorphism of norm $1$.
\end{claim}

Let $\Xi_n = \Gamma_{g(n)}^\ast$ for all $n\in\N$. By construction, $\treklength{\Xi_n} \leq \frac{1}{n+1}$. We therefore apply all the work we have done up to now with $\Xi_n$ in place of $\Gamma_n$ for all $n\in\N$. We thus obtain maps $h : \N\rightarrow\N$, $\vartheta : \B \rightarrow \A$ and $\Phi : \module{N}\rightarrow\module{M}$ such that $h$ is strictly increasing function, $(\vartheta,\Phi)$ is a module morphism of norm $1$ with the additional property that $\CDN_{\module{M}}(\Phi(\eta)) \leq \CDN_{\module{N}}(\eta)$, and such that:
\begin{equation*}
\lim_{n\rightarrow\infty} \Haus{\|\cdot\|_{\module{M}}}\left(\targetsettrek{\Gamma_{g(h(n))}^\ast}{\eta}{\CDN_{\module{N}}(\eta)}, \{\Phi(\eta)\} \right) = 0
\end{equation*}
and:
\begin{equation*}
\lim_{n\rightarrow\infty} \Haus{\|\cdot\|_{\B}}\left(\targetsettrek{\Gamma_{g(h(n))}^\ast}{b}{\Lip_\B(b)}, \{\vartheta(b)\} \right) = 0
\end{equation*}
for all $\eta\in\dom{\CDN_{\module{N}}}$ and $b \in \sa{\B}$.

Now, let $\omega\in\dom{\CDN_{\module{M}}}$ and $l > \CDN_{\module{M}}(\omega)$. We begin with a simple observation, owing to the symmetry in Definition (\ref{deck-seminorm-def}) of the deck seminorm of a bridge, which in turns implies symmetry in the notion of itinerary:
\begin{equation*}
\eta\in\targetsettrek{\Gamma_{g(h(n))}}{\omega}{l} \iff \omega\in\targetsettrek{\Gamma_{g(h(n))}^\ast}{\eta}{l}
\end{equation*}
for all $n\in\N$.

Let $\varepsilon > 0$. There exists $N\in\N$ such that for all $n\geq N$ we have:
\begin{enumerate}
\item $\frac{1}{n+1}\leq\frac{\varepsilon}{8l}$ so that $\max\left\{\treklength{\Gamma_{g(h(n))}}, \treklength{\Gamma_{g(h(n))}^\ast}\right\} \leq \frac{\varepsilon}{8l}$,
\item $\Haus{\|\cdot\|_{\module{N}}}\left(\targetsettrek{\Gamma_{g(h(n))}}{\omega}{l}, \{\Theta(\omega)\} \right) \leq \frac{\varepsilon}{2}$.
\end{enumerate}

Let $\zeta \in \targetsettrek{\Gamma_{g(h(n))}}{\omega}{l}$, so that in particular $\|\zeta - \Theta(\omega)\| \leq \frac{\varepsilon}{2}$. By symmetry, $\omega\in\targetsettrek{\Gamma_{g(h(n))}^\ast}{\zeta}{l}$. 

Now, let $\xi \in \targetsettrek{\Gamma_{g(h(n))}^\ast}{\Theta(\omega)}{l}$. We then compute, using Proposition (\ref{trek-targetset-morphism-prop}):
\begin{equation*}
\begin{split}
\left\|\omega - \Phi\circ\Theta(\omega)\right\|_{\module{M}} &\leq \left\|\omega - \xi \right\|_{\module{M}} + \left\| \xi - \Phi\circ\Theta(\omega) \right\|_{\module{M}} \\
&\leq 4 l \frac{\varepsilon}{8 l} + \frac{\varepsilon}{2} = \varepsilon \text{.}
\end{split}
\end{equation*}

Since $\varepsilon > 0$ is arbitrary, we conclude that $\omega = \Phi(\Theta(\omega))$. The same computation would establish that $\Phi\circ\Omega$ is the identity on $\dom{\CDN_{\module{N}}}$ as well. By continuity, $\Phi = \Theta^{-1}$. For similar reasons, $\vartheta = \theta^{-1}$.

We last note that $\CDN_{\module{M}} = \CDN_{\module{M}}\circ\Theta\circ\Theta^{-1} \leq \CDN_{\module{N}}\circ\Theta \leq \CDN_{\module{M}}$. Thus $\CDN_{\module{N}} \circ\Theta = \CDN_{\module{M}}$. This concludes our proof.
\end{proof}

We now turn to our first examples of convergence of {\gQVB s}. We begin with free modules, which gives us a chance to compare the modular Gromov-Hausdorff propinquity with the quantum Gromov-Hausdorff propinquity when working with {\gQqcms s} and their associated {\gQVB s} via Example (\ref{algebra-as-module-ex}).

\section{Convergence of Free modules}

We wish to answer the following natural question: if a sequence of {\gQqcms s} converge in the quantum propinquity, then, do free modules over them, seen as {\gQVB s} via Example (\ref{free-module-ex}), converge for the modular propinquity? One would certainly hope that the answer is positive, and we now prove it. An important side-product of this section is that the quantum propinquity and the modular propinquity restricted to the class of {\gQqcms s} --- using Example (\ref{algebra-as-module-ex}) --- are in fact equivalent.

The key step in our work is to lift a bridge between {\gQqcms s} to a bridge between a pair of free modules, in the manner given by the next lemma. In this section, we will employ the notations of Examples (\ref{algebra-as-module-ex}) and (\ref{free-module-ex}).

\begin{lemma}\label{bridge-free-module-lemma}
If $(\A,\Lip_\A)$ and $(\B,\Lip_\B)$ are two {\Qqcms{F}s}, $n\in\N\setminus\{0\}$ and $\gamma$ is some bridge from $\A$ to $\B$, then there exists a modular bridge $\gamma_{\mathrm{mod}}$ from $\left(\A^n,\inner{\cdot}{\cdot}{\A},\CDN_\A^n,\A,\Lip_\A\right)$ to $\left(\B^n,\inner{\cdot}{\cdot}{\B},\CDN_\B^n,\B,\Lip_\B\right)$ such that:
\begin{equation*}
\lambda \leq \bridgelength{\gamma_{\mathrm{mod}}} \leq 2n \left(\lambda + \sqrt{1 +  4 n \lambda\left( K(\lambda) + 2 + 2\lambda \right)} - 1\right)
\end{equation*}
where $K(\lambda) = F(1+2\lambda,1+2\lambda,1,1)$.
\end{lemma}

\begin{proof}
Let $\gamma = (\D,x,\pi_\A,\pi_\B)$ be a bridge from $\A$ to $\B$ of length $\lambda$.

Let $J_\A = \left\{ \omega \in \A^n : \CDN_\A^n(\omega) \leq 1 \right\}$ and $J_\B = \left\{ \eta \in \B^n : \CDN_\B^n(\eta)\leq 1 \right\}$. Let $J = J_\A \coprod J_\B$ be the disjoint union of $J_\A$ and $J_\B$. From now on, $J_\A$ and $J_\B$ are seen as their copies in $J$.

Let $j \in J_\A\subseteq J$, so that $j = \begin{pmatrix} a_1 \\ \vdots  \\ a_n \end{pmatrix}$ with $a_1,\ldots,a_n \in \A$ satisfying, by definition of $\CDN_\A^n$:
\begin{equation*}
\left\| \sum_{k=1}^n a_k a_k^\ast \right\|_\A \leq 1 \text{ and }\forall k \in \{1,\ldots,n\} \quad \max\{\Lip_\A(\Re a_k),\Lip_\A(\Im a_k)\} \leq 1
\end{equation*}
where $\Re c = \frac{c + c^\ast}{2}$ and $\Im c = \frac{c - c^\ast}{2i}$ for any $c\in \A$. 

In particular, for any $k\in\{1,\ldots,n\}$, we have:
\begin{equation*}
\|a_k\|_\A^2 = \|a_k a_k^\ast\|_\A \leq \|\inner{j}{j}{\A}\|_\A\leq 1\text{,}
\end{equation*}
and thus $\max\{ \|\Re a\|_\A , \|\Im a\|_\A \} \leq 1$.

For each $k\in\{1,\ldots,n\}$, we choose $e_k \in \targetsetbridge{\tau}{\Re a_k}{1}$ and $f_k \in \targetsetbridge{\tau}{\Im a_k}{1}$ and we let $c_k = e_k + i f_k$. 

We record that:
\begin{equation*}
\max\left\{ \Lip_\A(e_k),\Lip_\A(f_k) : k \in \{1,\ldots,n\} \right\} \leq 1
\end{equation*}
and:
\begin{equation*}
\max\left\{ \|e_k\|_\A,\|f_k\|_\A : k\in \{1,\ldots,n\} \right\} \leq 1 + 2\lambda \text{.}
\end{equation*}

A technical hurdle is that L-seminorms are only defined on Jordan-Lie subalgebra of self-adjoint elements, and yet we wish to estimate the Hilbert module norm of $\begin{pmatrix} c_1 \\ \vdots \\ c_n \end{pmatrix}$. Using the norm estimates for $c_1$,\ldots,$c_n$ is too rough, and we must employ a more subtle computation. This will now occupy our next efforts.

Noting that since $\Re c$ and $\Im c$ are self-adjoint for any $c\in\B$, so their multiplicative squares are also their Jordan product squares:
\begin{equation*}
\forall k \in \{1,\ldots,n\} \quad \max\left\{ \Lip_\B((\Re c_k)^2), \Lip_\B((\Im c_k)^2), \Lip_\B(\Lie{\Re c_k}{\Im c_k}) \right\} \leq F(1 + 2\lambda, 1 + 2\lambda, 1, 1) \text{.}
\end{equation*}
With this in mind, we set $K = F(1+2\lambda,1+2\lambda,1,1)$.

Moreover, we observe that for any $c\in \B$, we have:
\begin{equation*}
c c^\ast = (\Re c + i \Im c)(\Re c - i \Im c) = (\Re c)^2 + (\Im c)^2 + 2\Lie{\Re c}{\Im c} \text{.}
\end{equation*}
For any $c\in\A$, let $\theta_0(c) = (\Re c)^2$, $\theta_1(c) = 2\Lie{\Re c}{\Im c}$ and $\theta_2(c) = (\Im c)^2$, so that:
\begin{equation*}
c c^\ast = \theta_0(c) + \theta_1(c) + \theta_2(c) \text{.}
\end{equation*}

The point of this decomposition is that for $k\in\{1,\ldots,n\}$:
\begin{equation*}
\theta_0(c_k) = \Jordan{e_k}{e_k}\text{, }\theta_1(c_k) = 2\Lie{e_k}{f_k}\text{ and }\theta_2(c_k) = \Jordan{f_k}{f_k} \text{,}
\end{equation*}
so we relate $c_k c_k^\ast$ to the bridge $\gamma$ and its length.

Let $\varphi\in\StateSpace(\B)$. By Definition (\ref{bridge-height-def}) of the height of $\gamma$, there exists $\psi \in \StateSpace_1(\D|x)$ such that $\Kantorovich{\Lip_\A}(\varphi,\psi\circ\pi_\B) \leq \lambda$. 

We then compute:
\begin{align*}
0 &\leq \varphi\left(\sum_{k=1}^n c_k c_k^\ast \right) = \sum_{k=1}^n\varphi(c_k c_k^\ast) \\
&= \left|\sum_{k=1}^n \sum_{l=0}^2  \varphi(\theta_l(c_k)) \right|\\
&\leq \left|\sum_{k=1}^n\sum_{l=0}^2  \left(\varphi(\theta_l(c_k)) - \psi\circ\pi_\B(\theta_l(c_k))\right) \right| + \left|\sum_{k=1}^n\sum_{l=0}^2  \psi\circ\pi_\B(\theta_l(c_k)) \right| \\
&\leq 4 n \lambda K  + \left|\sum_{k=1}^n\sum_{l=0}^2  \psi\circ\pi_\B(\theta_l(c_k)) \right| \\
&\leq 4 K n \lambda + \left|\sum_{k=1}^n\sum_{l=0}^2  \psi(x \pi_\B(\theta_l(c_k))) \right| \\
&\leq 4 K n \lambda + \left|\sum_{k=1}^n\sum_{l=0}^2  \psi(x \pi_\B(\theta_l(c_k)) - \pi_\A(\theta_l(a_k)) x) \right| + \left|\sum_{k=1}^n\sum_{l=0}^2  \psi(x \pi_\A(\theta_l(a_k))) \right| \\
&\leq 4 K n \lambda + \sum_{k=1}^n\sum_{l=0}^2  \bridgenorm{\gamma}{\theta_l(a_k),\theta_l(c_k)} + \left|\sum_{k=1}^n\sum_{l=0}^2  \psi\circ\pi_\A(\theta_l(a_k))\right| \\ 
&\leq 4 K n \lambda + \sum_{k=1}^n\sum_{l=0}^2  \bridgenorm{\gamma}{\theta_l(a_k),\theta_l(c_k)} + \left|\psi\circ\pi_\A\left(\sum_{k=1}^n a_k a_k^\ast\right)\right| \\ 
&\leq 4 K n \lambda + \sum_{k=1}^n\sum_{l=0}^2  \bridgenorm{\gamma}{\theta_l(a_k),\theta_l(c_k)} + \left\|\sum_{k=1}^n a_k a_k^\ast\right\|_\A \\ 
&\leq 4 K n \lambda + \sum_{k=1}^n\sum_{l=0}^2  \bridgenorm{\gamma}{\theta_l(a_k),\theta_l(c_k)} + 1 \text{.}
\end{align*}
Using the Leibniz inequality for the bridge norm $\bridgenorm{\gamma}{\cdot,\cdot}$, we then estimate for all $k\in\{1,\ldots,n\}$:
\begin{align*}
\bridgenorm{\gamma}{\theta_0(a_k),\theta_0(c_k)} &= \bridgenorm{\gamma}{(\Re a_k)^2,(\Re c_k)^2} \\
&\leq \left(\|\Re a_k\|_\A + \|\Re c_k\|_\B\right) \bridgenorm{\gamma}{\Re a_k,\Re c_k} \\
&\leq 2(1 + \lambda)\lambda \text{,}
\end{align*}
and similarly:
\begin{equation*}
\bridgenorm{\gamma}{\theta_1(a_k),\theta_1(c_k)} \leq 4(1 + \lambda)\lambda \text{ and }\bridgenorm{\gamma}{\theta_2(a_k),\theta_2(c_k)} \leq 2(1 + \lambda)\lambda\text{.}
\end{equation*}

Consequently:
\begin{equation*}
\left\|\sum_{k=1}^n c_kc_k^\ast\right\|_\B \leq 1 + 4 n \lambda ( K + 2 + 2\lambda ) \text{.}
\end{equation*}

Set $Q = 4 n \left( K + 2 + 2\lambda \right)$.

Set $b_k = \frac{1}{\sqrt{ 1 + Q \lambda } } c_k$ for all $k\in\{1,\ldots,n\}$. We now set $\omega_j = j$ and $\eta_j = \begin{pmatrix} b_1 \\ \vdots \\ b_n \end{pmatrix}$. We note that by construction, $\|\inner{\eta_j}{\eta_j}{\B}\|_\B \leq 1$, and moreover $\CDN_\B^n(\eta_j) \leq 1$. In particular, we note that $\|b_k\|_\B \leq 1$ for all $k\in\{1,\ldots,n\}$.

We set $t_j = 1$ and $s_j = \sqrt{1+Q \lambda}$ as well for later bookkeeping.

We proceed symmetrically when $j \in J_\B$, where we set $\eta_j = j$ and $\omega_j$ constructed as above by rescaling elements from appropriate target sets for $\gamma^{-1}$; we record $t_j = \sqrt{1 + Q\lambda}$ and $s_j = 1$.

We now set:
\begin{equation*}
\gamma_{\mathrm{mod}} = \left( \D, x, \pi_\A, \pi_\B, (\omega_j)_{j\in J}, (\eta_j)_{j\in J} \right) \text{.}
\end{equation*}

By construction:
\begin{enumerate}
\item $\bridgeimprint{\gamma_{\mathrm{mod}}} = 0$,
\item $\bridgebasicreach{\gamma_{\mathrm{mod}}} = \bridgereach{\gamma}$,
\item $\bridgeheight{\gamma_{\mathrm{mod}}} = \bridgeheight{\gamma}$.
\end{enumerate}

We are left to compute the modular reach of $\gamma_{\mathrm{mod}}$.

Let $j,k \in J$. We write:
\begin{equation*}
\omega_j = \begin{pmatrix} a_1 \\ \vdots \\ a_n \end{pmatrix}\text{, }\omega_k = \begin{pmatrix} c_1 \\ \vdots \\ c_n \end{pmatrix} \text{ and }\eta_j = \begin{pmatrix} b_1 \\ \vdots \\ b_n \end{pmatrix}\text{, }\eta_k = \begin{pmatrix} d_1 \\ \vdots \\ d_n \end{pmatrix} \text{.}
\end{equation*}

We then have:
\begin{align*}
\|\pi_\B(a_l^\ast c_l) x - x \pi_\B(b_l^\ast d_l)\|_\D &\leq \|\pi_\B(a_l^\ast)\pi_\B(c_l) x - \pi_\B(a_l^\ast) x \pi_\B(d_l)\|_\D \\
&\quad + \|\pi_\B( a_l^\ast ) x \pi_\B(d_l) - x \pi_\B( b_l )\pi_\B( d_l )\|_\D\\
&\leq \|a_l\|_{\A} \|\pi_\B(c_l) x - x \pi_\B(d_l) \|_\D \\
&\quad + \|\pi_\B(a_l) x - x \pi_\B( b_l )\|_\D \|d_l\|_\B \\
&\leq \|a_l\|_{\A} \|\pi_\B(t_k c_l) x - x \pi_\B(s_k d_l) \|_\D \\
&\quad + \|\pi_\B(t_j a_l) x - x \pi_\B( s_j b_l )\|_\D \|d_l\|_\B \\
&\quad + \|(1-t_j)a_l\|_\A + \|(1-t_k)c_l\|_\A + \|(1-s_j)b_l\|_\B + \|(1-s_k)d_l\|_\B \\
&\leq \|\pi_\B(t_k c_l) x - x \pi_\B(s_k d_l) \|_\D \\
&\quad + \|\pi_\B(t_j a_l) x - x \pi_\B( s_j b_l )\|_\D \\
&\quad + 2\left(\sqrt{1 + Q\lambda} - 1\right) \\
&= \bridgenorm{\gamma}{t_k c_l, s_k d_l} + \bridgenorm{\gamma}{t_j a_l,s_j b_l} + 2\left(\sqrt{1+Q\lambda} - 1\right) \\
&\leq 2 \lambda + 2\left(\sqrt{1+Q\lambda} - 1\right) \text{.}
\end{align*}
Thus:
\begin{multline*}
\left\| \pi_\A\left(\inner{\begin{pmatrix} a_1 \\ \vdots \\ a_n \end{pmatrix}}{\begin{pmatrix} c_1 \\ \vdots \\ c_n  \end{pmatrix}}{\A}\right) x - x \pi_\B\left(\inner{\begin{pmatrix} b_1 \\ \vdots  \\ b_n \end{pmatrix}}{\begin{pmatrix} d_1 \\ \vdots \\ d_n \end{pmatrix}}{\B}\right)  \right\|_{\D} \\
\begin{split}
&= \left\| \sum_{j=1}^n \left( \pi_\A(a_j c_j) x - \pi_\B(b_j d_j)  \right)\right\|_{\D}\\
&\leq 2 \sum_{j=1}^n \left(\lambda+\sqrt{1+Q\lambda} - 1\right) = 2n\left(\lambda+\sqrt{1+Q\lambda} - 1\right) \text{.} 
\end{split}
\end{multline*}
Thus the modular reach of $\gamma_{\mathrm{mod}}$ is no more than $2n\left(\lambda+\sqrt{1+Q\lambda} - 1\right)$.

Therefore, the reach of $\gamma_{\mathrm{mod}}$ is no more than  $2n\left(\lambda+\sqrt{1+Q\lambda} - 1\right)$ and thus so is its length. 
\end{proof}

\begin{theorem}\label{free-module-thm}
If $(\A,\Lip_\A)$ and $(\B,\Lip_\B)$ are {\Qqcms{F}s} for some admissible function $F$, and if $n\in \N\setminus\{0\}$, then:
\begin{multline*}
\propinquity{}((\A,\Lip_\A),(\B,\Lip_\B)) \\ \leq \modpropinquity{}\left(\left(\A^n,\inner{\cdot}{\cdot}{\A},\CDN_\A^n,\A,\Lip_\A\right), \left(\B^n,\inner{\cdot}{\cdot}{\B},\CDN_\B^n,\A,\Lip_\B\right)\right) \\ 
\leq Q(\propinquity{}((\A,\Lip_\A),(\B,\Lip_\B))) \text{,}
\end{multline*}
where for all $\lambda \geq 0$:
\begin{equation*}
Q(\lambda) = 2n\lambda \left(1 + 4n F(1+2\lambda,1+2\lambda,1,1) + 2 + 2\lambda\right)
\end{equation*}
and where, for any {\gQqcms} $(\D,\Lip_\D)$, we set:
\begin{enumerate}
\item $\inner{\begin{pmatrix}
d_1 \\ \vdots \\ d_n
\end{pmatrix}}{\begin{pmatrix}
e_1 \\ \vdots \\ e_n
\end{pmatrix}}{\A} = \sum_{j=1}^n d_j e_j^\ast$ for all $\begin{pmatrix}
d_1 \\ \vdots \\ d_n
\end{pmatrix}$, $\begin{pmatrix}
e_1 \\ \vdots \\ e_n
\end{pmatrix}$ in $\D^n$,
\item $\CDN_\D^d\begin{pmatrix}
d_1 \\ \vdots \\ d_n
\end{pmatrix} = \max\left\{ \left\|\begin{pmatrix}
d_1 \\ \vdots \\ d_n
\end{pmatrix}\right\|_{\A^n}, \Lip_\D(\Re d_j), \Lip_\D(\Im d_j) : j\in \{1,\ldots n\} \right\}$ for all $\begin{pmatrix}
d_1 \\ \vdots \\ d_n
\end{pmatrix} \in \D^n$.
\end{enumerate}

In particular, the map $(\A,\Lip_\A) \mapsto (\A^n,\inner{\cdot}{\cdot}{\A},\CDN_\A^n,\A,\Lip_\A)$ is a continuous injection from the topology of the quantum propinquity to the topology of the modular propinquity.
\end{theorem}

\begin{proof}
Let $\Gamma$ be a trek from $(\A,\Lip_\A)$ to $(\B,\Lip_\B)$. Write:
\begin{equation*}
\Gamma = (\A_j,\Lip_j,\gamma^j,\A_{j+1},\Lip_{j+1})_{j\in\{1,\ldots,k\}}
\end{equation*}
for some bridges $\gamma^j$ ($j\in\{1,\ldots,k\}$) and some $k\in\N$. Let $\lambda_j$ be the length of $\gamma_j$ for all $j \in \{1,\ldots,k\}$.

Now, for each $j\in\{1,\ldots,k\}$, let $\gamma^j_{\mathrm{mod}}$ be the modular bridge given by Lemma (\ref{bridge-free-module-lemma}) applied to $\gamma^j$. It is then straightforward to check that $\Gamma_{\mathrm{mod}} = (\gamma^j_{\mathrm{mod}})_{j\in\{1,\ldots,k\}}$ is a modular trek from $\left(\A^n,\inner{\cdot}{\cdot}{\A},\CDN_\A^n,\A,\Lip_\A\right)$ to $\left(\B^n,\inner{\cdot}{\cdot}{\B},\CDN_\B^n,\A,\Lip_\B\right)$ whose length satisfies:
\begin{align*}
\treklength{\Gamma} &\leq \treklength{\Gamma_{\mathrm{mod}}} \\
&\leq \sum_{j=1}^k 2n \left(\lambda_j + \sqrt{1 + 4n\lambda_j(F(1+2\lambda_j,1+2\lambda_j,1,1) + 2 + 2\lambda_j)} - 1 \right)\\
&\leq \sum_{j=1}^k 2n \left(\lambda_j + 1 + 4n\lambda_j(F(1+2\lambda_j,1+2\lambda_j,1,1) + 2 + 2\lambda_j) - 1 \right)\\
&\leq \sum_{j=1}^k 2n \left(\lambda_j + 4n\lambda_j(F(1+2\lambda_j,1+2\lambda_j,1,1) + 2 + 2\lambda_j)  \right)\\
&\leq \sum_{j=1}^k 2n\lambda_j \left(1 + 4n (F(1+2\treklength{\Gamma},1+2\treklength{\Gamma},1,1) + 2 + 2\treklength{\Gamma})  \right)\\
&\leq 2n \treklength{\Gamma}\left(1 + 4n (F(1+2\treklength{\Gamma},1+2\treklength{\Gamma},1,1) + 2 + 2\treklength{\Gamma})  \right) \text{.}
\end{align*}

We thus conclude, by definition, that:
\begin{multline*}
\propinquity{}((\A,\Lip_\A),(\B,\Lip_\B)) \\ \leq \modpropinquity{}(\left(\A^n,\inner{\cdot}{\cdot}{\A},\CDN_\A^n,\A,\Lip_\A\right), \left(\B^n,\inner{\cdot}{\cdot}{\B},\CDN_\B^n,\A,\Lip_\B\right)) \\ \leq Q(\propinquity{}((\A,\Lip_\A),(\B,\Lip_\B)) \text{.}
\end{multline*}

This concludes our estimate. Now, we note that $Q$ is continuous at $0$ with $Q(0) = 0$ (note that $F$ is not assumed continuous, only increasing in the product order, but this is sufficient). Thus we get the continuity of the map $(\A,\Lip_\A) \mapsto (\A^n,\inner{\cdot}{\cdot}{\A},\CDN_\A^n,\A,\Lip_\A)$.

This concludes our proof.
\end{proof}

A simple yet reassuring consequence of Theorem (\ref{free-module-thm}) is that we have not introduced any new topology on the class of {\gQqcms s} with the modular propinquity, via the canonical Hilbert module structure carried on by any C*-algebra.

\begin{corollary}
The space of {\Qqcms{F}s} with the topology of the quantum propinquity is homeomorphic to its image by the map $(\A,\Lip_\A) \mapsto (\A,\inner{\cdot}{\cdot}{\A},\CDN_\A^1,\A,\Lip_\A)$.
\end{corollary}

\begin{proof}
This is the case $n=1$ of Theorem (\ref{free-module-thm}).
\end{proof}

Free modules are the direct sums, in the sense of Hilbert modules, of the canonical module associated with a C*-algebra. The next section discuss the matter of the continuity of the direct sum between general {\gQVB s} on certain well-behaved classes of {\gQqcms s}. We note that the D-norms constructed in this section, and the ones in the later section, differ in general: in this section, we constructed the D-norms from the underlying Lip-norms, while in the next section, we will be given D-norms on some modules and construct a new one on their direct sum.

\section{Iso-pivotal families and Direct sum of convergent modules}

Let $(\module{M},\inner{\cdot}{\cdot}{\module{M}})$ and $(\module{N},\inner{\cdot}{\cdot}{\module{N}})$ be two left Hilbert $\A$-module. The direct sum $\module{M}\oplus\module{N}$ is a left $\A$-module in an obvious manner, and a canonical $\A$-inner product on this direct sum is given by:
\begin{equation*}
\inner{(\omega,\eta)}{(\omega',\eta')}{\A} = \inner{\omega}{\omega'}{\module{M}} + \inner{\eta}{\eta'}{\module{N}}
\end{equation*}
for all $\omega,\omega'\in\module{M}$ and $\eta,\eta'\in\module{N}$.

Let us now assume we are given two {\gQVB s} $\Omega = (\module{M},\inner{\cdot}{\cdot}{\module{M}},\CDN_{\module{M}},\A,\Lip_\A)$ and $\Omega' = (\module{N},\inner{\cdot}{\cdot}{\module{N}},\CDN_{\module{N}},\A,\Lip_\A)$. For all $\omega\in\module{M}$ and $\eta\in\module{N}$, we set:
\begin{equation*}
\CDN(\omega,\eta) = \sqrt{\CDN_{\module{M}}(\omega)^2 + \CDN_{\module{N}}(\eta)^2}\text{.}
\end{equation*}

It is easy to check that $(\module{M}\oplus\module{N}, \inner{\cdot}{\cdot}{\A}, \CDN, \A,\Lip_\A)$ is a {\gQVB} as well. We will simply denote it by $\Omega \oplus \Omega'$.

We shall now prove a continuity result for direct sums of {\gQVB s}, under a uniformity assumption. In general, the construction of the inner product on the direct sum of two modules mixes up, in the base algebra, the contributions of each module to the reach of a given bridge. This complication is, however, not expected to often occur in practice. When working with modules over {\gQqcms s}, we envisage that modular bridges will be constructed out of bridges between the base quantum metric spaces --- indeed, this is what motivated our definition of the modular propinquity. Thus one may expect that the same bridge between the base quantum spaces may be reused for multiple modular bridges between different modules. This expectation is formalized in the following notion, which will serve as an hypothesis for our direct sum continuity result.

\begin{definition}\label{iso-pivotal-def}
Let $\Omega_{j,k} = (\module{M}_{j,k},\inner{\cdot}{\cdot}{j,k}, \CDN_{j,k}, \A_k, \Lip_k)$ be {\gQVB s} for $j,k \in \{1,2\}$. The family $((\Omega_{1,1},\Omega_{1,2}),(\Omega_{2,1},\Omega_{2,2}))$ is \emph{iso-pivotal} when for all $\varepsilon > 0$, there exist two modular treks $\Gamma^1$, from $\Omega_{1,1}$ to $\Omega_{1,2}$, and $\Gamma^2$, from $\Omega_{2,1}$ to $\Omega_{2,2}$, such that:
\begin{enumerate}[1.]
\item $\treklength{\Gamma^1} \leq \modpropinquity{}(\Omega_{1,1},\Omega_{1,2}) + \varepsilon$,
\item $\treklength{\Gamma^2} \leq \modpropinquity{}(\Omega_{2,1},\Omega_{2,2}) + \varepsilon$,
\item the basic treks $\Gamma_\flat^1$ and $\Gamma_\flat^2$ from $(\A_1,\Lip_1)$ to $(\A_2,\Lip_2)$ obtained from $\Gamma^1$ and $\Gamma^2$ are identical.
\end{enumerate}
\end{definition}

Informally, in an iso-pivotal family, one may find modular treks whose length is arbitrary close to the modular propinquity between each pair, and which differ only in the choice of the anchors and co-anchors. This notion can be extended in an obvious manner to classes of pairs of {\gQVB s} over various base spaces. 

With this concept, we have the following result:

\begin{theorem}\label{iso-pivotal-sum-thm}
Let $(\A,\Lip_\A)$ and $(\B,\Lip_\B)$ be two {\gQqcms s}. If $\Omega_{1,\A},\Omega_{2,\A}$ are {\gQVB s} over $(\A,\Lip_\A)$ and $\Omega_{1,\B},\Omega_{2,\B}$ are {\gQVB s} over $(\B,\Lip_\B)$ such that $((\Omega_{1,\A},\Omega_{1,\B}),(\Omega_{2,\A},\Omega_{2,\B}))$ is iso-pivotal, then:
\begin{equation*}
\modpropinquity{}((\Omega_{1,\A}\oplus\Omega_{2,\A}),(\Omega_{1,\B}\oplus\Omega_{2,\B})) \leq \modpropinquity{}(\Omega_{1,\A},\Omega_{1,\B}) + \modpropinquity{}(\Omega_{2,\A},\Omega_{2,\B})) \text{.}
\end{equation*}
\end{theorem}

\begin{proof}
We begin by setting our notations: let $\Omega_{j,k} = (\module{M}_{j,k},\inner{\cdot}{\cdot}{j,k}, \CDN_{j,k}, k)$ for $j \in \{1,2\}$ and $k \in \{ (\A,\Lip_\A), (\B,\Lip_\B)\}$.

Let $\gamma_1$ be a modular bridge from $\Omega_{1,\A}$ to $\Omega_{1,\B}$ and $\gamma_2$ be a modular bridge from $\Omega_{2,\A}$ to $\Omega_{2,\B}$. We assume that $\gamma_1$ and $\gamma_2$ are given as:
\begin{equation*}
\gamma_1 = (\D,x,\pi_\A,\pi_\B, (\omega_j)_{j\in J_1}, (\eta_j)_{j\in J_1})
\end{equation*}
and
\begin{equation*}
\gamma_2 = (\D,x,\pi_\A,\pi_\B,(\omega'_j)_{j\in J_2}, (\eta'_j)_{j\in J_2})\text{.}
\end{equation*}
Of course, modular bridges between $\Omega_{1,\A}$ and $\Omega_{1,\B}$, and between $\Omega_{2,\A}$ and $\Omega_{2,\B}$, may not share basic bridge; however, we will conclude this theorem using the iso-pivotal hypothesis, and thus this choice of bridge will always be possible, and sufficient for our purpose. 

We first show that we may as well assume $J_1=J_2$. Pick $j_\ast \in J_1$ and $k_\ast \in J_2$. Set $J = J_1 \coprod J_2$. If $j \in J_1 \setminus J_2$, we set $\omega'_j = \omega_{k_\ast}$ and $\eta'_j = \eta_{k^\ast}$. If $j \in J_2 \setminus J_1$, we set $\omega_j = \omega_{j_\ast}$ and $\eta_j = \eta_{j_\ast}$.

With this procedure, we note that, for instance, $(\D,x,\pi_\A,\pi_\B,(\omega_j)_{j\in J}, (\eta_j)_{j \in J})$ has the same length and the same basic bridge as $\gamma_1$. The same holds for $\gamma_2$. Thus, without loss of generality, we let $J = J_1 = J_2$. 

Now, set:
\begin{equation*}
\Xi_\A = \left\{ (\omega_j, \omega'_k)_{j,k \in J} : \|\inner{\omega_j}{\omega'_k}{\A}\|_\A \leq 1\right\}
\end{equation*}
and
\begin{equation*}
\Xi_\B = \left\{ (\eta_j, \eta'_k)_{j,k \in J} : \|\inner{\eta_j}{\eta'_k}{\B}\|_\B \leq 1\right\}\text{.}
\end{equation*}

Let now:
\begin{equation*}
\gamma_1 \vee \gamma_2 = (\D, x, \pi_\A, \pi_\B, \Xi_\A, \Xi_\B )\text{.}
\end{equation*}

Note that $\gamma_1\vee\gamma_2$ has, once again, the same basic bridge as $\gamma_1$ and $\gamma_2$. 
It is thus straightforward that:
\begin{equation*}
\bridgeheight{\gamma_1\vee\gamma_2} = \bridgeheight{\gamma_1} = \bridgeheight{\gamma_2}\text{,}
\end{equation*}
and:
\begin{equation*}
\bridgebasicreach{\gamma_1\vee\gamma_2} = \bridgebasicreach{\gamma_1} = \bridgebasicreach{\gamma_2}\text{.}
\end{equation*}

Let $\omega\in\modlip{1}{\Omega_{1,\A}}$ and $\omega'\in\modlip{1}{\Omega_{2,\A}}$. By Definition (\ref{bridge-imprint-def}) of the imprint of a bridge, there exist $j,k \in \{1,\ldots,n\}$ such that $\KantorovichMod{\Omega_\A}(\omega,\omega_j) \leq \bridgeimprint{\gamma_1}$ and $\KantorovichMod{\Omega_\A}(\omega',\omega'_k)\leq \bridgeimprint{\gamma_2}$. Therefore for all $(\omega,\omega') \in \Xi_\A$, noting that $\max\{\|\omega\|_{\Omega_{1,\A}},\|\omega\|_{\Omega_{2,\A}}\}\leq 1$, we then have:

\begin{multline*}
\KantorovichMod{\Omega_{1,\A}\oplus\Omega_{2,\A}}((\omega,\omega'), (\omega_j,\omega'_k))\\
\begin{split}
&= \sup\left\{ \|\inner{(\omega,\omega') - (\omega_j,\omega_k)}{(\eta,\eta')}{\A}\|_\A : \CDN(\eta,\eta') \leq 1\right\} \\
&\leq \sup\left\{ \|\inner{\omega - \omega_j}{\eta}{\module{M}_{1,\A}} + \inner{\omega'-\omega_k}{\eta'}{\module{M}_{2,\A}}\|_\A : \CDN(\eta,\eta') \leq 1\right\} \\
&\leq \sup\left\{ \|\inner{\omega}{\eta}{\module{M}_{1,\A}} - \inner{\omega_j}{\eta}{\module{M}_{1,\A}}\|_\A : \CDN_{1,\A}(\eta)\leq 1\right\} \\
&\quad + \sup\left\{ \|\inner{\omega'}{\eta'}{\module{M}_{2,\A}} - \inner{\omega_k'}{\eta'}{\module{M}_{2,\A}}\|_\A : \CDN_{2,\A}(\eta')\leq 1\right\} \\
&\leq \KantorovichMod{\Omega_{1,\A}}(\omega,\omega_j) + \KantorovichMod{\Omega_{2,\A}}(\omega',\omega'_k))\\
&\leq \bridgeimprint{\gamma} + \bridgeimprint{\gamma'} \text{.}
\end{split}
\end{multline*}

The same argument can be made in $\Omega_{2,\B}\oplus\Omega_{2,\B}$. Thus:
\begin{equation*}
\bridgeimprint{\gamma_1\vee\gamma_2} \leq \bridgeimprint{\gamma_1} + \bridgeimprint{\gamma_2} \text{.}
\end{equation*}

Last, let $j,k \in J$. By Definition (\ref{modular-reach-def}) of the modular reach, we have:
\begin{equation*}
\left\|\inner{\omega}{\omega_j}{\module{M}_{1,\A}}x - x\inner{\eta}{\eta_j}{\module{M}_{1,\B}}\right\|_{\D} \leq \bridgemodularreach{\gamma_1} 
\end{equation*}
and
\begin{equation*}
\left\|\inner{\omega'}{\omega_k}{\module{M}_{2,\A}}x - x\inner{\eta}{\eta_k}{\module{M}_{2.\B}}\right\|_{\D} \leq \bridgemodularreach{\gamma_2}\text{.}
\end{equation*}

We thus compute:
\begin{multline*}
\decknorm{\gamma_1\vee\gamma_2}{(\omega_j,\omega'_k),(\eta_j,\eta'_k)} \\
\begin{split}
&= \max_{n,m \in J} \left\| \pi_\A \left(\inner{(\omega_j,\omega_k')}{(\omega_n,\omega_m')}{\A} \right) x - x \pi_\B\left( \inner{(\eta_j,\eta'_k)}{(\eta_n,\eta'_m)}{\B} \right) \right\|_\D \\
&= \max_{n \in J} \left\| \pi_\A \left(\inner{\omega_j}{\omega_n}{\module{M}_{1,\A}} \right) x - x \pi_\B\left( \inner{\eta_j}{\eta_n}{\module{M}_{1,\B}} \right) \right\|_\D \\
&\quad + \max_{m \in J} \left\| \pi_\A \left(\inner{\omega'_k}{\omega'_m}{\module{M}_{2,\A}} \right) x - x \pi_\B\left( \inner{\eta'_k}{\eta'_m}{\module{M}_{2,\B}} \right) \right\|_\D \\
&\leq \decknorm{\gamma_1}{\omega_j,\eta_j} + \decknorm{\gamma_2}{\omega_k,\eta_k}\\
&\leq \bridgemodularreach{\gamma_1} + \bridgemodularreach{\gamma_2} \text{.}
\end{split}
\end{multline*}

We have therefore proven:
\begin{equation*}
\bridgelength{\gamma_1\vee\gamma_2} \leq \bridgelength{\gamma_1} + \bridgelength{\gamma_2}\text{.}
\end{equation*}

Now, let $\varepsilon > 0$. As we work with an iso-pivotal family, there exist two treks $\Gamma^1$, from $\Omega_{1,\A}$ to $\Omega_{2,\B}$, and $\Gamma^2$, from $\Omega_{2,\A}$ to $\Omega_{2,\B}$, such that at once:
\begin{enumerate}
\item $\treklength{\Gamma^1} \leq \modpropinquity{}(\Omega_{1,\A},\Omega_{1,\B}) + \frac{\varepsilon}{2}$,
\item $\treklength{\Gamma^2} \leq \modpropinquity{}(\Omega_{2,\A},\Omega_{2,\B}) + \frac{\varepsilon}{2}$,
\item $\Gamma^1_\flat = \Gamma^2_\flat$.
\end{enumerate}

We set $\Gamma\vee\Gamma' = (\gamma_1^j\vee\gamma_2^j)$; using our work in the first part of this proof and a trivial induction, we conclude that $\Gamma_1\vee\Gamma_2$ is a modular trek from $\Omega_{1,\A}\oplus\Omega_{2,\A}$ to $\Omega_{1,\B}\oplus\Omega_{2,\B}$ such that:
\begin{equation*}
\treklength{\Gamma_1\vee\Gamma_2} \leq \treklength{\Gamma_1} + \treklength{\Gamma_2} \varepsilon \leq \modpropinquity{}(\Omega_{1,\A},\Omega_{1,\B}) + \modpropinquity{}(\Omega_{2,\A}, \Omega_{2,\B})) + \varepsilon \text{.}
\end{equation*}
By Definition (\ref{modular-propinquity-def}) of the modular propinquity, we thus conclude that for all $\varepsilon > 0$, the following holds:
\begin{equation*}
\modpropinquity{}(\Omega_{1,\A}\oplus\Omega_{2,\A}, \Omega_{1,\B}\oplus\Omega_{2,\B}) \leq \modpropinquity{}(\Omega_{1,\A},\Omega_{1,\B}) + \modpropinquity{}(\Omega_{2,\A}, \Omega_{2,\B})) + \varepsilon \text{.}
\end{equation*}

Our theorem is now proven.
\end{proof}

We provide an extension of the concept of iso-pivotal class to classes of {\gQVB s}, by slight abuse of terminology.

\begin{definition}
Let $\mathfrak{L}$ be a class of {\Qqcms{F}s}. A class $\mathfrak{M}$ of {\QVB{F}{G}{H}s} over elements of $\mathfrak{L}$ is \emph{iso-pivotal} when, for any two $(\A,\Lip_\A)$ and $(\B,\Lip_\B)$ in $\mathcal{L}$, for every $\varepsilon > 0$, for any two $\Omega_\A, \Omega_\A' \in \mathfrak{M}$ over $(\A,\Lip_\A)$ and $\Omega_\B,\Omega_\B' \in \mathfrak{M}$ over $(\B,\Lip_\B)$, the family $((\Omega_\A,\Omega_\B), (\Omega'_\A,\Omega'_\B))$ is iso-pivotal.
\end{definition}

We thus can state:
\begin{corollary}
The direct sum is continuous on any iso-pivotal class of {\gQVB s}.
\end{corollary}

\begin{proof}
This follows immediately from Theorem (\ref{iso-pivotal-sum-thm}).
\end{proof}

To fully reflect the potential of the modular propinquity, we now move toward a much more involved example of convergence for modules. The next section deals with non-free, finitely generated projective modules over quantum $2$-tori.

\section{Heisenberg Modules}

We conclude this paper with our main example of convergence for modules over quantum $2$-tori. The following theorem requires an extensive proof, which is contained in the companion paper \cite{Latremoliere17a}. We just lay out some notation first and we refer to \cite{Latremoliere17a} for a much more extensive description of this fundamental example. For any $\theta\in\R$, the quantum $2$-torus $\qt{\theta}$ is the universal C*-algebra generated by two unitaries $U$, $V$ such that $UV = \exp(2i\pi\theta)VU$. The C*-algebra $\qt{\theta}$ carries an action of the $2$-torus $\T^2$ denoted by $\beta_\theta$ and such that for all $(z_1,z_2) \in \T^2$, we have $\beta_\theta^{z_1,z_2}(U) = z_1U$ and $\beta_\theta^{z_1,z_2}V = z_2 V$. This action is known as the dual action of $\T^2$ on $\qt{\theta}$. We refer to \cite{Rieffel81} for an overview of this important family of C*-algebras.

Connes proposed in \cite{Connes80} the following construction of non-free, finitely generated modules over $\qt{\theta}$. Let $p\in\N$, $q\in\N\setminus\{0\}$ and $d\in q\N\setminus\{0\}$. Denote $\theta-\frac{p}{q}$ by $\eth$. The space $L^2(\R)$ carries actions of the Heisenberg group $\HeisenbergGroup = \left\{ \begin{bmatrix} 1 & x & z \\ 0 & 1 & y \\ 0 & 0 & 1 \end{bmatrix} : x,y,z \in \R^3 \right\}$ --- in fact, it carries all non-trivial irreducible unitary representations of this group. Let $\alpha_{\eth,1}$ be the representation given by:
\begin{equation}\label{alpha-eq}
\alpha_{\eth,1}^{x,y,z} \omega : s \in \R\longmapsto \exp\left(2i\pi\left(\eth z + sx \right)\right)\omega(s + \eth y ) \text{,}
\end{equation}
for all $\omega\in L^2(\R)$ and where we identify $\HeisenbergGroup$ with $\R^3$ as a set, for notational convenience. The map $\sigma_{\eth,1} : (x,y) \in \R^2 \longmapsto \sigma_{\eth,1}^{x,y} = \alpha_{1,\eth}^{x,y,\frac{xy}{2}}$ is a projective representation of $\R^2$, and in fact all irreducible projective representation of $\R^2$ which are not regular representation are unitarily equivalent to $\sigma_{\eth,1}$ for some $\eth$.

The \emph{Heisenberg module} $\HeisenbergMod{\theta}{p,q,d}$ is the module over $\qt{\theta}$ defined as follows. Let $\rho_{p,q,d}$ be a projective action of $\Z_q^2$ on $\C^d$ associated with the multiplier:
\begin{equation*}
(([n],[m]),([n'],[m'])) \in \Z_q^2\times \Z_q^2 \longmapsto \exp\left(i\pi p \frac{n m' - n' m}{q}\right) \text{,}
\end{equation*}
where $\bigslant{\Z}{q\Z} = \Z_q$, and for $n \in \Z$, we denote the class of $n$ in $\Z_q$ by $[n]$ and $[m]$ (we easily check that the multiplier is indeed well-defined).

For all $n,m \in \Z^2$, we then set:
\begin{equation*}
\varpi_{p,q,\eth,d}^{n,m} = \alpha_{\eth,1}^{n,m,0}\otimes\rho_{p,q}^{[n],[m]} \text{.}
\end{equation*}

For all $n,m\in\Z$, the map $\varpi_{p,q,\eth,d}^{n,m}$ is a unitary of $L^2(\R) \otimes \C^d$, and moreover $\varpi_{p,q,\eth,d}$ is a projective representation of $\Z^2$ for the multiplier:
\begin{equation*}
((n,m),(n',m')) \in \Z^2\times \Z^2 \longmapsto \exp\left(i\pi\theta(n m' - n' m)\right) \text{.}
\end{equation*}

By universality, the Hilbert space $L^2(\R)\otimes \C^d$ is a module over $\qt{\theta}$, with, in particular, for all $f\in \ell^1(\Z^2)$ and $\xi \in L^2(\R,\C^d) = L^2(\R)\otimes \C^d$:
\begin{equation*}
f \xi = \sum_{n,m\in\Z} f(n,m) \varpi_{p,q,\eth,d}^{n,m}\xi\text{.}
\end{equation*}

Let $\module{S}_\theta^{p,q,d} \subseteq L^2(\R)\otimes \C^d$ be the space of $\C^d$-valued Schwarz functions over $\R$. For all $\xi,\omega \in \module{S}_\theta^{p,q,d}$, define $\inner{\xi}{\omega}{\HeisenbergMod{\theta}{p,q,d}}$ as the function in $\ell^1(\Z^2)$ given by:
\begin{equation*}
\inner{\xi}{\omega}{\HeisenbergMod{\theta}{p,q,d}} : (n,m) \in \Z^2 \longmapsto \inner{\varpi_{p,q,\eth,d}^{n,m}\xi}{\omega}{L^2(\R)\otimes E} \text{.}
\end{equation*}

The \emph{Heisenberg module $\HeisenbergMod{\theta}{p,q,d}$} is the completion of $\module{S}_\theta^{p,q,d}$ for the norm associated with the $\qt{\theta}$-inner product $\inner{\cdot}{\cdot}{\HeisenbergMod{\theta}{p,q,d}}$.

Connes constructed these modules in \cite{Connes80}, and Rieffel proved in \cite{Rieffel83} that all finitely generated modules over the quantum $2$-tori are the direct sum of an Heisenberg module and a free module. In \cite{ConnesRieffel87}, the action $\alpha_{\eth,d}$ of $\HeisenbergGroup$ on $\HeisenbergMod{\theta}{p,q,d}$ is used to construct a connection $\nabla$ on $\HeisenbergMod{\theta}{p,q,d}$ as follows. The Heisenberg Lie algebra is generated by three vectors $X=\begin{bmatrix} 0 & 1 & 0 \\ 0 & 0 & 0 \\ 0 & 0 & 0\end{bmatrix}$, $Y = \begin{bmatrix} 0 & 0 & 0 \\ 0 & 0 & 1 \\ 0 & 0 & 0 \end{bmatrix}$ and $Z = \begin{bmatrix} 0 & 0 & 1 \\ 0 & 0 & 0 \\ 0 & 0 & 0 \end{bmatrix}$ such that $Z$ is central and $[X,Y] = Z$. For all $x,y \in \R$, and for any $\omega \in \module{S}_\theta^{p,q,d}  \subseteq \HeisenbergMod{\theta}{p,q,d}$, we set:
\begin{equation*}
\nabla_{xX + yY} \omega = \lim_{t\rightarrow 0} \frac{\alpha_{\eth,1}^{\exp(txX + tyY)}\omega-\omega}{t} \text{.}
\end{equation*}
Notably, $\nabla$ is a metric connection for $\inner{\cdot}{\cdot}{\HeisenbergMod{\theta}{p,q,d}}$ and it plays a central role in the Yang-Mills problem over the quantum $2$-torus.

Now, if we endow $\R^2$ with any norm, we may regard for all $\omega \in \HeisenbergMod{\theta}{p,q,d}$ the operator $(x,y) \in \R^2 \mapsto \nabla_{x X + y Y}\omega$ as a linear map between normed vector spaces. As such, it too has an operator norm $\opnorm{\nabla \omega}{}{}$. The natural candidate for our D-norm on $\HeisenbergMod{\theta}{p,q,d}$ is thus a lower semicontinuous norm on $\HeisenbergMod{\theta}{p,q,d}$ which agrees, for all $\omega \in \module{S}_\theta^{p,q,d} $, with $\max\left\{\|\omega\|_{\HeisenbergMod{\theta}{p,q,d}},\opnorm{\nabla \omega}{}{} \right\}$. With these notations, we indeed prove in \cite{Latremoliere17a} that:

\begin{theorem}[{\cite{Latremoliere17a}}]
Let $\|\cdot\|$ be a norm on $\R^2$. For all $\theta \in \R$, we equip the quantum torus $\qt{\theta}$ with the L-seminorm:
\begin{equation*}
\Lip_\theta : a \in \sa{\A} \mapsto \sup\left\{ \frac{\left\| \beta_\theta^{\exp(i x),\exp(i y)}a - a\right\|_{\qt{\theta}}}{\|(x,y)\|} : (x,y) \in \R^2\setminus\{0\} \right\} \text{.}
\end{equation*}

For all $p\in\Z$, $q\in\N\setminus\{0\}$ and $d\in q\N\setminus\{0\}$ and for all $\theta\in \R\setminus\left\{\frac{p}{q}\right\}$, setting $\eth = \theta - \frac{p}{q}$, we endow the Heisenberg module $\HeisenbergMod{\theta}{p,q,d}$ with norm:
\begin{multline*}
\CDN_\theta^{p,q,d} : \xi \in \HeisenbergMod{\theta}{p,q,d} \mapsto \\
\sup\left\{ \|\xi\|_{\HeisenbergMod{\theta}{p,q,d}}, \frac{\left\|\exp\left(i\pi \eth x y \right)\alpha_{\eth,d}^{x,y,\frac{xy}{2}}\xi - \xi \right\|_{\HeisenbergMod{\theta}{p,q,d}}}{2\pi\eth \|(x,y)\| } : (x,y) \in \R^2\setminus\{0\} \right\} \text{.}
\end{multline*}

The norm $\CDN_{\theta}^{p,q,d}$ is a D-norm on $\HeisenbergMod{\theta}{p,q,d}$, i.e. $\left(\HeisenbergMod{\vartheta}{p,q,d}, \inner{\cdot}{\cdot}{\HeisenbergMod{\vartheta}{p,q,d}}\right)$ is a Leibniz {\gQVB}. Moreover, this D-norm agrees with $\max\left\{\|\omega\|_{\HeisenbergMod{\theta}{p,q,d}},\opnorm{\nabla \omega}{}{} \right\}$ on the domain of the Heisenberg connection $\nabla$.

Let $p \in \Z$ and $q \in \N\setminus\{0\}$. Let $d \in q\N\setminus\{0\}$. For any $\theta \in \R\setminus\left\{\frac{p}{q}\right\}$, we have:
\begin{multline*}
\lim_{\vartheta\rightarrow\theta} \modpropinquity{}\left(\left(\HeisenbergMod{\vartheta}{p,q,d}, \inner{\cdot}{\cdot}{\HeisenbergMod{\vartheta}{p,q,d}}, \CDN_\vartheta^{p,q,d}, \qt{\vartheta}, \Lip_{\vartheta} \right),\right. \\
\left. \left(\HeisenbergMod{\theta}{p,q,d}, \inner{\cdot}{\cdot}{\HeisenbergMod{\theta}{p,q,d}}, \CDN_\theta^{p,q,d}, \qt{\theta}, \Lip_{\theta}\right) \right) = 0\text{.}
\end{multline*}
\end{theorem}

\bibliographystyle{amsplain}
\bibliography{../thesis}

\providecommand{\bysame}{\leavevmode\hbox to3em{\hrulefill}\thinspace}
\providecommand{\MR}{\relax\ifhmode\unskip\space\fi MR }
\providecommand{\MRhref}[2]{%
  \href{http://www.ams.org/mathscinet-getitem?mr=#1}{#2}
}
\providecommand{\href}[2]{#2}
\begin{thebibliography}{10}

\bibitem{Latremoliere15d}
{K}. {A}guilar and {F}. {L}atr{\'e}moli{\`e}re, \emph{Quantum ultrametrics on
  af algebras and the {G}romov--{H}ausdorff propinquity}, Studia Mathematica
  \textbf{231} (2015), no.~2, 149--194, ArXiv: 1511.07114.

\bibitem{burago01}
{D}. {B}urago, {Y}. {B}urago, and {S}. {I}vanov, \emph{A course in metric
  geometry}, Graduate Texts in Mathematics, vol.~33, American Mathematical
  Society, 2001.

\bibitem{Rieffel15b}
{M}. {C}hrist and {M.}~{A.} {R}ieffel, \emph{Nilpotent group
  {$C^\ast$-algebras}-algebras as compact quantum metric spaces}, Submitted
  (2015), 22 pages, ArXiv: 1508.00980.

\bibitem{Connes80}
{A}. {C}onnes, \emph{{C*}--alg{\`e}bres et g{\'e}om{\'e}trie differentielle},
  {C}. {R}. de l'academie des Sciences de Paris (1980), no.~series A-B, 290.

\bibitem{Connes89}
A.~{C}onnes, \emph{Compact metric spaces, {F}redholm modules and
  hyperfiniteness}, Ergodic Theory and Dynamical Systems \textbf{9} (1989),
  no.~2, 207--220.

\bibitem{Connes}
\bysame, \emph{Noncommutative geometry}, Academic Press, San Diego, 1994.

\bibitem{ConnesRieffel87}
A.~{C}onnes and M.~A. {R}ieffel, \emph{Yang-mills for noncommutative two-tori},
  Contemporary Math \textbf{62} (1987), no.~Operator algebras and mathematical
  physics (Iowa City, Iowa, 1985), 237--266.

\bibitem{Gromov81}
M.~{G}romov, \emph{Groups of polynomial growth and expanding maps},
  Publications math{\'e}matiques de l' {I. H. E. S.} \textbf{53} (1981),
  53--78.

\bibitem{Gromov}
\bysame, \emph{Metric structures for {R}iemannian and non-{R}iemannian spaces},
  Progress in Mathematics, Birkh{\"a}user, 1999.

\bibitem{Hausdorff}
{F}. {H}ausdorff, \emph{{G}rundz{\"u}ge der {M}engenlehre}, Verlag Von Veit und
  Comp., 1914.

\bibitem{Kadison91}
{R}. {K}adison and {J}. {R}ingrose, \emph{Fundamentals of the operator algebras
  {III}}, AMS, 1991.

\bibitem{Kadison97}
\bysame, \emph{Fundamentals of the theory of operator algebras {I}}, Graduate
  Studies in Mathematics, vol.~15, AMS, 1997.

\bibitem{Kantorovich40}
{L}.~{V}. {K}antorovich, \emph{On one effective method of solving certain
  classes of extremal problems}, Dokl. Akad. Nauk. USSR \textbf{28} (1940),
  212--215.

\bibitem{Kantorovich58}
{L}.~{V}. {K}antorovich and {G}.~{Sh}. {R}ubinstein, \emph{On the space of
  completely additive functions}, Vestnik Leningrad Univ., Ser. Mat. Mekh. i
  Astron. \textbf{13} (1958), no.~7, 52--59, In Russian.

\bibitem{kerr02}
D.~{K}err, \emph{Matricial quantum {G}romov-{H}ausdorff distance}, J. Funct.
  Anal. \textbf{205} (2003), no.~1, 132--167, math.OA/0207282.

\bibitem{Latremoliere05}
{F}. {L}atr{\'e}moli{\`e}re, \emph{Approximation of the quantum tori by finite
  quantum tori for the quantum {G}romov-{H}ausdorff distance}, Journal of
  Funct. Anal. \textbf{223} (2005), 365--395, math.OA/0310214.

\bibitem{Latremoliere05b}
\bysame, \emph{Bounded-lipschitz distances on the state space of a
  {C*}-algebra}, Tawainese Journal of Mathematics \textbf{11} (2007), no.~2,
  447--469, math.OA/0510340.

\bibitem{Latremoliere12b}
\bysame, \emph{Quantum locally compact metric spaces}, Journal of Functional
  Analysis \textbf{264} (2013), no.~1, 362--402, ArXiv: 1208.2398.

\bibitem{Latremoliere13c}
\bysame, \emph{Convergence of fuzzy tori and quantum tori for the quantum
  {G}romov--{H}ausdorff {P}ropinquity: an explicit approach.}, M{\"u}nster
  Journal of Mathematics \textbf{8} (2015), no.~1, 57--98, ArXiv:
  math/1312.0069.

\bibitem{Latremoliere15c}
\bysame, \emph{Curved noncommutative tori as {L}eibniz compact quantum metric
  spaces}, Journal of Math. Phys. \textbf{56} (2015), no.~12, 123503, 16 pages,
  ArXiv: 1507.08771.

\bibitem{Latremoliere13b}
\bysame, \emph{The dual {G}romov--{H}ausdorff {P}ropinquity}, Journal de
  Math{\'e}matiques Pures et Appliqu{\'e}es \textbf{103} (2015), no.~2,
  303--351, ArXiv: 1311.0104.

\bibitem{Latremoliere15b}
\bysame, \emph{Quantum metric spaces and the {G}romov-{H}ausdorff propinquity},
  Accepted in Contemp. Math. (2015), 88 pages, ArXiv: 150604341.

\bibitem{Latremoliere15}
\bysame, \emph{A compactness theorem for the dual {G}romov-{H}ausdorff
  propinquity}, Accepted in Indiana University Journal of Mathematics (2016),
  40 Pages, ArXiv: 1501.06121.

\bibitem{Latremoliere16b}
\bysame, \emph{Equivalence of quantum metrics with a common domain}, Journal of
  Mathematical Analysis and Applications \textbf{443} (2016), 1179--1195,
  ArXiv: 1604.00755.

\bibitem{Latremoliere16c}
{F}. {L}atr{\'e}moli{\`e}re, \emph{The modular {G}romov--{H}ausdorff
  propinquity}, Submitted (2016), 67 pages, ArXiv: 1608.04881.

\bibitem{Latremoliere13}
{F}. {L}atr{\'e}moli{\`e}re, \emph{The {Q}uantum {G}romov-{H}ausdorff
  {P}ropinquity}, Trans. Amer. Math. Soc. \textbf{368} (2016), no.~1, 365--411,
  electronically published on May 22, 2015,
  http://dx.doi.org/10.1090/tran/6334, ArXiv: 1302.4058.

\bibitem{Latremoliere17a}
{F.} {L}atr{\'e}moli{\`e}re, \emph{Convergence of {H}eisenberg modules over
  quantum two-tori for the modular {G}romov-{H}ausdorff propinquity}, Submitted
  (2017), 70 pages.

\bibitem{Latremoliere14}
{F}. {L}atr{\'e}moli{\`e}re, \emph{The triangle inequality and the dual
  {G}romov-{H}ausdorff propinquity}, Indiana University Journal of Mathematics
  \textbf{66} (2017), no.~1, 297--313, ArXiv: 1404.6633.

\bibitem{Latremoliere16}
{F}. {L}atr{\'e}moli{\`e}re and {J}. {P}acker, \emph{Noncommutative solenoids
  and the gromov-hausdorff propinquity}, Accepted in Proc. AMS. (2016), 14
  pages, ArXiv: 1601.02707.

\bibitem{li03}
H.~{L}i, \emph{{$C^\ast$}-algebraic quantum {G}romov-{H}ausdorff distance},
  (2003), ArXiv: math.OA/0312003.

\bibitem{li05}
\bysame, \emph{{$\theta$}-deformations as compact quantum metric spaces}, Comm.
  Math. Phys. \textbf{1} (2005), 213--238, ArXiv: math/OA: 0311500.

\bibitem{Ozawa05}
{N}. {O}zawa and M.~A. {R}ieffel, \emph{Hyperbolic group {$C\sp\ast$}-algebras
  and free product {$C\sp\ast$}-algebras as compact quantum metric spaces},
  Canad. J. Math. \textbf{57} (2005), 1056--1079, ArXiv: math/0302310.

\bibitem{Rieffel74}
M.~A. {R}ieffel, \emph{Induced representations of {C*}-algebras}, Advances in
  Math. \textbf{13} (1974), 176--257.

\bibitem{Rieffel81}
\bysame, \emph{{C*}-algebras associated with irrational rotations}, Pacific
  Journal of Mathematics \textbf{93} (1981), 415--429.

\bibitem{Rieffel83}
\bysame, \emph{The cancellation theorem for the projective modules over
  irrational rotation {$C^\ast$}-algebras}, Proc. London Math. Soc. \textbf{47}
  (1983), 285--302.

\bibitem{Rieffel98a}
\bysame, \emph{Metrics on states from actions of compact groups}, Documenta
  Mathematica \textbf{3} (1998), 215--229, math.OA/9807084.

\bibitem{Rieffel99}
\bysame, \emph{Metrics on state spaces}, Documenta Math. \textbf{4} (1999),
  559--600, math.OA/9906151.

\bibitem{Rieffel02}
\bysame, \emph{Group {$C\sp\ast$}-algebras as compact quantum metric spaces},
  Documenta Mathematica \textbf{7} (2002), 605--651, ArXiv: math/0205195.

\bibitem{Rieffel01}
\bysame, \emph{Matrix algebras converge to the sphere for quantum
  {G}romov--{H}ausdorff distance}, Mem. Amer. Math. Soc. \textbf{168} (2004),
  no.~796, 67--91, math.OA/0108005.

\bibitem{Rieffel09}
\bysame, \emph{Distances between matrix alegbras that converge to coadjoint
  orbits}, Proc. Sympos. Pure Math. \textbf{81} (2010), 173--180, ArXiv:
  0910.1968.

\bibitem{Rieffel10c}
\bysame, \emph{{L}eibniz seminorms for "matrix algebras converge to the
  sphere"}, Clay Math. Proc. \textbf{11} (2010), 543--578, ArXiv: 0707.3229.

\bibitem{Rieffel11}
\bysame, \emph{Leibniz seminorms and best approximation from
  {$C^\ast$}-subalgebras}, Sci China Math \textbf{54} (2011), no.~11,
  2259--2274, ArXiv: 1008.3773.

\bibitem{Rieffel15}
\bysame, \emph{Matricial bridges for "matrix algebras converge to the sphere"},
  Submitted (2015), 31 pages, ArXiv: 1502.00329.

\bibitem{Rieffel00}
\bysame, \emph{{G}romov-{H}ausdorff distance for quantum metric spaces}, Mem.
  Amer. Math. Soc. \textbf{168} (March 2004), no.~796, math.OA/0011063.

\end{thebibliography}
\vfill

\end{document}